\documentclass[11pt, reqno]{amsart}

\usepackage{amsmath}

\usepackage{tabu}
\usepackage{tikz}
\usepackage[margin=1.25in]{geometry}
\usepackage[matrix,arrow,curve,frame]{xy}
 \usepackage{hyperref}
\usepackage{amsthm}
\usepackage{amsfonts}
\usepackage{amssymb}
\usepackage{wasysym}
\usepackage{mathrsfs}
\usepackage{mathtools}

\definecolor{my-linkcolor}{rgb}{0.75,0,0}
\definecolor{my-citecolor}{rgb}{0.1,0.57,0}
\definecolor{my-urlcolor}{rgb}{0,0,0.75}
\hypersetup{
	colorlinks,
	linkcolor={my-linkcolor},
	citecolor={my-citecolor},
	urlcolor={my-urlcolor}
}

\title[Non-semisimple Kazhdan-Lusztig category]{The non-semisimple Kazhdan-Lusztig category for affine $\mathfrak{sl}_2$ at admissible levels}
 \author{Robert McRae and Jinwei Yang}
\date{}

 \address{(R. M.) Yau Mathematical Sciences Center, Tsinghua University, Beijing 100084, China}
  \email{rhmcrae@tsinghua.edu.cn}
  
  \address{(J. Y.) School of Mathematical Sciences, Shanghai Jiao Tong University, Shanghai 200240, China}
  \email{jinwei2@sjtu.edu.cn}

 \subjclass{Primary 17B67, 17B69, 17B37, 18M15, 81R10, 81T40}

\newtheorem{thm}{Theorem}[section]
\newtheorem{cor}[thm]{Corollary}
\newtheorem{lem}[thm]{Lemma}
\newtheorem{prop}[thm]{Proposition}
\newtheorem{conj}[thm]{Conjecture}
\newtheorem{ques}[thm]{Question}
\theoremstyle{definition}\newtheorem{defi}[thm]{Definition}
\theoremstyle{definition}\newtheorem{rem}[thm]{Remark}
\theoremstyle{definition}
\theoremstyle{definition}

\newcommand{\cD}{\mathcal{D}}
\newcommand{\cE}{\mathcal{E}}

\newcommand{\cY}{\mathcal{Y}}

\newcommand{\cV}{\mathcal{V}}
\newcommand{\cA}{\mathcal{A}}
\newcommand{\cR}{\mathcal{R}}

\newcommand{\cS}{\mathcal{S}}

\newcommand{\cF}{\mathcal{F}}

\newcommand{\cJ}{\mathcal{J}}
\newcommand{\cL}{\mathcal{L}}
\newcommand{\cO}{\mathcal{O}}
\newcommand{\cP}{\mathcal{P}}
\newcommand{\cC}{\mathcal{C}}
\newcommand{\cG}{\mathcal{G}}
\newcommand{\cT}{\mathcal{T}}

\newcommand{\til}{\widetilde}

\newcommand{\fg}{\mathfrak{g}}
\newcommand{\fsl}{\mathfrak{sl}}

\newcommand{\CC}{\mathbb{C}}
\newcommand{\ZZ}{\mathbb{Z}}

\newcommand{\RR}{\mathbb{R}}
\newcommand{\QQ}{\mathbb{Q}}

\newcommand{\Id}{\mathrm{Id}}

\newcommand{\even}{\bar{0}}
\newcommand{\odd}{\bar{1}}
\newcommand{\tens}{\boxtimes}
\newcommand{\vac}{\mathbf{1}}
\newcommand{\ind}{\mathrm{Ind}}
 \DeclareMathOperator{\im}{Im}

 \let\ker\relax
 \let\hom\relax
 \DeclareMathOperator{\ker}{Ker}
 \DeclareMathOperator{\hom}{Hom}
 
 \newcommand{\ghat}{\widehat{\mathfrak{g}}}
 \newcommand{\slhat}{\widehat{\mathfrak{sl}}}

\begin{document}
\bibliographystyle{alpha}

\numberwithin{equation}{section}

 \begin{abstract}
We show that Kazhdan and Lusztig's category $KL^k(\mathfrak{sl}_2)$ of modules for the affine Lie algebra $\widehat{\mathfrak{sl}}_2$ at an admissible level $k$, equivalently the category of finite-length grading-restricted generalized modules for the universal affine vertex operator algebra $V^k(\mathfrak{sl}_2)$, is a braided tensor category. Although this tensor category is not rigid, we show that the subcategory of all rigid objects in $KL^k(\fsl_2)$ is equal to the subcategory of all projective objects, and that every simple module in $KL^k(\fsl_2)$ has a projective cover. Moreover, we show that the full subcategory of projective objects in $KL^k(\fsl_2)$ is monoidal equivalent to the category of tilting modules for quantum $\fsl_2$ at the root of unity $\zeta=e^{\pi i/(k+2)}$. Using this, we establish a universal property of the tensor category $KL^k(\fsl_2)$, and as an application, we prove a weak Kazhdan-Lusztig correspondence, that is, we obtain an exact essentially surjective (but not full or faithful) tensor functor from $KL^k(\fsl_2)$ to the category of finite-dimensional weight modules for the quantum group associated to $\fsl_2$ at the root of unity $\zeta $. We also use the universal property to classify the categories $KL^k(\fsl_2)$ up to (braided) tensor equivalence and to obtain a tensor-categorical version of quantum Drinfeld-Sokolov reduction, that is, we construct a braided tensor functor from $KL^k(\fsl_2)$ to a category of modules for the Virasoro algebra at central charge $1-\frac{6(k+1)^2}{k+2}$.
\end{abstract}

\maketitle

\tableofcontents

\section{Introduction}

In the celebrated series of papers \cite{KL1}-\cite{KL4}, Kazhdan and Lusztig constructed braided tensor categories of modules at fixed level for the affine Lie algebra $\ghat$ associated to a finite-dimensional simple Lie algebra $\fg$. Indeed, given $\fg$ and a level $k=-h^\vee+\kappa$ (where $h^\vee$ is the dual Coxeter number of $\fg$ and $\kappa$ is the \textit{shifted level}), they defined a category, which we denote $KL^k(\fg)$, whose objects are finite-length $\ghat$-modules of level $k$, all of whose composition factors are irreducible quotients of generalized Verma $\ghat$-modules induced from finite-dimensional irreducible $\fg$-modules (see \cite[Definition 2.15]{KL1}). For $\kappa\notin\mathbb{Q}_{\geq 0}$, Kazhdan and Lusztig showed that $KL^k(\fg)$ is naturally a braided tensor category. Moreover, they showed that $KL^k(\fg)$ is tensor equivalent to the category $\cC(\zeta,\fg)$ of finite-dimensional weight modules for the quantum group of $\fg$ at parameter $\zeta=e^{\pi i/r\kappa}$, where $r$ is the lacing number of $\fg$. This tensor equivalence is commonly called the \textit{Kazhdan-Lusztig correspondence}. 

It is still unclear whether there is a braided tensor category structure on $KL^k(\fg)$ for $\kappa \in \mathbb{Q}_{>0}$, or whether there is a relation with $\cC(\zeta,\fg)$. A natural approach to these problems is the vertex algebraic tensor category theory developed by Huang-Lepowsky-Zhang \cite{HLZ1}-\cite{HLZ8}, because $KL^k(\fg)$ is a category of modules for the universal affine vertex operator algebra $V^k(\fg)$ at level $k$ (see for example \cite[Theorem 6.2.23]{LL}). Vertex operator algebras are algebraic structures which were first introduced by Borcherds \cite{Bo} and Frenkel-Lepowsky-Meurman \cite{FLM} in the context of the monstrous moonshine problem, but they also provide a mathematically rigorous approach to two-dimensional conformal quantum field theories in physics. The tensor product of modules for a vertex operator algebra $V$ is crucially \textit{not} based on the vector space tensor product of $V$-modules, but is rather the ``fusion product'' of conformal field theory. Mathematically, this fusion product is defined by a universal property in terms of intertwining operators \cite{FHL, HLZ2, HLZ3}, which are building blocks of correlation functions in the conformal field theory associated to $V$. Under certain conditions, the fusion tensor product of $V$-modules gives suitable categories of $V$-modules the structure of braided tensor categories \cite{HLZ8}.

For a level $k=-h^\vee+\kappa$, $KL^k(\fg)$ is precisely the category of finite-length grading-restricted generalized modules for the universal affine vertex operator algebra $V^k(\fg)$. When $\kappa\notin\QQ_{\geq 0}$, $V^k(\fg)$ is simple (see \cite[Proposition 2.12]{KL1}), but this is not usually the case when $\kappa\in\QQ_{>0}$. Thus for $\kappa\in\QQ_{>0}$, most previous work has concentrated on braided tensor category structure for the smaller category $KL_k(\fg)$ of finite-length grading-restricted generalized modules for the simple quotient vertex operator algebra $L_k(\fg)$; see especially the recent progress in \cite{CHY,CY}. In this case, $KL_k(\fg)$ is usually a small subcategory of $KL^k(\fg)$, and little is known in general about tensor structure on the larger category $KL^k(\fg)$ of $V^k(\fg)$-modules.

\subsection{Main results on \texorpdfstring{$KL^k(\fsl_2)$}{KLk(sl2)}}

In this work, we take $\fg=\fsl_2$ (so $h^\vee=2$) and show that the Kazhdan-Lusztig category $KL^k(\fsl_2)$ of finite-length grading-restricted generalized $V^k(\fsl_2)$-modules is naturally a braided tensor category for all $k=-2+\kappa$ with $\kappa \in\QQ_{>0}$. We mainly focus on the case that $k$ is a Kac-Wakimoto \textit{admissible} level \cite{KW}, that is, $\kappa=p/q$ for relatively prime $p\in\ZZ_{\geq 2}$ and $q\in\ZZ_{\geq 1}$, since the case $p=1$ has already been handled in \cite{CY}. For admissible levels $k=-2+p/q$, we determine the structure and properties of $KL^k(\fsl_2)$ in considerable detail, and as a consequence, we obtain a weak Kazhdan-Lusztig correspondence, that is, an exact and essentially surjective (but not fully faithful) tensor functor from $KL^k(\fsl_2)$ to the quantum group category $\cC(\zeta,\fsl_2)$ at $\zeta=e^{\pi i q/p}$.

For $\fsl_2$, admissible levels are interesting because $V^k(\fsl_2)$ is non-simple if and only if $k$ is admissible. Thus the smaller category $KL_k(\fsl_2)$ of modules for the simple quotient vertex operator algebra $L_k(\fsl_2)$ is a proper subcategory of $KL^k(\fsl_2)$ if and only if $k$ is admissible. On the one hand, $KL^k(\fsl_2)$ at admissible level $k$ is not semisimple and has infinitely many simple objects $\cL_r$ for $r \in \mathbb{Z}_{\geq 1}$; each $\cL_r$ is the simple quotient of the generalized Verma module (also called Weyl module) $\cV_r$ induced from the $r$-dimensional simple $\fsl_2$-module. On the other hand, $KL_k(\fsl_2)$ is semisimple with finitely many simple objects $\cL_r$ for $1\leq r\leq p-1$ \cite{AM}; moreover, $KL_k(\fsl_2)$ is a rigid braided tensor category \cite{CHY}.

To prove that $KL^k(\fsl_2)$ for $k=-2+p/q$ is also a braided tensor category, we use \cite[Theorem~3.3.4]{CY}. The key point is to show that any finitely-generated grading-restricted generalized $V^k(\fsl_2)$-module also has finite length and thus is an object of $KL^k(\fsl_2)$. For this, it is enough to show that the generalized Verma modules $\cV_r$ have finite length, and in fact, it follows straightforwardly from Malikov's structural results on Verma modules for rank-$2$ Kac-Moody Lie algebras \cite{Ma} that $\cV_r$ is simple if $p\mid r$ and has length $2$ if $p\nmid r$ (see Theorem \ref{thm:gen_Verma_structure}). Thus we prove:
\begin{thm}[Theorem \ref{thm:existencebtc}, Corollary \ref{cor:KL_k_tens_ideal}, Theorem \ref{thm:inclusion_is_lax_monoidal}]
Let $k=-2+\kappa$ for $\kappa\in\QQ_{>0}$. Then $KL^k(\fsl_2)$ admits the braided tensor category structure of \cite{HLZ8}, and the embedding $KL_k(\fsl_2)\hookrightarrow KL^k(\fsl_2)$ is a lax monoidal functor. Moreover, $KL_k(\fsl_2)$ is both a tensor ideal and a tensor subcategory of $KL^k(\fsl_2)$ (with a different unit object if $k$ is admissible).
\end{thm}

Unlike $KL_k(\fsl_2)$, the larger category $KL^k(\fsl_2)$ is not rigid for $k$ admissible, that is, not every object has a dual in the sense of tensor categories. For example, objects of $KL_k(\fsl_2)$ are not rigid when considered as objects of $KL^k(\fsl_2)$ because the unit object $\cL_1=L_k(\fsl_2)$ of the tensor category $KL_k(\fsl_2)$ is not the same as the unit object $\cV_1=V^k(\fsl_2)$ of $KL^k(\fsl_2)$.
However, by \cite[Theorem 2.12]{ALSW}, vertex algebraic contragredient modules \cite{FHL} give $KL^k(\fsl_2)$ a weaker duality structure, that of a ribbon Grothendieck-Verdier category \cite{BD} whose dualizing object is the contragredient of $\cV_1$. We will use this Grothendieck-Verdier category structure on $KL^k(\fsl_2)$ to prove that the weak Kazhdan-Lusztig correspondence mentioned above is an exact functor.

Although $KL^k(\fsl_2)$ as a whole is not rigid, determining which particular objects are rigid turns out to be critical for exploring the detailed tensor structure of $KL^k(\fsl_2)$. The first non-trivial rigid object we obtain in $KL^k(\fsl_2)$ is the generalized Verma module $\cV_2$ induced from the standard $2$-dimensional simple $\fsl_2$-module, which we prove is self-dual by using Knizhnik-Zamolodchikov (KZ) equations \cite{KZ} to derive explicit expressions for compositions of intertwining operators involving $\cV_2$; similar methods for proving rigidity have been used in many recent works, including \cite{TW, CMY-singlet, CMY3, MY2, MY3, MS}. The basic properties of $\cV_2$ are summarized in the following theorem, where we use $\tens$ to denote the tensor product operation on $KL^k(\fsl_2)$:
\begin{thm}[Theorem \ref{thm:V12_times_Vrs}, Theorem \ref{thm:V12_rigid}, Theorem \ref{thm:V12_times_Vrp}]\label{thm:intro_V2}
Let $k=-2+p/q$ for relatively prime $p\in\ZZ_{\geq 2}$ and $q\in\ZZ_{\geq 1}$, and let $\zeta=e^{\pi iq/p}$.
\begin{enumerate}
\item The generalized Verma module $\cV_2$ is rigid and self-dual in $KL^k(\fsl_2)$.
\item The self-dual module $\cV_2$ has intrinsic dimension $-\zeta-\zeta^{-1}$, that is, if $e_{\cV_2}: \cV_2\tens\cV_2\rightarrow\cV_1$ and $i_{\cV_2}: \cV_1\rightarrow\cV_2\tens\cV_2$ denote evaluation and coevaluation maps for $\cV_2$, then
\begin{equation*}
e_{\cV_2}\circ i_{\cV_2}=(-\zeta-\zeta^{-1})\cdot\Id_{\cV_1}.
\end{equation*}

\item For $r\in\ZZ_{\geq 2}$, there is a short exact sequence
\begin{equation}\label{eqn:intro_exact_seq}
0\longrightarrow \cV_{r-1} \longrightarrow\cV_2\tens\cV_r\longrightarrow\cV_{r+1}\longrightarrow 0
\end{equation} 
which splits if and only if $p\nmid r$.
\end{enumerate}
\end{thm}

Next, we use rigidity of $\cV_2$ and \eqref{eqn:intro_exact_seq} to deduce that $\cV_{r}$ is rigid and self dual for $1 \leq r \leq p$. 
Then $\cV_2\tens\cV_p$ is a rigid indecomposable module which turns out to be projective in $KL^k(\fsl_2)$. In fact, we construct all indecomposable projective objects as follows. First recall that an abelian category has \textit{enough projectives} if every object is a homomorphic image of a projective object. Since every object of $KL^k(\fsl_2)$ has finite length, $KL^k(\fsl_2)$ has enough projectives if and only if every simple module $\cL_r$ has a \textit{projective cover}, which is an indecomposable projective module that surjects onto $\cL_r$. It is easy to prove that the generalized Verma module $\cV_1=V^k(\fsl_2)$ is projective in $KL^k(\fsl_2)$, and in any tensor category with projective unit object, all rigid objects are projective (Corollary \ref{cor:rigid_is_projective}). Therefore, the generalized Verma module $\cV_{r}$ is projective and a projective cover of $\cL_r$ for $1 \leq r \leq p$. In general, we obtain a projective cover for any $\cL_r$ as a direct summand of $\cV_2^{\tens (r-1)}$. Such direct summands of tensor powers of the rigid module $\cV_2$ are also rigid, and it follows that all projective objects of $KL^k(\fsl_2)$ are also rigid. We summarize the main properties of rigid and projective objects in $KL^k(\fsl_2)$ in the following theorem:
\begin{thm}[Theorem \ref{thm:V12_times_Vrp},   Theorem \ref{thm:Prs}, Theorem \ref{thm:Prs_properties}, Corollary \ref{cor:projective_is_rigid}, Proposition \ref{prop:Pr_self_contra}, Theorem \ref{thm:Pr_log}]\label{thm:intro_Pr}
Let $k=-2+p/q$ for relatively prime $p\in\ZZ_{\geq 2}$ and $q\in\ZZ_{\geq 1}$. Then an object of $KL^k(\fsl_2)$ is rigid if and only if it is projective. Moreover, $\cL_r$ for all $r\in\ZZ_{\geq 1}$ has a projective cover $\cP_r$ in $KL^k(\fsl_2)$ as follows:
\begin{enumerate}

\item For $1\leq r\leq p-1$ and for $p\mid r$, $\cP_r=\cV_r$. In particular, $\cP_r$ is simple if $p\mid r$.

\item For $n\in\ZZ_{\geq 1}$ and $1\leq r\leq p-1$, there is a non-split short exact sequence
\begin{equation*}
0\longrightarrow \cV_{np-r} \longrightarrow \cP_{np+r}\longrightarrow \cV_{np+r}\longrightarrow 0,
\end{equation*}
and $\cP_{np+r}$ has Loewy diagram
\begin{equation*}
\begin{matrix}
  \begin{tikzpicture}[->,>=latex,scale=1.5]
\node (b1) at (1,0) {$\cL_{np+r}$};
\node (c1) at (-1, 1){$\cP_{np+r}:$};
   \node (a1) at (0,1) {$\cL_{np-r}$};
   \node (b2) at (2,1) {$\cL_{(n+2)p-r}$};
    \node (a2) at (1,2) {$\cL_{np+r}$};
\draw[] (b1) -- node[left] {} (a1);
   \draw[] (b1) -- node[left] {} (b2);
    \draw[] (a1) -- node[left] {} (a2);
    \draw[] (b2) -- node[left] {} (a2);
\end{tikzpicture}
\end{matrix} .
\end{equation*}
Also, $\cP_{np+r}$ is self-contragredient and logarithmic, that is, the Virasoro operator $L(0)$ acts non-semisimply on $\cP_{np+r}$.
\end{enumerate}
\end{thm}

Logarithmic modules for vertex operator algebras are so-called because of their role in logarithmic conformal field theory in physics: non-semisimple actions of $L(0)$ lead to logarithmic singularities in correlation functions. It is worth noting that it is not very easy to construct logarithmic modules for affine Lie algebras directly. Probably the only previously known logarithmic $\widehat{\fsl}_2$-modules at admissible levels are those constructed by Adamovi\'{c} in \cite{Ad-inverse-QH-red} (these modules are not objects of $KL^k(\fsl_2)$). Further logarithmic $\widehat{\fsl}_2$-modules at admissible levels were conjectured in \cite{Ra} but not actually constructed; our Theorem \ref{thm:Pr_log} proves that some of these conjectured modules indeed exist. Our modules $\cP_r$ seem to be the first known logarithmic $\widehat{\fsl}_2$-modules with finite-dimensional generalized $L(0)$-eigenspaces.

As part of the construction of the logarithmic modules $\cP_{np+r}$ and the proof of Theorem \ref{thm:intro_Pr}, we also compute the tensor product of $\cV_2$ with each indecomposable projective module:
\begin{thm}[Theorem \ref{thm:V12_times_Vrp}, Theorem \ref{thm:V12_times_Pr1_p=2}, Theorem \ref{thm:Prs_properties}]\label{thm:intro_V2_times_Pr}
Let $k=-2+p/q$ for relatively prime $p\in\ZZ_{\geq 2}$ and $q\in\ZZ_{\geq 1}$. Then using the convention $\cP_r=0$ if $r\leq 0$:
\begin{enumerate}
\item If $p=2$, then for $n\geq 0$ and $r=1,2$,
\begin{equation*}
\cV_2\boxtimes\cP_{2n+r} =\left\lbrace\begin{array}{lll}
\cP_{2(n-1)}\oplus 2\cdot\cP_{2n}\oplus\cP_{2(n+1)} & \text{if} & r=1\\
\cP_{2(n+1)+1} & \textit{if} & r=2\\
\end{array}\right. .
\end{equation*}

\item If $p\geq 3$, then for $n\geq 0$ and $1\leq r\leq p$,
\begin{equation*}
\cV_{2}\tens\cP_{np+r}\cong\left\lbrace\begin{array}{lll}
2\cdot\cP_{np}\oplus\cP_{np+2} & \text{if} & r=1\\
\cP_{np+r-1}\oplus\cP_{np+r+1} & \text{if} & 2\leq r\leq p-2\\
\cP_{(n-1)p}\oplus\cP_{(n+1)p-2}\oplus\cP_{(n+1)p} & \text{if} & r=p-1\\
\cP_{(n+1)p+1} & \text{if} & r=p
\end{array}\right. .
\end{equation*}
\end{enumerate}
\end{thm}

Our final result on the tensor category $KL^k(\fsl_2)$ itself is a classification of its braidings. We show in Theorem \ref{thm:KLk_braidings} that $KL^k(\fsl_2)$ admits four natural braiding isomorphisms $\cR$ which satisfy the hexagon axiom for braided tensor categories, and each is completely determined by the automorphism $\cR_{\cV_2,\cV_2}$ of $\cV_2\tens\cV_2$. The first braiding is the official one specified in \cite{HLZ8}, and a second is the reverse braiding defined by $\cR_{W_1,W_2}^{\mathrm{rev}}=\cR_{W_2,W_1}^{-1}$ for modules $W_1$, $W_2$ in $KL^k(\fsl_2)$. The remaining two braidings are obtained from the first two by changing $\cR_{W_1,W_2}$ by a sign if the Cartan generator $h\in\fsl_2$ acts on both $W_1$ and $W_2$ by odd-integer eigenvalues (in particular $\cR_{\cV_2,\cV_2}$ changes to $-\cR_{\cV_2,\cV_2}$). We also show that for each of the four braidings, $KL^k(\fsl_2)$ admits two ribbon twists $\theta$ which satisfy the balancing equation
\begin{equation*}
\theta_{W_1\tens W_2} =\cR_{W_2,W_1}\circ\cR_{W_1,W_2}\circ(\theta_{W_1}\tens\theta_{W_2})
\end{equation*}
for objects $W_1$, $W_2$ of $KL^k(\fsl_2)$. When $\cR$ is the official braiding specified in \cite{HLZ8}, one of these twists is the official one given by $\theta=e^{2\pi i L(0)}$.

\subsection{Universal property of \texorpdfstring{$KL^k(\fsl_2)$}{KLk(sl2)} and a weak Kazhdan-Lusztig correspondence}

Let $k=-2+p/q$ be an admissible level for $\fsl_2$, and let $\zeta =e^{\pi i q/p}$. It is obvious that $KL^k(\fsl_2)$ is not tensor equivalent to the quantum group category $\cC(\zeta,\fsl_2)$ since $\cC(\zeta,\fsl_2)$ is rigid while $KL^k(\fsl_2)$ is not. However, we show that there is an exact and essentially surjective tensor functor $\cF: KL^k(\fsl_2)\rightarrow\cC(\zeta,\fsl_2)$ which we call a weak Kazhdan-Lusztig correspondence. It turns out that the tensor ideal $KL_k(\fsl_2)$ of modules for the simple affine vertex operator algebra $L_k(\fsl_2)$ is what obstructs $\cF$ from being full or faithful, so one could say that as tensor categories, $KL^k(\fsl_2)$ is something like an extension of $\cC(\zeta,\fsl_2)$ by $KL_k(\fsl_2)$.

To obtain the weak Kazhdan-Lusztig correspondence, we show that $KL^k(\fsl_2)$ satisfies a universal property, analogous to Ostrik's universal property of $\cC(\zeta,\fsl_2)$ from \cite[Remark 2.10]{Os} (see also \cite[Theorem 5.3]{GN}). Such universal properties derive from a universal property of the category of tilting modules for quantum $\fsl_2$. Indeed, let $\cT_\zeta\subseteq\cC(\zeta,\fsl_2)$ denote the subcategory of tilting modules introduced in \cite{An}; one can show that $\cT_\zeta$ is the smallest subcategory of $\cC(\zeta,\fsl_2)$ which contains the irreducible two-dimensional standard $U_\zeta(\fsl_2)$-module $\mathbf{X}$ and is closed under tensor products, finite direct sums, and direct summands. The tilting module $\mathbf{X}$ has intrinsic dimension $-\zeta-\zeta^{-1}$, like the generalized Verma module $\cV_2$ in $KL^k(\fsl_2)$ (recall Theorem \ref{thm:intro_V2}(2)). Then Ostrik showed in \cite[Theorem 2.4]{Os} that a self-dual object $X$ of intrinsic dimension $-\zeta-\zeta^{-1}$ in any tensor category $\cC$ induces a unique tensor functor $\cT_\zeta\rightarrow\cC$ sending $\mathbf{X}$ to $X$. Thus taking $\cC=KL^k(\fsl_2)$ and $X=\cV_2$ yields a tensor functor $\cT_\zeta\rightarrow KL^k(\fsl_2)$.

Let $\cP^k$ be the subcategory of all projective (equivalently, all rigid) objects in $KL^k(\fsl_2)$. Then $\cP^k$ contains $\cV_2$ and is closed under tensor products, finite direct sums, and direct summands, so the image of the tensor functor $\cT_\zeta\rightarrow KL^k(\fsl_2)$ is contained in $\cP^k$. In fact, we prove in Theorem \ref{thm:Tzeta_Pk_equiv} that this functor is a tensor equivalence between $\cT_\zeta$ and $\cP^k$. In other words, we can view $KL^k(\fsl_2)$ as an ``abelianization'' of $\cT_\zeta$ into which $\cT_\zeta$ embeds as the full subcategory of projective objects. This together with the universal property of $\cT_\zeta$ yields the universal property of $KL^k(\fsl_2)$:
\begin{thm}[Theorem \ref{thm:KLk_univ_prop}]\label{thm:intro_univ_prop}
Let $k=-2+p/q$ for relatively prime $p\in\ZZ_{\geq 2}$ and $q\in\ZZ_{\geq 1}$, let $\cC$ be a (not necessarily rigid) tensor category with right exact tensor product $\tens_\cC$, and let $X$ be a rigid self-dual object of $\cC$ with evaluation $e_X: X\tens_\cC X\rightarrow\vac_\cC$ and coevaluation $i_X: \vac_\cC\rightarrow X\tens_\cC X$ such that
\begin{equation*}
e_X\circ i_X = -(e^{\pi i q/p}+e^{-\pi i q/p})\cdot\Id_{\vac_\cC}.
\end{equation*}
Then there is a unique up to natural isomorphism right exact tensor functor $\cF: KL^k(\mathfrak{sl}_2)\rightarrow\cC$, equipped with isomorphism $\varphi:\cF(\cV_{1})\rightarrow\vac_\cC$ and natural isomorphism
\begin{equation*}
F: \tens_\cC\circ(\cF\times \cF)\longrightarrow\cF\circ\tens,
\end{equation*}
 such that $\cF(\cV_{2})=X$ and
 \begin{equation*}
\varphi\circ \cF(e_{\cV_{2}})\circ F_{\cV_{2},\cV_{2}} = e_X, \qquad F_{\cV_{2},\cV_{2}}^{-1}\circ\cF(i_{\cV_{2}})\circ\varphi^{-1}= i_X.
 \end{equation*}
\end{thm}

Taking $\cC=\cC(\zeta,\fsl_2)$ and $X=\mathbf{X}$ in this theorem yields the weak Kazhdan-Lusztig correspondence $\cF: KL^k(\fsl_2)\rightarrow\cC(\zeta,\fsl_2)$. With more work, we can prove a few more properties of $\cF$; in particular, we show that $\cF$ is exact by using the Grothendieck-Verdier category structure \cite{BD, ALSW} on $KL^k(\fsl_2)$ to turn right exact into left exact sequences:
\begin{thm}[Theorem \ref{thm:weak_KL_correspondence}, Lemma \ref{lem:weak_KL_lack_of_faithful_objects}, Proposition \ref{prop:weak_KL_lack_of_faithful}]\label{thm:intro_weak_KL}
Let $k=-2+p/q$ for $p\in\ZZ_{\geq 2}$ and $q\in\ZZ_{\geq 1}$, and let $\zeta=e^{\pi iq/p}$. Then there is a unique exact and essentially surjective tensor functor $\cF: KL^k(\fsl_2)\rightarrow\cC(\zeta,\fsl_2)$ extending the equivalence $\cP^k\xrightarrow{\sim}\cT_\zeta$. Moreover, for $W$ an object in $KL^k(\fsl_2)$,   $\cF(W)=0$ if and only if $W$ is an object of $KL_k(\fsl_2)$, and for $f$ a morphism in $KL^k(\fsl_2)$, $\cF(f)=0$ if and only if $\mathrm{Im}\,f$ is an object of $KL_k(\fsl_2)$.
\end{thm}

Thus the subcategory $KL_k(\fsl_2)$ prevents the weak Kazhdan-Lusztig correspondence $\cF$ from being faithful; for similar reasons, $\cF$ is also not full. However, it was pointed out to us by Cris Negron that our embedding of $\cT_\zeta$ into $KL^k(\fsl_2)$ as the full subcategory of projective objects yields a \textit{derived} Kazhdan-Lusztig correspondence, that is, a tensor equivalence between suitable derived categories. In particular, for $\cC$ an abelian category, let $\ind\,\cC$ denote its Ind-category, $D^b(\cC)$ denote its bounded derived category, and $D(\cC)$ denote the unbounded derived category of $\ind\,\cC$. Then using Theorem \ref{thm:Tzeta_Pk_equiv}:
\begin{thm}[Theorem \ref{thm:derived_KL_corr}]\label{thm:intro_derived_KL_corr}
The tensor equivalence $\cT_{\zeta} \xrightarrow{\sim} \cP^k$ induces a tensor equivalence $\ind\,D^b(\cC(\zeta, \fsl_2)) \xrightarrow{\sim} D(KL^k(\fsl_2))$.
\end{thm}

There is a similar derived Kazhdan-Lusztig correspondence for general simple Lie algebras $\mathfrak{g}$ at positive rational shifted levels $\kappa$ stated in \cite[Sections 9.4 and 9.5]{Ga}, where it is obtained through a duality between suitable derived versions of $KL^{\kappa+h^\vee}(\mathfrak{g})$ and $KL^{-\kappa+h^\vee}(\mathfrak{g})$ together with the usual Kazhdan-Lusztig correspondence at negative rational shifted levels. The precise relationship between Theorem \ref{thm:intro_derived_KL_corr} and the $\fsl_2$ case of \cite{Ga} is not entirely clear to us; for example, it does not seem to be explicitly stated in \cite[Section 9.4]{Ga} whether the derived Kazhdan-Lusztig correspondence at positive rational shifted levels considered there preserves monoidal structures. But in any case, the equivalences of derived categories in Theorem \ref{thm:intro_derived_KL_corr} or in \cite{Ga} by themselves are not enough to yield the precise relationship between the abelian tensor categories $KL^k(\fsl_2)$ and $\cC(\zeta,\fsl_2)$ specified in Theorem \ref{thm:intro_weak_KL}.

It is worth mentioning that Theorem \ref{thm:intro_weak_KL} is similar to
a conjectural weak Kazhdan-Lusztig correspondence between the triplet vertex operator algebra $W_{p,q}$ for relatively prime $p,q\in\ZZ_{\geq 2}$ and a quantum group $\mathfrak{g}_{p,q}$ constructed as a quotient of the tensor product of restricted quantum groups of $\fsl_2$ at roots of unity $e^{\pi iq/p}$ and $e^{\pi ip/q}$ \cite{FGST1, FGST2}. The triplet algebra $W_{p,q}$ is a non-simple extension of the universal Virasoro vertex operator algebra at central charge $1-\frac{6(p-q)^2}{pq}$, and similar to $KL^k(\fsl_2)$, its module category $\mathrm{Rep}\,W_{p,q}$ is not rigid and thus cannot be tensor equivalent to $\mathrm{Rep}\,\fg_{p,q}$. However, there appear to be correspondences between fusion rules in $\mathrm{Rep}\,W_{p,q}$ and $\mathrm{Rep}\,\fg_{p,q}$ (see in particular the discussions in \cite[Section 6.3]{FGST1}, \cite[Section 2.4]{GRW} and \cite[Remark 5.9]{Na}), so it is natural to expect an essentially surjective but not fully faithful tensor functor $\mathrm{Rep}\,W_{p,q}\rightarrow\mathrm{Rep}\,\fg_{p,q}$ similar to that in Theorem \ref{thm:intro_weak_KL}. Possibly the methods used to prove Theorem \ref{thm:intro_weak_KL} might contribute to the construction of such a tensor functor.

\subsection{Further applications and open problems}

The universal property of $KL^k(\fsl_2)$ has several applications
besides the weak Kazhdan-Lusztig correspondence. We can also classify the tensor categories $KL^k(\fsl_2)$ for admissible $k$ up to (braided) tensor equivalence:
\begin{thm}[Theorem \ref{thm:KLk_tens_equiv}, Theorem \ref{thm:KLk_class_braided}]\label{thm:intro_classify}
If $k$ is any  admissible level for $\fsl_2$, then $KL^k(\fsl_2)$ is tensor equivalent to $KL^{-2+p/q}(\fsl_2)$ for unique relatively prime $p\in\ZZ_{\geq 2}$ and $q\in\lbrace 1,2,\ldots, p-1\rbrace$. Specifically, for such $p$ and $q$,
\begin{equation*}
KL^{-2+p/q}(\fsl_2) \cong KL^{-2+p/(\pm q+2np)}(\fsl_2)
\end{equation*}
as tensor categories for all $n\in\ZZ_{\geq 1}$. Moreover, $KL^{-2+p/(\pm q+2np)}(\fsl_2)$ equipped with the standard braiding specified in \cite{HLZ8} is braided tensor equivalent to $KL^{-2+p/q}(\fsl_2)$ equipped with one of the four braidings from Theorem \ref{thm:KLk_braidings} that depends on $n$ and the choice of $\pm$ (see Theorem \ref{thm:KLk_class_braided} for details).
\end{thm}

Although the tensor categories $KL^{-2+p/q}(\fsl_2)$ for $1\leq q\leq p-1$ are not tensor equivalent, some are related by $3$-cocycle twists (see for example \cite[Section 1]{KWe}). The $3$-cocycle twist $KL^k(\fsl_2)^\tau$ agrees with $KL^k(\fsl_2)$ as a category and has the same tensor product, but the associativity isomorphism $\cA_{W_1,W_2,W_3}: W_1\tens(W_2\tens W_3)\rightarrow(W_1\tens W_2)\tens W_3$ is changed by a sign if the Cartan generator $h\in\fsl_2$ acts on all three of the modules $W_1$, $W_2$, $W_3$ by odd-integer eigenvalues. In Theorem \ref{thm:KLk_tau_equiv} and Corollary \ref{cor:class_KLk_tau_braid}, we show that for $1\leq q\leq p-1$,
\begin{equation*}
KL^{-2+p/q}(\fsl_2)^\tau\cong KL^{-2+p/(p-q)}(\fsl_2)
\end{equation*}
as tensor categories. In particular, taking $p=2$, $KL^0(\fsl_2)$ is equivalent to its $3$-cocycle twist, while for $p\geq 3$, $KL^{-2+p/q}(\fsl_2)$ for any $q$ relatively prime to $p$ is tensor equivalent to either such a category with $1\leq q<\frac{p}{2}$ or its $3$-cocycle twist.

We can also use the universal property of Theorem \ref{thm:intro_univ_prop} to relate $KL^k(\fsl_2)$ to the Virasoro algebra at central charge $c_{p,q}=1-\frac{6(p-q)^2}{pq}$. For relatively prime $p,q\in\ZZ_{\geq 1}$, let $\cO_{c_{p,q}}$ denote the category of $C_1$-cofinite modules for the universal Virasoro vertex operator algebra at central charge $c_{p,q}$; it was shown in \cite{CJORY} that $\cO_{c_{p,q}}$ is a locally finite braided tensor category. For $q=1$, the detailed tensor structure of $\cO_{c_{p,1}}$ was determined in \cite{MY2}, and it was shown in \cite{GN} that $\cO_{c_{p,1}}$ is tensor equivalent to the quantum group category $\cC(\zeta,\fsl_2)$ for $\zeta=e^{\pi i/p}$. Thus for $q=1$, the weak Kazhdan-Lusztig correspondence yields an exact and essentially surjective tensor functor $\cF: KL^{-2+p}(\fsl_2)\rightarrow\cO_{c_{p,1}}$.

For $p,q\geq 2$, some of the detailed structure of $\cO_{c_{p,q}}$ was determined in \cite{MS}, and using these results together with Theorem \ref{thm:intro_univ_prop}, we show in Theorem \ref{thm:tens_cat_qDS_red} that there are two right exact braided tensor functors $\cF_{p,q}: KL^{-2+p/q}(\fsl_2) \rightarrow \cO_{c_{p,q}}$ and $\cF_{q,p}: KL^{-2+q/p}(\fsl_2) \rightarrow \cO_{c_{p,q}}$. There are in fact exact functors between these categories given by quantum Drinfeld-Sokolov reduction \cite{FKW, Ar}, but it is not known in these cases whether quantum Drinfeld-Sokolov reduction is a tensor functor. Thus we conjecture:
\begin{conj}[Conjecture \ref{conj:tens_cat_qDS_red}]\label{conj:intro}
For relatively prime $p,q\geq 2$, the braided tensor functors $\cF_{p,q}$ and $\cF_{q,p}$ are exact and are naturally isomorphic to the restrictions of quantum Drinfeld-Sokolov reduction to $KL^{-2+p/q}(\fsl_2)$ and $KL^{-2+q/p}(\fsl_2)$. 
\end{conj}

Another conjecture was communicated to us by Azat Gainutdinov. Recall that before Theorem \ref{thm:intro_univ_prop}, we remarked that the embedding of the tilting module category $\cT_\zeta$ into $KL^k(\fsl_2)$ as the subcategory of projective objects shows that $KL^k(\fsl_2)$ is an abelianization of $\cT_\zeta$. The category $\cT_\zeta$ contains the Temperley-Lieb category $TL(-\zeta-\zeta^{-1})$, that is, the monoidal subcategory generated by the two-dimensional simple tilting module $\mathbf{X}$, so one could also view $KL^k(\fsl_2)$ as an abelianization of $TL(-\zeta-\zeta^{-1})$. In fact, an abelianization of the Temperley-Lieb category has already been constructed by Gainutdinov and Saleur \cite{GS}. This category is an $N\to\infty$ limit of module categories for the finite-dimensional Temperley-Lieb algebras on $N$ strands, and it has two natural tensor products. Gainutdinov has conjectured that $KL^k(\fsl_2)$ is tensor equivalent to the category in \cite{GS}, for one of its two tensor products. (Note that $KL^k(\fsl_2)$ also has two tensor products, the usual vertex algebraic tensor product $W_1\tens W_2$, and a second using contragredient modules, $(W_2'\tens W_1')'$; see page 2 of \cite{ALSW}. These two tensor products are not equivalent because $KL^k(\fsl_2)$ is not rigid.) As evidence for Gainutdinov's conjecture, projective modules in $KL^k(\fsl_2)$ (recall Theorem \ref{thm:intro_Pr}) have the same structure as projective modules in Gainutdinov-Saleur's category (see \cite[Section 5.7.1]{GS}), indicating that $KL^k(\fsl_2)$ and Gainutdinov-Saleur's category are at least equivalent as abelian categories.

Finally, it is natural to ask whether the results of this paper generalize to higher-rank affine Lie algebras:
\begin{ques}
Let $\fg$ be a finite-dimensional simple Lie algebra and let $k=-h^\vee+\kappa$ for $\kappa\in\QQ_{>0}$.
\begin{enumerate}
\item Does $KL^k(\fg)$ admit the vertex algebraic braided tensor category structure of \cite{HLZ8}? If so, is $KL^k(\fg)$ rigid, and what is the relation between rigid and projective objects of $KL^k(\fg)$?

\item Let $\zeta=e^{\pi i/r\kappa}$ where $r$ is the lacing number of $\fg$. If $KL^k(\fg)$ is a tensor category, is there a natural tensor functor $\cF: KL^k(\fg)\rightarrow\cC(\zeta,\fg)$, and if so, what further properties (such as exactness and essential surjectivity) might this functor have?
\end{enumerate}
\end{ques}

For the first question, the main obstacle is whether generalized Verma modules for $\widehat{\fg}$ have finite length. If they do, then $KL^k(\fg)$ is a (locally finite) braided tensor category by \cite[Theorem 3.3.4]{CY}. If they do not, then $KL^k(\fg)$ is probably not a tensor category, but one could replace $KL^k(\fg)$ with the category of finitely-generated grading-restricted $V^k(\fg)$-modules, which might possibly still be a (finitely cocomplete) braided monoidal category. One could also try replacing $KL^k(\fg)$ with the category $KL_k(\fg)$ of finite-length modules for the simple quotient vertex operator algebra $L_k(\fg)$. In any case, once one shows that $KL^k(\fg)$ (or one of its replacements) is a braided monoidal category, then Proposition \ref{prop:V11_proj} below generalizes to show that the unit object is projective, and therefore all rigid objects are projective. But to show that all projective objects are also rigid, one would need to prove rigidity for suitable non-trivial objects of $KL^k(\fg)$. In particular, as in this paper, one could first try to use KZ equations to prove rigidity of the generalized Verma module induced from the irreducible $\fg$-module whose highest weight is the first fundamental weight of $\fg$. But this depends on how explicitly the KZ equations can be solved.

Related to the second question, it is tempting to conjecture that, as in Theorem \ref{thm:Tzeta_Pk_equiv}, the subcategory of tilting modules in $\cC(\zeta,\fg)$ embeds into $KL^k(\fg)$ as the full subcategory of projective objects, at least in cases where $KL^k(\fg)$ is a tensor category. But our proof of Theorem \ref{thm:Tzeta_Pk_equiv} heavily relies on explicit structural information for indecomposable tilting modules of quantum $\fsl_2$, and this may be difficult to obtain in higher rank cases. Thus it may be difficult to obtain a tensor functor $\cF: KL^k(\fg)\rightarrow\cC(\zeta,\fg)$ using a universal property like Theorem \ref{thm:intro_univ_prop}. It might be worth exploring whether Kazhdan and Lusztig's original methods in \cite{KL3, KL4} could generalize to this situation. 

Another method for proving equivalences between module categories for vertex operator algebras and quantum groups is currently under development in \cite{CLR, Len} and might also possibly be useful here. This method requires a free field realization of the vertex operator algebra whose twisted representation theory is equivalent to representations of the Borel part of the quantum group. For $V^k(\fg)$, one might try to use the Wakimoto module realization \cite{Wa, Fr}. Even in the $\fsl_2$ case, it would interesting to see whether such a free field realization could be used to recover Theorem \ref{thm:intro_weak_KL} without using the universal property of the category of tilting modules for quantum $\fsl_2$.

\subsection{Outline}

The remaining contents of this paper are structured as follows. In Section \ref{sec:tens_cat} we show that $KL^k(\fsl_2)$ is a braided tensor category and discuss some basic properties of its tensor category structure. Section \ref{subsec:sl2_hat} introduces the affine Lie algebra $\widehat{\fsl}_2$ and the universal affine vertex operator algebra $V^k(\fsl_2)$. Then in Section \ref{subsec:Malikov}, we use Malikov's results \cite{Ma} on Verma modules for $\widehat{\fsl}_2$ to determine the structure of generalized Verma modules. Section \ref{subsec:tens_cat} defines and characterizes the Kazhdan-Lusztig category $KL^k(\fsl_2)$, shows that it is a braided tensor category, and then discusses some basic results on projective objects in $KL^k(\fsl_2)$. In Section \ref{subsec:KL_k}, we show that the subcategory $KL_k(\fsl_2)$ of modules for the simple affine vertex operator algebra $KL_k(\fsl_2)$ is a tensor subcategory of $KL^k(\fsl_2)$ with a different unit object, and is also a tensor ideal. Then in Section \ref{subsec:intw_op}, we prove some results on intertwining operators and tensor products involving the generalized Verma module $\cV_2$ in $KL^k(\fsl_2)$.

In Section \ref{sec:rigidity}, we prove that the generalized Verma module $\cV_2$ is rigid and self-dual in $KL^k(\fsl_2)$, and we calculate its intrinsic dimension, using KZ equations.

In Section \ref{sec:proj}, we prove Theorems \ref{thm:intro_Pr} and \ref{thm:intro_V2_times_Pr}. First in Section \ref{subsec:tens_prod_V2}, we completely determine how $\cV_2$ tensors with all generalized Verma and simple modules in $KL^k(\fsl_2)$. Using these results, we construct the indecomposable modules $\cP_r$, $r\in\ZZ_{\geq 1}$, in Section \ref{subsec:further_indecomp}, and we show that they are projective and rigid in $KL^k(\fsl_2)$. As a consequence, we show that rigid objects in $KL^k(\fsl_2)$ are the same as projective objects. Then in Section \ref{subsec:more_prop}, we show that $\cP_r$ is self-contragredient if $r\geq p$ and logarithmic if $r>p$ and $p\nmid r$.

In Section \ref{sec:braiding}, we determine all braidings and ribbon twists on $KL^k(\fsl_2)$ and its $3$-cocycle twist $KL^k(\fsl_2)^\tau$. We also determine a criterion for tensor functors $\cF: KL^k(\fsl_2)\rightarrow\cC$ to be braided, where $\cC$ is any braided tensor category.

In Section \ref{sec:univ_prop}, we derive the universal property of $KL^k(\fsl_2)$. We first discuss the structure of the category $\cT_\zeta$ of tilting modules for quantum $\fsl_2$ at a root of unity $\zeta$ in Section \ref{subsec:tilting}, and we show that as a monoidal category, $\cT_\zeta$ embeds into $KL^k(\fsl_2)$ as the full subcategory of projective objects. 
We then prove Theorem \ref{thm:intro_univ_prop} in Section \ref{subsec:univ_prop}. We also determine criteria for the tensor functors $\cF$ guaranteed by Theorem \ref{thm:intro_univ_prop} to be braided or exact.

In Section \ref{sec:applications}, we give applications of the universal property of $KL^k(\fsl_2)$. First we classify the categories $KL^k(\fsl_2)$ up to (braided) tensor equivalence in Section \ref{subsec:classify}, proving Theorem \ref{thm:intro_classify}.  We then prove the weak and derived Kazhdan-Lusztig correspondences (Theorems \ref{thm:intro_weak_KL} and \ref{thm:intro_derived_KL_corr}) in Section \ref{subsec:weak_KL}. We obtain the tensor functors $\cF_{p,q}$ and $\cF_{q,p}$ to the Virasoro category $\cO_{c_{p,q}}$ and state Conjecture \ref{conj:intro} in Section \ref{subsec:Vir}. 

\medskip

 \noindent{\bf Acknowledgments.} We thank Nicolai Reshetikhin for raising the question of possible relations between tilting modules for quantum $\fsl_2$ at a root of unity and vertex operator algebras, that largely inspired us to begin working on this paper. We thank Cris Negron for explaining to us how Theorem \ref{thm:intro_derived_KL_corr} follows from our results. We thank Azat Gainutdinov for informing us of his construction with Saleur of an abelianization of the Temperley-Lieb category in \cite{GS}. We also thank Azat Gainutdinov, Cris Negron, and Victor Ostrik for discussions on tilting modules for quantum $\fsl_2$, and we thank Yi-Zhi Huang, David Ridout, Siddhartha Sahi, Simon Wood, and the referees for comments.

\section{Tensor category structure on \texorpdfstring{$KL^k(\fsl_2)$}{KLk(sl2)}}\label{sec:tens_cat}

In this section, we will define the Kazhdan-Lusztig category $KL^k(\fsl_2)$ and show that it is a braided tensor category. We begin with the affine Lie algebra $\widehat{\fsl}_2$ and its associated universal affine vertex operator algebras.

\subsection{The affine Lie algebra \texorpdfstring{$\slhat_2$}{sl2-hat}}\label{subsec:sl2_hat}

As usual, $\fsl_2$ is the simple Lie algebra over $\CC$ with basis $\lbrace e,f,h\rbrace$ and Lie brackets
\begin{equation*}
[e,f]=h,\qquad [h,e]=2e,\qquad [h,f]=-2f.
\end{equation*}
We fix the non-degenerate invariant bilinear form $\langle\cdot,\cdot\rangle$ on $\fsl_2$ such that
\begin{equation*}
\langle e, f\rangle =\langle f,e\rangle =1,\qquad\langle h,h\rangle =2,
\end{equation*}
with all other pairings of basis elements $0$. We always use the Cartan subalgebra $\mathfrak{h}=\CC h$, so that the root lattice of $\fsl_2$ is $Q=\ZZ\alpha$ where $\alpha\in\mathfrak{h}^*$ satisfies $\alpha(h)=2$. We denote the weight lattice $\ZZ\frac{\alpha}{2}$ by $P$.

The affine Lie algebra associated to $\fsl_2$ is
\begin{equation*}
\slhat_2=\fsl_2\otimes\CC[t,t^{-1}]\oplus\CC\mathbf{k}
\end{equation*}
with $\mathbf{k}$ central and 
\begin{equation*}
[a\otimes t^m,b\otimes t^n]=[a,b]\otimes t^{m+n}+m\langle a,b\rangle\mathbf{k}
\end{equation*}
for $a,b\in\fsl_2$ and $m,n\in\ZZ$. The affine Lie algebra has a triangular decomposition
\begin{equation*}
\slhat_2 =(\slhat_2)_+\oplus(\slhat_2)_0\oplus(\slhat_2)_-
\end{equation*}
where
\begin{equation*}
(\slhat_2)_{\pm} =\fsl_2\otimes t^{\pm 1}\CC[t^{\pm 1}],\qquad(\slhat_2)_0=\fsl_2\otimes t^0\oplus\CC\mathbf{k}.
\end{equation*}
We also set $(\slhat_2)_{\geq 0}=(\slhat_2)_+\oplus(\slhat_2)_0$.

For any $a\in\fsl_2$ and $n\in\ZZ$, we use the notation $a(n)$ to denote the action of $a\otimes t^n$ on an $\slhat_2$-module. We say that an $\slhat_2$-module $W$ has \textit{level} $k\in\CC$ if the central element $\mathbf{k}$ acts on $W$ by the scalar $k$. For any level $k\in\CC$, generalized Verma $\slhat_2$-modules are constructed as follows: For any $\fsl_2$-module $M$, we extend $M$ to an $(\slhat_2)_{\geq 0}$ module on which $\mathbf{k}$ acts by $k$ and $(\slhat_2)_+$ acts trivially. Then the generalized Verma module $\cV^k_{M}$ is the induced $\slhat_2$-module
\begin{equation*}
\cV^k_{M}=\ind_{(\slhat_2)_{\geq 0}}^{\slhat_2} M.
\end{equation*} 
The generalized Verma module $\cV^k_{M}$ is linearly spanned by vectors of the form
\begin{equation*}
a_1(-n_1)\cdots a_j(-n_j)m
\end{equation*}
where $a_1,\ldots,a_j\in\fsl_2$, $n_1,\ldots,n_j\in\ZZ_{\geq 1}$, and $m\in M$. There is also a natural $\ZZ_{\geq 0}$-grading $\cV^k_{M}=\bigoplus_{n=0}^\infty \cV^k_{M}(n)$ where
\begin{equation*}
\cV^k_{M}(n)=\mathrm{span}\lbrace a_1(-n_1)\cdots a_j(-n_j)m\,\,\vert\,\,n_1+\ldots+n_j=n\rbrace
\end{equation*}
for $n\in\ZZ_{\geq 0}$. In case $M=M_r$ is the $r$-dimensional irreducible $\fsl_2$-module for some $r\in\ZZ_{\geq 1}$, we denote $\cV^k_{M_r}=\cV^k_{r}$, or simply $\cV_r$ if the level $k$ is understood.

The generalized Verma module $\cV^k_{1}$ induced from $M_1=\CC\vac$ is a vertex algebra (\cite{FZ}; see also \cite[Section 6.2]{LL}) with vacuum vector $\vac$. This means in particular there is a linear map
\begin{align*}
Y: \cV^k_1 & \rightarrow\mathrm{End}(\cV^k_1)[[x,x^{-1}]]\nonumber\\
v & \mapsto Y(v,x) =\sum_{n\in\ZZ} v_n\,x^{-n-1},
\end{align*}
called the vertex operator, which satisfies the vacuum property $Y(\vac, x)=\Id_{\cV_1^k}$ and the vertex algebra Jacobi identity of \cite[Chapter 8]{FLM}. In fact, $Y$ is determined by the axioms of a vertex algebra (see for example \cite[Definition 3.3.1]{LL}) together with
\begin{equation}\label{eqn:vertex_operator}
Y(a(-1)\vac,x)=a(x):=\sum_{n\in\ZZ} a(n)\,x^{-n-1}
\end{equation}
for $a\in\fsl_2$. We denote $\cV^k_{1}$ by $V^k(\fsl_2)$ when we consider it as a vertex algebra, and we call $V^k(\fsl_2)$ the \textit{universal affine vertex algebra} associated to $\fsl_2$ at level $k$.

 When $k$ is non-critical, which for $\fsl_2$ means $k\neq -2$, $V^k(\fsl_2)$ is also a vertex operator algebra in the sense of \cite{FLM, LL}, with conformal vector
\begin{equation*}
\omega=\frac{1}{2(k+2)}\left(e(-1)f(-1)\vac+\frac{1}{2}h(-1)^2\vac+f(-1)e(-1)\vac\right).
\end{equation*}
Writing $Y(\omega,x)=\sum_{n\in\ZZ} L(n)\,x^{-n-2}$, the vertex operator modes $L(n)$ define a representation of the Virasoro Lie algebra on $V^k(\fsl_2)$ with central charge $\frac{3k}{k+2}$. The most important Virasoro modes that we will use are
\begin{align}
L(0) & =\frac{1}{2(k+2)}  \left(e(0)f(0)+\frac{1}{2}h(0)^2+f(0)e(0)\right)\nonumber\\
&\hspace{5em}+\frac{1}{k+2}\sum_{n=1}^\infty \left(e(-n)f(n)+\frac{1}{2}h(-n)h(n)+f(-n)e(n)\right),\label{eqn:L0}\\
\label{eqn:L-1}
L(-1) & =\frac{1}{k+2}\sum_{n=0}^\infty\left(e(-n-1)f(n)+\frac{1}{2}h(-n-1)h(n)+f(-n-1)e(n)\right).
\end{align}
In general,
\begin{equation}\label{eqn:Vir_aff_commutator}
[L(m),a(n)]=-n a(m+n)
\end{equation}
for $m,n\in\ZZ$ and $a\in\fsl_2$. From \eqref{eqn:L0} and \eqref{eqn:Vir_aff_commutator}, the conformal weight grading $V^k(\fsl_2)=\bigoplus_{n\in\ZZ} V^k(\fsl_2)_{(n)}$ of $V^k(\fsl_2)$ by $L(0)$-eigenvalues agrees with the $\ZZ_{\geq 0}$-grading on generalized Verma modules discussed above, that is, $V^k(\fsl_2)_{(n)}=\cV^k_{1}(n)$ for all $n\in\ZZ$.

For any finite-dimensional $\fsl_2$-module $M$ and level $k$, the generalized Verma module $\cV^k_{M}$ is a $V^k(\fsl_2)$-module in the sense of \cite[Definition 4.1.1]{LL}. In particular, there is a vertex operator map
\begin{align*}
Y_{\cV_M^k}: V^k(\fsl_2) & \rightarrow \mathrm{End}(\cV^k_M)[[x,x^{-1}]]
\end{align*}
 also characterized by \eqref{eqn:vertex_operator}. Moreover, if $k\neq -2$, then $L(0)$ acts on $\cV^k_{M}(0)$ by $\frac{1}{2(k+2)}\Omega$ where $\Omega=ef+\frac{1}{2}h^2+fe$ is the Casimir operator associated to the bilinear form $\langle\cdot,\cdot\rangle$. Thus if $M=M_r$ is the $r$-dimensional irreducible $\fsl_2$-module for some $r\in\ZZ_{\geq 1}$, then $\cV^k_{r}(n)$ for any $n\in\ZZ_{\geq 0}$ is the $L(0)$-eigenspace with eigenvalue $h_r+n$, where
\begin{equation}\label{eqn:h_lambda}
h_r=\frac{r^2-1}{4(k+2)}.
\end{equation}
So $\cV^k_{r}$ is a module for $V^k(\fsl_2)$ considered as a vertex operator algebra, in the sense of \cite[Definition 4.1.6]{LL}, with a conformal weight grading given by $L(0)$-eigenvalues.

If $k\neq -2$ and $r\in\ZZ_{\geq 1}$, then any submodule of $\cV^k_r$ is $L(0)$-stable and thus graded, so $\cV^k_{r}$ has a unique maximal proper submodule $\cJ^k_{r}$ given by the sum of all (graded) submodules that intersect $\cV^k_{r}(0)=M_r$ trivially. We denote the unique irreducible quotient $\cV^k_{r}/\cJ^k_{r}$ by $\cL^k_r$, or simply $\cL_r$ if the level $k$ is understood. Any irreducible $V^k(\fsl_2)$-module is isomorphic to $\cL^k_{r}$ for some $r\in\ZZ_{\geq 1}$ \cite[Theorem 6.2.23]{LL}. The simple module $\cL^k_{1}$ is the unique simple vertex operator algebra quotient of $V^k(\fsl_2)$, which we denote by $L_k(\fsl_2)$ when we consider it as a vertex operator algebra. When $L_k(\fsl_2)$ is a proper quotient of $V^k(\fsl_2)$, that is, when $V^k(\fsl_2)$ is not simple, only a subset of $\cL^k_{r}$ for $r\in\ZZ_{\geq 1}$ are $L_k(\fsl_2)$-modules.

We close this subsection with a discussion of the affine Kac-Moody algebra associated to $\slhat_2$. For more details on affine Kac-Moody algebras, see for example \cite{Ka,Ca}; here we mainly use the notation of \cite{MY1}. We define
\begin{equation*}
\til{\fsl}_2=\slhat_2\oplus\CC\mathbf{d}
\end{equation*}
where
\begin{equation*}
[\mathbf{d},\mathbf{k}]=0,\qquad [\mathbf{d},a\otimes t^n]=n(a\otimes t^n).
\end{equation*}
The affine Kac-Moody Lie algebra $\til{\fsl}_2$ has Cartan subalgebra
\begin{equation*}
\mathfrak{H}=\mathfrak{h}\oplus\CC\mathbf{k}\oplus\CC\mathbf{d};
\end{equation*}
the simple coroots in $\mathfrak{H}$ are $h_0=-h+\mathbf{k}$ and $h_1=h$. The dual
\begin{equation*}
\mathfrak{H}^*=\mathfrak{h}^*\oplus\CC\mathbf{k}'\oplus\CC\mathbf{d}'
\end{equation*}
has basis $\lbrace\frac{\alpha}{2},\mathbf{k}',\mathbf{d}'\rbrace$ dual to the basis $\lbrace h,\mathbf{k},\mathbf{d}\rbrace$ of $\mathfrak{H}$. The simple roots of $\til{\fsl}_2$ are $\alpha_0=-\alpha+\mathbf{d}'$ and $\alpha_1=\alpha$; the real roots of $\til{\fsl}_2$ have the form $\pm\alpha+m\mathbf{d}'$ for $m\in\ZZ$, and the imaginary roots have the form $m\mathbf{d}'$ for $m\neq 0$. We will need a dominant integral weight $\rho\in\mathfrak{H}^*$ such that $\rho(h_0)=\rho(h_1)=1$; in fact, we can take $\rho=\frac{\alpha}{2}+2\mathbf{k}'$.

By the $m=0$ case of \eqref{eqn:Vir_aff_commutator}, any $\slhat_2$-module of level $k\neq -2$ with a well-defined action of $L(0)$ is also an $\til{\fsl}_2$-module on which $\mathbf{d}$ acts by $-L(0)$. Thus the only highest-weight $\til{\fsl}_2$-modules we will consider will have highest weights of the form 
\begin{equation*}
\Lambda^k_r=(r-1)\frac{\alpha}{2}+k\mathbf{k}'-h_r\mathbf{d}'
\end{equation*}
for $r\in\CC$ and $k\neq -2$. We denote the Verma $\til{\fsl}_2$-module of such a highest weight by $V^{\Lambda^k_r}$. For $r\in\ZZ_{\geq 1}$, the generalized Verma module $\cV^k_r$ is a quotient of $V^{\Lambda^k_r}$, and the proof of \cite[Proposition 2.1]{Le} (see also \cite[Proposition 3.5]{MY1}) shows that there is a short exact sequence
\begin{equation}\label{eqn:gen_Verma_exact_seq}
0\longrightarrow V^{\Lambda^k_{-r}}\longrightarrow V^{\Lambda^k_r}\longrightarrow\cV^k_r\longrightarrow 0.
\end{equation}
Note that $\cL^k_r$ is the unique irreducible quotient of both $V^{\Lambda^k_r}$ and $\cV^k_r$.

\subsection{Structure of generalized Verma modules}\label{subsec:Malikov}

We now fix $k=-2+\kappa$ where $\kappa\in\QQ_{>0}$. We write $\kappa=p/q$ where $p,q\in\ZZ_{\geq 1}$ are relatively prime. In this subsection, we determine the structure of the generalized Verma $\slhat_2$-modules $\cV^k_{r}$ for $r\in\ZZ_{\geq 1}$, showing in particular that they have finite length. This result is easily deduced from the results of Rocha-Caridi and Wallach \cite{RW} and Malikov \cite{Ma} on the structure of Verma modules for rank-$2$ Kac-Moody Lie algebras. Here we use \cite{Ma} as a reference, since the results there are general enough to cover non-integer levels $k$.

Theorem A(1) in \cite{Ma} describes the structure of the Verma module $V^\Lambda$ for $\Lambda\in\mathfrak{H}^*$ satisfying the following conditions:
\begin{itemize}
\item $\Lambda$ is in the Tits cone, that is, $(\Lambda+\rho)(\mathbf{k})\in\RR_{\geq 0}$.

\item $(\Lambda+\rho)(h_\beta)\in\ZZ$ for at least two positive real roots $\beta$.
\end{itemize}
In particular, the weights $\Lambda^k_r =(r-1)\frac{\alpha}{2}+k\mathbf{k}'-h_r\mathbf{d}'$ for $r\in\ZZ$ satisfy these conditions (the first because our shifted level $\kappa$ is a positive rational number). Since the level $k$ is fixed, we will denote $\Lambda^k_r$ by $\Lambda_r$ from now on.

Given a weight $\Lambda_r$ for some $r\in\ZZ$, and following \cite[Section 2]{Ma}, define $\Delta_k^+$ to be the set of positive roots $\beta$ such that $(\Lambda_r+\rho)(h_\beta)\in\ZZ$. Since
\begin{equation*}
(\Lambda_r+\rho)(h_\beta) =\left\lbrace\begin{array}{lll}
r &\text{if} & \beta=\alpha\\
\pm r+mp/q & \text{if} & \beta=\pm\alpha+m\mathbf{d}',\,\,\,m\in\ZZ_{\geq 1}\\
mp/q & \text{if} & \beta=m\mathbf{d}',\,\,\,m\in\ZZ_{\geq 1}
\end{array}\right. ,
\end{equation*}
$\Delta_k^+$ is independent of $r$ and consists of the roots $\alpha$, $\pm\alpha+mq\mathbf{d}'$ for $m\in\ZZ_{\geq 1}$, and $mq\mathbf{d}'$ for $m\in\ZZ_{\geq 1}$. In fact, $\Delta_k^+$ forms the set of positive roots of a root subsystem that is also of type $\til{\fsl}_2$. For this root subsystem, we choose simple roots $\beta_0=-\alpha+q\mathbf{d}'$, $\beta_1=\alpha$ with corresponding coroots $h_{\beta_0}=-h+q\mathbf{k}$, $h_{\beta_1}=h$. Let $W_k=\langle s_0,s_1\rangle$ be the subgroup of the Weyl group of $\til{\fsl}_2$ generated by the reflections $s_0$ and $s_1$ associated to the roots $\beta_0$ and $\beta_1$, respectively.

Recall the dot action of $W_k$ on the weights of $\til{\fsl}_2$:
\begin{equation*}
w\cdot\Lambda=w(\Lambda+\rho)-\rho
\end{equation*}
for $w\in W_k$, $\Lambda\in\mathfrak{H}^*$. In particular,
\begin{equation*}
s_i\cdot\Lambda=\Lambda-(\Lambda+\rho)(h_{\beta_i})\beta_i
\end{equation*}
for $i=0,1$. The following lemma is an elementary calculation:
\begin{lem}\label{lem:dot_action}
For $r\in\ZZ$, we have $s_0\cdot\Lambda_r =\Lambda_{2p-r}$ and $s_1\cdot\Lambda_r = \Lambda_{-r}$.
\end{lem}

From \cite[Lemma 4.1]{Ma}, we can deduce that $V^{\Lambda_{r}}$ for any $r\in\ZZ$ embeds in a unique Verma module $V^{\Lambda}$ such that $(\Lambda+\rho)(h_{\beta_i})\in\ZZ_{\geq 0}$ for $i=0,1$. A quick calculation shows that the $\Lambda$ such that $(\Lambda+\rho)(h_{\beta_i})\in\ZZ_{\geq 0}$ for both of $i=0,1$ are precisely the $\Lambda_r$ for $0\leq r\leq p$. For $1\leq r\leq p-1$, $\Lambda_{r}$ is not fixed by the dot action of either $s_0$ or $s_1$, and \cite[Lemma 4.1(1)]{Ma} shows we have an embedding diagram of Verma modules
\begin{equation*}
\xymatrixrowsep{1pc}
\xymatrix{
 & \ar[ld] V^{s_0\cdot\Lambda_{r}} & \ar[l] \ar[ldd] V^{s_1s_0\cdot\Lambda_{r}} & \ar[l] \ar[ldd] V^{s_0s_1s_0\cdot\Lambda_{r}} & \ar[l]\ar[ldd] V^{s_1 s_0 s_1 s_0\cdot\Lambda_r} & \ar[l]\ar[ldd]\cdots \\
V^{\Lambda_{r}} & & & & &\\
 &  \ar[lu] V^{s_1\cdot\Lambda_{r}} & \ar[l] \ar[luu] V^{s_0s_1\cdot\Lambda_{r}} & \ar[l]\ar[luu] V^{s_1s_0s_1\cdot\Lambda_{r}} & \ar[l]\ar[luu] V^{s_0 s_1 s_0 s_1\cdot\Lambda_r} & \ar[l]\ar[luu]\cdots \\
}
\end{equation*}
On the other hand, $\Lambda_{0}$ is fixed by the dot action of $s_1$ and $\Lambda_{p}$ is fixed by the dot action of $s_0$, so \cite[Lemma 4.1(2)]{Ma} gives embedding diagrams
\begin{equation*}
\xymatrix{
V^{\Lambda_{0}} & \ar[l] V^{s_0\cdot\Lambda_{0}} & \ar[l] V^{s_1s_0\cdot\Lambda_{0}} & \ar[l] V^{s_0s_1s_0\cdot\Lambda_{0}} & \ar[l] V^{s_1 s_0 s_1 s_0\cdot\Lambda_0} & \ar[l]\cdots \\
}
\end{equation*}
and
\begin{equation*}
\xymatrix{
V^{\Lambda_{p}} & \ar[l] V^{s_1\cdot\Lambda_{p}} & \ar[l] V^{s_0s_1\cdot\Lambda_{p}} & \ar[l] V^{s_1s_0s_1\cdot\Lambda_{p}} & \ar[l] V^{s_0 s_1 s_0 s_1\cdot\Lambda_p} & \ar[l]\cdots \\
}
\end{equation*}
In these diagrams, each arrow represents the unique (up to scaling) homomorphism from one Verma module to another, and each homomorphism of Verma modules is injective.

By Lemma \ref{lem:dot_action}, the embedding diagrams become
\begin{equation*}
\xymatrixrowsep{1pc}
\xymatrix{
& \ar[ld]  V^{\Lambda_{2p-r}} & \ar[l]\ar[ldd] V^{\Lambda_{-2p+r}} & \ar[l] \ar[ldd] V^{\Lambda_{4p-r}} & \ar[l] \ar[ldd] V^{\Lambda_{-4p+r}} & \ar[l]\ar[ldd]\cdots \\
V^{\Lambda_{r}} & & & & & \\
& \ar[lu] V^{\Lambda_{-r}} & \ar[l] \ar[luu] V^{\Lambda_{2p+r}} & \ar[l]\ar[luu] V^{\Lambda_{-2p-r}} & \ar[l]\ar[luu] V^{\Lambda_{4p+r}} & \ar[l]\ar[luu] \cdots \\
}
\end{equation*}
for $1\leq r\leq p-1$,
\begin{equation*}
\xymatrix{
V^{\Lambda_{0}} & \ar[l] V^{\Lambda_{2p}} & \ar[l] V^{\Lambda_{-2p}} & \ar[l] V^{\Lambda_{4p}} & \ar[l] V^{\Lambda_{-4p}} & \ar[l]\cdots \\
}
\end{equation*}
and
\begin{equation*}
\xymatrix{
V^{\Lambda_{p}} & \ar[l] V^{\Lambda_{-p}} & \ar[l] V^{\Lambda_{3p}} & \ar[l] V^{\Lambda_{-3p}} & \ar[l] V^{\Lambda_{5p}} & \ar[l]\cdots \\
}
\end{equation*}
Now a crucial consequence of \cite[Theorem A(1)]{Ma} (see \cite[Corollary 2.1(2)]{Ma}) is that every submodule of every Verma module $V^{\Lambda_{r}}$ is generated by its singular vectors, and thus we can read off the maximal proper submodule of $V^{\Lambda_{r}}$ from the embedding diagrams. In particular, for $n\in\ZZ_{\geq 0}$ and $1\leq r\leq p-1$, we have an exact sequence
\begin{equation}\label{eqn:Verma_structure_s}
0\longrightarrow K_{np+r}\longrightarrow V^{\Lambda_{np+r}} \longrightarrow\cL_{np+r}\longrightarrow 0
\end{equation}
where $K_{np+r}$ is the sum of two Verma submodules $V^{s_1\cdot\Lambda_{np+r}}=V^{\Lambda_{-np-r}}$ and $V^{\Lambda_{(n+2)p-r}}$. We also have an exact sequence
\begin{equation}\label{eqn:Verma_structure_p}
0\longrightarrow V^{\Lambda_{-np}} \longrightarrow V^{\Lambda_{np}}\longrightarrow\cL_{np}\longrightarrow 0,
\end{equation}
for $n\in\ZZ_{\geq 1}$.

It is now easy to use \eqref{eqn:gen_Verma_exact_seq}, \eqref{eqn:Verma_structure_s}, and \eqref{eqn:Verma_structure_p} to obtain the structure of the generalized Verma modules $\cV_r$ for $r\in\ZZ_{\geq 1}$:
\begin{thm}\label{thm:gen_Verma_structure}
For $n\in\ZZ_{\geq 1}$, the generalized Verma module $\cV_{np}$ is irreducible. For $n\in\ZZ_{\geq 0}$ and $1\leq r\leq p-1$, there is a short exact sequence
\begin{equation}\label{exactseq:vrs}
0\longrightarrow\cL_{(n+2)p-r}\longrightarrow\cV_{np+r}\longrightarrow\cL_{np+r}\longrightarrow 0.
\end{equation}
In particular, $\cV_{r}$ has finite length for all $r\in\ZZ_{\geq 1}$.
\end{thm}
\begin{proof}
The irreducibility of $\cV_{np}$ is immediate from \eqref{eqn:gen_Verma_exact_seq} and \eqref{eqn:Verma_structure_p}. For $n\in\ZZ_{\geq 0}$ and $1\leq r\leq p-1$, \eqref{eqn:gen_Verma_exact_seq} and \eqref{eqn:Verma_structure_s} yield a commutative diagram
\begin{equation*}
\xymatrixrowsep{.5pc}
\xymatrix{
0 \ar[r] & K_{np+r} \ar[r] \ar[dd] & V^{\Lambda_{np+r}} \ar[dd] \ar[rd] & & \\
& & & \cL_{np+r} \ar[r] & 0\\
0 \ar[r] & K_{np+r}/V^{\Lambda_{-np-r}} \ar[r] & \cV_{np+r} \ar[ru] &  & \\
}
\end{equation*}
with exact rows and surjective vertical arrows. So the maximal proper submodule $\cJ_{np+r}$ of $\cV_{np+r}$ is given by
\begin{equation*}
\cJ_{np+r}\cong K_{np+r}/V^{\Lambda_{-np-r}}\cong V^{\Lambda_{(n+2)p-r}}/(V^{\Lambda_{-np-r}}\cap V^{\Lambda_{(n+2)p-r}}).
\end{equation*}
Since every submodule of $V^{\Lambda_{(n+2)p-r}}$ is generated by its singular vectors, the embedding diagrams show that $V^{\Lambda_{-np-r}}\cap V^{\Lambda_{(n+2)p-r}}$ is the maximal proper submodule of $V^{\Lambda_{(n+2)p-r}}$, that is, $\cJ_{np+r}\cong\cL_{(n+2)p-r}$ as required.
\end{proof}

\begin{rem}
The $p=1$ case of Theorem \ref{thm:gen_Verma_structure}, that is, $k=-2+1/q$ for $q\in\ZZ_{\geq 1}$, has appeared in \cite{Cr}; in this case, all generalized Verma modules are irreducible. For $k\in\ZZ_{\geq 0}$, the short exact sequence \eqref{exactseq:vrs} has appeared in \cite{MY1}, using the results of \cite{RW} for the structure of Verma modules whose highest weights are Weyl group translates of dominant integral weights (rather than the more general results of \cite{Ma}).
\end{rem}

\subsection{Tensor category structure}\label{subsec:tens_cat}

We first define the Kazhdan-Lusztig category for $\slhat_2$ at positive rational shifted level following \cite[Definition 2.15]{KL1}:
\begin{defi}
Fix a level $k=-2+p/q$ for $p,q\in\ZZ_{\geq 1}$ relatively prime. The \textit{Kazhdan-Lusztig category} for $\fsl_2$ at level $k$ is the category $KL^k(\fsl_2)$ of finite-length level-$k$ $\slhat_2$-modules whose composition factors come from the irreducible modules $\cL_{r}$, $r\in\ZZ_{\geq 1}$.
\end{defi}

Objects of $KL^k(\fsl_2)$ are modules for $V^k(\fsl_2)$ considered as a vertex algebra \cite[Theorem 6.2.13]{LL}, but they might not be modules for $V^k(\fsl_2)$ considered as a vertex operator algebra, in the sense of \cite[Definition 4.1.6]{LL}, because $L(0)$ might act non-semisimply. However, since $L(0)$ acts semisimply on the simple objects $\cL_r$ of $KL^k(\fsl_2)$, and since objects of $KL^k(\fsl_2)$ have finite length, $L(0)$ does act locally finitely on any object of $KL^k(\fsl_2)$. Thus any object $W$ of $KL^k(\fsl_2)$ is a \textit{grading-restricted generalized $V^k(\fsl_2)$-module}: It is a $V^k(\fsl_2)$-module in the sense of \cite[Definition 4.1.1]{LL} that also has a \textit{conformal weight} grading
\begin{equation*}
W=\bigoplus_{h\in\CC} W_{[h]},
\end{equation*}
where $W_{[h]}$ is the generalized $L(0)$-eigenspace with generalized eigenvalue $h$, such that $\dim W_{[h]} <\infty$ for all $h\in\CC$, and for any $h\in\CC$, $W_{[h+n]}=0$ for all sufficiently negative $n\in\ZZ$. In fact, by \eqref{eqn:h_lambda}, all generalized $L(0)$-eigenvalues on $W$ are non-negative rational numbers when $k+2\in\QQ_{>0}$.

Any grading-restricted generalized $V^k(\fsl_2)$-module $W=\bigoplus_{h\in\CC} W_{[h]}$ has a \textit{contragredient} $W'=\bigoplus_{h\in\CC} W_{[h]}^*$ (so as vector spaces, $W'$ is the \textit{graded dual} of $W$). The $V^k(\fsl_2)$-module vertex operator on $W'$ is defined by
\begin{equation*}
\langle Y_{W'}(v,x)w', w\rangle =\langle w', Y_W(e^{xL(1)}(-x^{-2})^{L(0)} v,x^{-1})w\rangle
\end{equation*}
for $v\in V^k(\fsl_2)$, $w\in W$, $w'\in W'$. This means the $\widehat{\fsl}_2$-module structure on $W'$ is given by
\begin{equation}\label{eqn:sl2_contra_structure}
\langle a(n) w', w\rangle =-\langle w', a(-n) w\rangle
\end{equation}
for $a\in\fsl_2$, $n\in\ZZ$, $w\in W$, $w'\in W'$. Contragredients induce an exact contravariant functor on the category of grading-restricted generalized $V^k(\fsl_2)$-modules, with the contragredient of a morphism $f: W_1\rightarrow W_2$ defined in the natural way:
\begin{equation*}
\langle f'(w_2'), w_1\rangle =\langle w_2', f(w_1)\rangle
\end{equation*}
for $w_1\in W_1$ and $w_2'\in W_2'$. The category $KL^k(\fsl_2)$ is closed under taking contragredients because the contragredient module $W'$ has the same length as $W$.

In the next proposition, we characterize $KL^k(\fsl_2)$ as a category of grading-restricted generalized $V^k(\fsl_2)$-modules in three different ways. Before stating this result, we recall the notion of $C_1$-cofinite module for a vertex operator algebra $V$. If $W$ is a grading-restricted generalized $V$-module, let $C_1(W)\subseteq W$ denote the subspace of $W$ spanned by vectors of the form $v_{-n} w$ for $v\in V$ of strictly positive conformal weight, $n\geq 1$, and $w\in W$. Then $W$ is \textit{$C_1$-cofinite} if $\dim W/C_1(W)<\infty$. From \eqref{eqn:vertex_operator}, a grading-restricted generalized $V^k(\fsl_2)$-module is $C_1$-cofinite if it is finitely generated as an $(\widehat{\fsl}_2)_-$-module.

\begin{prop}\label{prop:KL_cat_char}
The category $KL^k(\fsl_2)$ is equal to all of the following categories:
\begin{itemize}
\item The category of finite-length grading-restricted generalized $V^k(\fsl_2)$-modules.

\item The category of finitely-generated grading-restricted generalized $V^k(\fsl_2)$-modules.

\item The category of $C_1$-cofinite grading-restricted generalized $V^k(\fsl_2)$-modules.
\end{itemize}
\end{prop}
\begin{proof}
Every grading-restricted generalized $V^k(\fsl_2)$-module is an $\slhat_2$-module of level $k$, and every irreducible subquotient of a grading-restricted generalized $V^k(\fsl_2)$-module is grading-restricted. So because the modules $\cL_{r}$, $r\in\ZZ_{\geq 1}$, exhaust the irreducible grading-restricted $V^k(\fsl_2)$-modules up to isomorphism by \cite[Theorem 6.2.23]{LL}, every finite-length grading-restricted generalized $V^k(\fsl_2)$-module is an object of $KL^k(\fsl_2)$. Conversely, by \cite[Theorem 6.2.7]{LL}, every object of $KL^k(\fsl_2)$ is a $V^k(\fsl_2)$-module in the sense of \cite[Definition 4.1.1]{LL} and is moreover a grading-restricted generalized $V^k(\fsl_2)$-module because it is a finite-length module with grading-restricted irreducible subquotients.

Next, any $C_1$-cofinite grading-restricted generalized $V^k(\fsl_2)$-module is finitely generated (see for example \cite[Propositions 2.1 and 2.2]{CMY-completions}). Conversely, if $W$ is a grading-restricted generalized $V^k(\fsl_2)$-module that is generated by a finite set $S$, then the grading-restriction conditions imply that $U((\slhat_2)_{\geq 0})\cdot S$ is finite dimensional, so that
\begin{equation*}
W=U(\slhat_2)\cdot S=U((\slhat_2)_-)U((\slhat_2)_{\geq 0})\cdot S
\end{equation*}
is a finitely-generated $(\slhat_2)_-$-module. This means $W$ is $C_1$-cofinite.

Finally, the equality of the categories of finite-length and of $C_1$-cofinite grading-restricted generalized $V^k(\fsl_2)$-modules follows from \cite[Theorem 3.3.5]{CY}, since every generalized Verma module $\cV_{r}$ has finite length by Theorem \ref{thm:gen_Verma_structure}, and since every irreducible grading-restricted $V^k(\fsl_2)$-module $\cL_{r}$ is $C_1$-cofinite.
\end{proof}

We will use this proposition together with results in \cite{CY} to show that $KL^k(\fsl_2)$ is a braided tensor category as described in \cite{HLZ1}-\cite{HLZ8}. But first, we recall how the tensor product operation $\tens: KL^k(\fsl_2)\times KL^k(\fsl_2)\rightarrow KL^k(\fsl_2)$ is defined. Tensor products of modules for a vertex operator algebra $V$ are defined in terms of intertwining operators, which in turn are defined in \cite[Definition 3.10]{HLZ2}, for example. In particular, for grading-restricted generalized $V$-modules $W_1$, $W_2$, $W_3$, an intertwining operator of type $\binom{W_3}{W_1\,W_2}$ is a linear map
\begin{align*}
\cY: W_1\otimes W_2 & \rightarrow W_3[\log x]\lbrace x\rbrace\\
w_1\otimes w_2 & \mapsto \cY(w_1,x)w_2=\sum_{h\in\CC}\sum_{k\in\ZZ_{\geq 0}} (w_1)_{h;k} w_2\,x^{-h-1}(\log x)^k 
\end{align*}
which satisfies several properties, especially the $L(-1)$-derivative property
\begin{equation*}
\dfrac{d}{dx}\cY(w_1,x)w_2 =\cY(L(-1)w_1,x)w_2
\end{equation*}
and the intertwining operator Jacobi identity. For $V^k(\fsl_2)$-modules, the Jacobi identity amounts to the following commutator and iterate formulas:
\begin{align}
a(n)\cY(w_1,x) & =\cY(w_1,x)a(n)+\sum_{i\geq 0} \binom{n}{i} x^{n-i}\cY(a(i)w_1,x)\label{eqn:intw_op_comm}\\
 \cY(a(n) w_1,x) & = \sum_{i\geq 0}\binom{n}{i} (-x)^i a(n-i)\cY(w_1,x) - \sum_{i\geq 0}\binom{n}{i}(-x)^{n-i}\cY(w_1,x)a(i)\label{eqn:intwo_op_it}
\end{align}
for $a\in\mathfrak{sl}_2$, $w_1\in W_1$, and $n\in\ZZ$.

Now we can define tensor products of modules for a vertex operator algebra as follows (see \cite[Definition 4.15]{HLZ3} and also \cite[Proposition 4.8]{HLZ3}):
\begin{defi}
Let $V$ be a vertex operator algebra, $\cC$ a category of grading-restricted generalized $V$-modules, and $W_1$, $W_2$ objects of $\cC$. A \textit{tensor product} of $W_1$ and $W_2$ in $\cC$ is (if it exists) an object $W_1\tens W_2$ of $\cC$ equipped with an intertwining operator $\cY_\tens$ of type $\binom{W_1\tens W_2}{W_1\,W_2}$ satisfying the following universal property: For any object $W_3$ of $\cC$ and intertwining operator $\cY$ of type $\binom{W_3}{W_1\,W_2}$, there is a unique $V$-module homomorphism $f: W_1\tens W_2\rightarrow W_3$ such that $f\circ\cY_\tens=\cY$.
\end{defi}

It is not obvious in general when a given category of $V$-modules is closed under tensor products. Nevertheless, if a category $\cC$ is closed under tensor products, and if $\cC$ satisfies suitable further conditions, then it is shown in \cite{HLZ1}-\cite{HLZ8} that tensor products endow $\cC$ with the structure of a braided tensor category with unit object $V$. See \cite{HLZ8} or the exposition in \cite[Section 3.3]{CKM-exts} for a description of this braided tensor category structure. In particular, for an object $W$ of $\cC$, the left and right unit isomorphisms
\begin{equation*}
l_W: V\tens W\longrightarrow W,\qquad r_W: W\tens V\longrightarrow W
\end{equation*}
are characterized by
\begin{equation}\label{eqn:unit_isos}
l_W(\cY_\tens(v,x)w) = Y_W(v,x)w,\qquad r_W(\cY_\tens(w,x)v=e^{xL(-1)}Y_W(v,-x)w
\end{equation}
for $v\in V$, $w\in W$, and for objects $W_1$, $W_2$ of $\cC$, the braiding isomorphism
\begin{equation*}
\cR_{W_1,W_2}: W_1\tens W_2\longrightarrow W_2\tens W_1
\end{equation*}
is characterized by 
\begin{equation}\label{eqn:braiding_iso}
\cR_{W_1,W_2}(\cY_\tens(w_1,x)w_2) =e^{xL(-1)}\cY_\tens(w_2,e^{\pi i}x)w_1
\end{equation}
for $w_1\in W_1$, $w_2\in W_2$. For an object $W$ of $\cC$, the automorphism $\theta_W=e^{2\pi i L(0)}$ also defines a ribbon twist which satisfies the balancing equation
\begin{equation*}
\theta_{W_1\tens W_2} = \cR^2_{W_1,W_2}\circ(\theta_{W_1}\tens\theta_{W_2})
\end{equation*}
for objects $W_1$, $W_2$ of $\cC$, where $\cR_{W_1,W_2}^2=\cR_{W_2,W_1}\circ\cR_{W_1,W_2}$. 

For objects $W_1$, $W_2$, $W_3$ of $\cC$, the associativity isomorphism
\begin{equation*}
\cA_{W_1,W_2,W_3}: W_1\tens(W_2\tens W_3)\longrightarrow(W_1\tens W_2)\tens W_3
\end{equation*}
is more complicated: for any $r_1,r_2\in\RR$ such that $r_1>r_2>r_1-r_2>0$, we have an equality
\begin{equation}\label{eqn:assoc_iso}
\left\langle w', \overline{\cA_{W_1,W_2,W_3}}\left(\cY_{\tens}(w_1,r_1)\cY_\tens(w_2,r_2)w_3\right)\right\rangle =\langle w', \cY_\tens(\cY_\tens(w_1,r_1-r_2)w_2,r_2)w_3\rangle
\end{equation}
for $w_1\in W_1$, $w_2\in W_2$, $w_3\in W_3$, and $w'\in((W_1\tens W_2)\tens W_3)'$. Here we substitute positive real numbers for formal variables $x$ and $\log x$ in intertwining operators using the real-valued branch of logarithm to interpret complex powers of $x$ and integer powers of $\log x$, and $\overline{\cA_{W_1,W_2,W_3}}$ denotes the natural extension of the associativity isomorphism to the \textit{algebraic completion} of $W_1\tens(W_2\tens W_3)$, that is, the direct product (as opposed to direct sum) of the conformal weight spaces of $W_1\tens(W_2\tens W_3)$. 

\begin{rem}
It is a non-trivial problem in general to show that compositions of intertwining operators (with formal variables specialized to suitable non-zero complex numbers using some choice of branch of logarithm) converge to well-defined elements of the algebraic completion of a $V$-module. For intertwining operators among $C_1$-cofinite grading-restricted generalized $V$-modules, such convergence results are proved using regular singular point differential equations \cite{Hu-diff-eqns, HLZ7}. When $V$ is an affine vertex operator algebra such as $V^k(\fsl_2)$, such differential equations amount to Knizhnik-Zamolodchikov (KZ) equations \cite{KZ, HL}.
\end{rem}

The category of $C_1$-cofinite grading-restricted generalized $V$-modules is always closed under tensor products \cite{Mi}. Moreover, from \cite{HLZ7} and \cite[Theorem 3.3.4]{CY}, the category of $C_1$-cofinite grading-restricted generalized modules satisfies the further conditions of \cite{HLZ1}-\cite{HLZ8} for the existence of braided tensor category structure if it equals the category of finite-length modules. Thus from Proposition \ref{prop:KL_cat_char} we immediately conclude:
\begin{thm}\label{thm:existencebtc}
For any level $k=-2+\kappa$ with $\kappa\in\QQ_{>0}$, the category $KL^k(\fsl_2)$ admits the braided tensor category structure of \cite{HLZ1}-\cite{HLZ8}, with unit object $V^k(\fsl_2)=\cV_{1}$.
\end{thm}
\begin{rem}
While we use the notation $V^k(\fsl_2)$ for $\cV_{1}$ considered as a vertex operator algebra, we will typically use the notation $\cV_{1}$ when considering it as a $V^k(\fsl_2)$-module, or as an object of $KL^k(\fsl_2)$.
\end{rem}

\begin{rem}
By \textit{braided tensor category}, we mean a braided monoidal category which is also an abelian category, such that the tensor product bifunctor induces bilinear maps on morphisms. We do not require tensor categories to be rigid, that is, we do not require every object of a tensor category to have a (left or right) dual in the sense of tensor categories.
\end{rem}

We now discuss some properties of projective objects in $KL^k(\fsl_2)$:
\begin{prop}\label{prop:V11_proj}
The generalized Verma module $\cV_{1}$ is projective in $KL^k(\fsl_2)$.
\end{prop}
\begin{proof}
Consider a diagram
\begin{equation*}
\xymatrix{
& \cV_{1} \ar[d]^f\\
W \ar[r]^p & X\\
}
\end{equation*}
in $KL^k(\fsl_2)$ with $p$ surjective. We need to show that there is a morphism $g: \cV_{1}\rightarrow W$ such that $p\circ g=f$.

Both conformal weight spaces $W_{[0]}$ and $X_{[0]}$ are finite-dimensional (and thus semisimple) $\fsl_2$-modules on which $a\in\fsl_2$ acts by $a(0)$. Thus because $p$ is surjective, there is an $\fsl_2$-module homomorphism $q: X_{[0]}\rightarrow W_{[0]}$ such that $p\vert_{W_{[0]}}\circ q=\Id_{X_{[0]}}$. Then
\begin{equation*}
a(0)q(f(\vac))=q(f(a(0)\vac))=0.
\end{equation*}
Moreover, since the shifted level $\kappa=p/q$ is a positive number, the conformal weight $h_{r}=\frac{q}{4p}(r^2-1)$ is non-negative for all $r\in\ZZ_{\geq 1}$. Thus the conformal weights of $W$ are non-negative, and it follows that $a(n)q(f(\vac))=0$ for all $a\in\fsl_2$, $n>0$. Thus $\vac\mapsto q(f(\vac))$ defines an $(\slhat_2)_{\geq 0}$-module homomorphism from $\CC\vac$ to $W$, and then the universal property of generalized Verma modules induces a unique $\slhat_2$-module homomorphism $g:\cV_{1}\rightarrow W$ such that $g(\vac)=q(f(\vac))$. Thus
\begin{equation*}
(p\circ g)(\vac)=p(q(f(\vac)) =f(\vac),
\end{equation*}
and then $p\circ g=f$ as desired because $\vac$ generates $\cV_{1}$ as an $\slhat_2$-module.
\end{proof}

Tensor categories with projective unit objects are special; the following lemma generalizes \cite[Corollary 4.2.13]{EGNO} (see also \cite[Lemma 3.6]{McR-ss}) to tensor categories in which not every object is necessarily rigid:
\begin{lem}\label{lem:proj_unit_tens_cat}
Suppose $(\cC,\tens,\vac,l,r,\cA)$ is a tensor category with projective unit $\vac$ and such that the tensoring functor $\bullet\tens W$ preserves surjections for every object $W$ in $\cC$. Then every left rigid object of $\cC$ is projective.
\end{lem}
\begin{proof}
Suppose $R$ is an object of $\cC$ with left dual $R^*$, evaluation $e_R: R^*\tens R\rightarrow\vac$, and coevaluation $i_R:\vac\rightarrow R\tens R^*$, and consider a diagram
\begin{equation*}
\xymatrix{
& R\ar[d]^f\\
W \ar[r]^{p} & X\\
}
\end{equation*}
in $\cC$ with $p$ surjective. Then by assumption $p\tens\Id_{R^*}: W\tens R^*\rightarrow X\tens R^*$ is also surjective, so because $\vac$ is projective in $\cC$, we have a commutative diagram
\begin{equation*}
\xymatrixcolsep{4pc}
\xymatrix{
& \vac\ar[ld]_{\til{g}}\ar[d]^{(f\tens\Id_{R^*})\circ i_R}\\
W\tens R^* \ar[r]_{p\tens\Id_{R^*}} & X\tens R^*\\
}
\end{equation*}
for some morphism $\til{g}: \vac\mapsto W\tens R^*$. We define $g: R\rightarrow W$ to be the composition
\begin{equation*}
R\xrightarrow{l_R^{-1}} \vac\tens R\xrightarrow{\til{g}\tens\Id_R} (W\tens R^*)\tens R\xrightarrow{\cA_{W,R^*,R}^{-1}} W\tens(R^*\tens R)\xrightarrow{\Id_W\tens e_R} W\tens\vac\xrightarrow{r_W} W,
\end{equation*}
and then
\begin{align*}
p\circ g & = p\circ r_W\circ(\Id_W\tens e_R)\circ\cA_{W,R^*,R}^{-1}\circ(\til{g}\tens\Id_R)\circ l_R^{-1}\nonumber\\
& = r_X\circ(\Id_X\tens e_R)\circ\cA_{X,R^*,R}^{-1}\circ((p\tens\Id_{R^*})\tens\Id_R)\circ(\til{g}\tens\Id_R)\circ l_R^{-1}\nonumber\\
& =r_X\circ(\Id_X\tens e_R)\circ\cA_{X,R^*,R}^{-1}\circ((f\tens\Id_{R^*})\tens\Id_R)\circ(i_R\tens\Id_R)\circ l_R^{-1}\nonumber\\
&=f\circ r_R\circ(\Id_R\tens e_R)\circ\cA_{R,R^*,R}^{-1}\circ(i_R\tens\Id_R)\circ l_R^{-1}\nonumber\\
&=f
\end{align*}
using the left rigidity of $R$. Thus $R$ is projective.
\end{proof}

In $KL^k(\fsl_2)$, the tensoring functor $\bullet\tens W$ for any object $W$ is right exact by \cite[Proposition 4.26]{HLZ3}, so Proposition \ref{prop:V11_proj} and Lemma \ref{lem:proj_unit_tens_cat} imply:
\begin{cor}\label{cor:rigid_is_projective}
Every left rigid object of $KL^k(\fsl_2)$ is projective.
\end{cor}

\subsection{The tensor subcategory \texorpdfstring{$KL_k(\fsl_2)$}{KLk(sl2)}}\label{subsec:KL_k}
The maximal proper submodule of the generalized Verma module $\cV_{1}$ is also the maximal proper ideal of the vertex operator algebra $V^k(\fsl_2)$. Thus $\cL_{1}$ has a simple vertex operator algebra structure, which we denote $L_k(\fsl_2)$.
\begin{defi}
The category $KL_k(\fsl_2)$ is the category of finitely-generated grading-restricted generalized $L_k(\fsl_2)$-modules.
\end{defi}

Objects of $KL_k(\fsl_2)$ are precisely the finitely-generated grading-restricted generalized $V^k(\fsl_2)$-modules on which the maximal proper ideal acts trivially. Thus $KL_k(\fsl_2)$ is a full subcategory of $KL^k(\fsl_2)$. The case $p=1$, that is, $k=-2+1/q$ for $q\in\ZZ_{\geq 1}$, was analyzed in \cite{Cr,CY}. In this case, it can be seen from Theorem \ref{thm:gen_Verma_structure} that $V^k(\fsl_2)=L_k(\fsl_2)$ and that $\cV_{r}=\cL_{r}$ for all $r\in\ZZ_{\geq 1}$. From this, it is easy to show that $KL^k(\fsl_2)=KL_k(\fsl_2)$ is semisimple. Moreover, it is shown in \cite{CY} using the results of \cite{McR-orb, ACGY} that $KL^k(\fsl_2)$ is a rigid braided tensor category such that
\begin{equation*}
\cL_{r}\tens\cL_{r'} \cong\bigoplus_{\substack{r''=\vert r-r'\vert+1\\ r+r'+r''\equiv 1\,(\mathrm{mod}\,2)}}^{r+r'-1} \cL_{r''} 
\end{equation*}
for all $r,r'\in\ZZ_{\geq 1}$. If $p>1$, then $k=-2+p/q$ is an admissible level for $\fsl_2$ \cite{KW} and $L_k(\fsl_2)$ is a proper quotient of $V^k(\fsl_2)$. Known results about $KL_k(\fsl_2)$ are summarized in the following theorem:
\begin{thm}\label{thm:adm_level_KL_k}
Let $k=-2+p/q$ be an admissible level for $\fsl_2$.
\begin{enumerate}
\item \cite{AM, DLM} The category $KL_k(\fsl_2)$ of grading-restricted generalized $L_k(\fsl_2)$-modules is semisimple with simple objects $\cL_{r}$ for $1\leq r\leq p-1$.

\item \cite{CHY} The category $KL_k(\fsl_2)$ admits the vertex algebraic braided tensor category structure of \cite{HLZ1}-\cite{HLZ8} and is rigid. Moreover, $KL_k(\fsl_2)$ is a modular tensor category if and only if $q$ is odd.

\item \cite{BF, DLM, CHY} Tensor products of simple modules in $KL_k(\fsl_2)$ are as follows:
\begin{equation}\label{eqn:KL_k_adm_level_fus_rules}
\cL_{r}\tens\cL_{r'} \cong\bigoplus_{\substack{r''=\vert r-r'\vert+1\\ r+r'+r''\equiv 1\,(\mathrm{mod}\,2)}}^{\min(r+r'-1,2p-r-r'-1)} \cL_{r''} 
\end{equation}
for $1\leq r, r'\leq p-1$.
\end{enumerate}
\end{thm}

For $p>1$, $KL_k(\fsl_2)$ is not exactly a tensor subcategory of $KL^k(\fsl_2)$ since its unit object $\cL_{1}$ is different from the unit object $\cV_{1}$ of $KL^k(\fsl_2)$. However, we will show that the inclusion $\iota: KL_k(\fsl_2)\rightarrow KL^k(\fsl_2)$ is a lax monoidal functor, and that the difference in unit objects is the only reason that $\iota$ is not a strong monoidal functor. First, we need a lemma:
\begin{lem}\label{lem:KL_k_tensor_ideal}
Suppose that $\cY$ is a surjective $V^k(\fsl_2)$-module intertwining operator of type $\binom{W_3}{W_1\,W_2}$ where $W_1$ is an object of $KL_k(\fsl_2)$ and $W_2$, $W_3$ are objects of $KL^k(\fsl_2)$. Then $W_3$ is an object of $KL_k(\fsl_2)$.
\end{lem}
\begin{proof}
By assumption, $W_3$ is spanned by coefficients of powers of $x$ and $\log x$ in $\cY(w_1,x)w_2=\sum_{h\in\CC}\sum_{k\in\ZZ_{\geq 0}} (w_1)_{h;k} w_2\,x^{-h-1}(\log x)^k$ as $w_1$ and $w_2$ run over $W_1$ and $W_2$, respectively. Thus to show that $W_3$ is an object of $KL_k(\fsl_2)$, that is, $W_3$ is an $L_k(\fsl_2)$-module, we need to show that $v_n(w_1)_{h;k} w_2=0$ for all $w_1\in W_1$, $w_2\in W_2$, $h\in\CC$, $k\in\ZZ_{\geq 0}$, $n\in\ZZ$, and $v$ in the maximal proper ideal of $V^k(\fsl_2)$. In fact, the easy intertwining operator generalization of \cite[Proposition 4.5.7]{LL} shows that $v_n(w_1)_{h;k} w_2$ is a linear combination of vectors $(v_m w_1)_{j;k} w_2$ for $m\in\ZZ$ and $j\in\CC$. Since $W_1$ is an $L_k(\fsl_2)$-module, each $v_m w_1=0$, and thus each $v_n(w_1)_{h;k}w_2=0$ as well.
\end{proof}

If $W_1$ is an object of $KL_k(\fsl_2)$ and $W_2$ is an object of $KL^k(\fsl_2)$, then the tensor product intertwining operator of type $\binom{W_1\tens W_2}{W_1\,W_2}$ is surjective, so the preceding lemma shows that $W_1\tens W_2$ is an object of $KL_k(\fsl_2)$. That is,
\begin{cor}\label{cor:KL_k_tens_ideal}
$KL_k(\fsl_2)$ is a tensor ideal of $KL^k(\fsl_2)$.
\end{cor}

We can now give the inclusion $\iota: KL_k(\fsl_2)\rightarrow KL^k(\fsl_2)$ the structure of a lax monoidal functor. First, we have the surjection 
\begin{equation*}
\varphi: \cV_{1}\longrightarrow\iota(\cL_{1})
\end{equation*}
between unit objects. Now suppose $W_1$ and $W_2$ are objects in $KL_k(\fsl_2)$. We temporarily use $\tens^k$ and $\tens_k$ to denote the tensor products in $KL^k(\fsl_2)$ and $KL_k(\fsl_2)$, respectively, and we use $\cY^k$ and $\cY_k$ to denote the tensor product intertwining operators of types $\binom{W_1\tens^k W_2}{W_1\,W_2}$ and $\binom{W_1\tens_k W_2}{W_1\,W_2}$, respectively. Then the universal property of $\tens^k$ induces a unique $V^k(\fsl_2)$-module homomorphism
\begin{equation*}
\Phi_{W_1,W_2}: \iota(W_1)\tens^k\iota(W_2)\longrightarrow\iota(W_1\tens_k W_2)
\end{equation*}
such that the diagram
\begin{equation*}
\xymatrix{
W_1\otimes W_2 \ar[d]_{\cY^k} \ar[r]^(.4){\cY_k} & \iota(W_1\tens_k W_2)[\log x]\lbrace x\rbrace\\
(\iota(W_1)\tens^k\iota(W_2))[\log x]\lbrace x\rbrace \ar[ru]_{\quad\Phi_{W_1,W_2}} & \\
}
\end{equation*}
commutes. 

Since $\iota(W_1)\tens^k\iota(W_2)$ is an object of $KL_k(\fsl_2)$ by Lemma \ref{lem:KL_k_tensor_ideal}, the universal property of $\tens_k$ also induces $\til{\Phi}_{W_1,W_2}:\iota(W_1\tens_k W_2)\rightarrow\iota(W_1)\tens^k\iota(W_2)$ such that $\til{\Phi}_{W_1,W_2}\circ\cY_k=\cY^k$. Since $\cY^k$ and $\cY_k$ are both surjective, it follows that $\til{\Phi}_{W_1,W_2}$ is the inverse of $\Phi_{W_1,W_2}$, that is, $\Phi_{W_1,W_2}$ is an isomorphism. This isomorphism is natural because if $f_1: W_1\rightarrow X_1$ and $f_2: W_2\rightarrow X_2$ are morphisms in $KL_k(\fsl_2)$, then
\begin{align*}
[\iota(f_1\tens_k f_2)\circ\Phi_{W_1,W_2}] & (\cY^k(w_1,x)w_2) =(f_1\tens_k f_2)(\cY_k(w_1,x)w_2)\nonumber\\
& =\cY_k(f_1(w_1),x)f_2(w_2) =\Phi_{X_1,X_2}(\cY^k(\iota(f_1)(w_1),x)\iota(f_2)(w_2))\nonumber\\
& =[\Phi_{X_1,X_2}\circ(\iota(f_1)\tens^k\iota(f_2))](\cY^k(w_1,x)w_2)
\end{align*}
for all $w_1\in W_1$ and $w_2\in W_2$. Thus the isomorphisms $\Phi_{W_1,W_2}$ define a natural isomorphism
\begin{equation*}
\Phi: \tens^k\circ(\iota\times\iota)\longrightarrow\iota\circ\tens_k.
\end{equation*}
We can now prove:
\begin{thm}\label{thm:inclusion_is_lax_monoidal}
The surjection $\varphi$ and the natural isomorphism $\Phi$ give the inclusion $\iota: KL_k(\fsl_2)\rightarrow KL^k(\fsl_2)$ the structure of a lax braided ribbon monoidal functor.
\end{thm}
\begin{proof}
It remains to show that $\varphi$ and $\Phi$ are suitably compatible with the unit, associativity, braiding, and ribbon twist isomorphisms of $KL_k(\fsl_2)$ and $KL^k(\fsl_2)$. For the unit isomorphisms, we need to show that the diagram
\begin{equation*}
\xymatrixcolsep{4pc}
\xymatrix{
\cV_{1}\tens^k\iota(W) \ar[r]^(.48){\varphi\tens^k\Id_{\iota(W)}} \ar[d]^{l_{\iota(W)}} & \iota(\cL_{1})\tens^k\iota(W) \ar[d]^{\Phi_{\cL_{1},W}} \\
\iota(W) & \iota(\cL_{1,1}\tens_k W) \ar[l]^(.55){\iota(l_W)} \\
}
\end{equation*}
commutes for all objects $W$ in $KL_k(\fsl_2)$, as well as a similar diagram for the right unit isomorphisms. Indeed, using \eqref{eqn:unit_isos}, we have
\begin{align*}
[\iota(l_W)\circ\Phi_{\cL_{1},W}\circ(\varphi\tens^k\Id_{\iota(W)})]& (\cY^k(v,x)w) = [\iota(l_W)\circ\Phi_{\cL_{1},W}](\cY^k(\varphi(v),x)w)\nonumber\\
& = l_W(\cY_k(\varphi(v),x)w) =Y_{W}(\varphi(v),x)w = Y_{\iota(W)}(v,x)w\nonumber\\
& = l_{\iota(W)}(\cY^k(v,x)w)
\end{align*}
for all $v\in\cV_{1}=V^k(\fsl_2)$ and $w\in W$, where the next to last equality holds because the maximal proper ideal of $V^k(\fsl_2)$, that is, the kernel of $\varphi$, acts trivially on $L_k(\fsl_2)$-modules. Compatibility of $\varphi$ and $\Phi$ with the right unit isomorphisms is proved in the same way.

For the associativity and braiding isomorphisms, the diagrams
\begin{equation*}
\xymatrixcolsep{6pc}
\xymatrix{
\iota(W_1)\tens^k(\iota(W_2)\tens^k\iota(W_3)) \ar[r]^{\cA_{\iota(W_1),\iota(W_2),\iota(W_3)}} \ar[d]^{\Id_{\iota(W_1)}\tens^k\Phi_{W_2,W_3}} & (\iota(W_1)\tens^k\iota(W_2))\tens^k\iota(W_3) \ar[d]^{\Phi_{W_1,W_2}\tens^k\Id_{\iota(W_3)}}\\
\iota(W_1)\tens^k\iota(W_2\tens_k W_3) \ar[d]^{\Phi_{W_1,W_2\tens_k W_3}} & \iota(W_1\tens_k W_2)\tens^k\iota(W_3) \ar[d]^{\Phi_{W_1\tens_k W_2,W_3}}\\
\iota(W_1\tens_k(W_2\tens_k W_3)) \ar[r]^{\iota(\cA_{W_1,W_2,W_3})} & \iota((W_1\tens_k W_2)\tens_k W_3)\\
}
\end{equation*}
and
\begin{equation*}
\xymatrixcolsep{5pc}
\xymatrix{
\iota(W_1)\tens^k\iota(W_2) \ar[r]^{\cR_{\iota(W_1),\iota(W_2)}} \ar[d]^{\Phi_{W_1,W_2}} & \iota(W_2)\tens^k\iota(W_1) \ar[d]^{\Phi_{W_2,W_1}}\\
\iota(W_1\tens_k W_2) \ar[r]^{\iota(\cR_{W_1,W_2})} & \iota(W_2\tens_k W_1)\\
}
\end{equation*}
commute for objects $W_1$, $W_2$, $W_3$ in $KL_k(\mathfrak{sl}_2)$ simply because, by \eqref{eqn:braiding_iso} and \eqref{eqn:assoc_iso}, the associativity and braiding isomorphisms in $KL^k(\mathfrak{sl}_2)$ are defined in terms of the tensor product intertwining operators $\cY^k$ in the same way that the associativity and braiding isomorphisms in $KL_k(\mathfrak{sl}_2)$ are defined in terms of the $\cY_k$, and because the natural isomorphism $\Phi$ maps $\cY^k$ to $\cY_k$. Finally, the functor $\iota$ preserves ribbon twists in the sense that $\theta_{\iota(W)}=\iota(\theta_W)$ for all $W$ in $KL_k(\mathfrak{sl}_2)$ because $\theta=e^{2\pi i L(0)}$ on both $KL^k(\mathfrak{sl}_2)$ and $KL_k(\mathfrak{sl}_2)$.
\end{proof}

By the preceding theorem (or more precisely, because $\Phi$ is a natural isomorphism), we no longer need to distinguish between $\tens^k$ and $\tens_k$, so we will just use the notation $\tens$ for the tensor product in $KL^k(\fsl_2)$. Using Lemma \ref{lem:KL_k_tensor_ideal} and Theorem \ref{thm:inclusion_is_lax_monoidal}, we can derive the tensor products of $\cL_{r}$ for $1\leq r\leq p-1$ with all other simple modules $\cL_{r'}$ in $KL^k(\fsl_2)$. First, we have:
\begin{lem}\label{lem:L11_times_W}
For any $V^k(\fsl_2)$-module $W$ in $KL^k(\fsl_2)$, $\cL_{1}\tens W$ is a homomorphic image of $W$ which is an $L_k(\fsl_2)$-module. In particular, assuming $p>1$, $\cL_{1}\tens\cL_{r}=0$ if $r\geq p$.
\end{lem}
\begin{proof}
The $V^k(\fsl_2)$-module homomorphism
\begin{equation*}
W\xrightarrow{l_W^{-1}}\cV_{1}\tens W\xrightarrow{\varphi\tens\Id_W}\cL_{1}\tens W
\end{equation*}
is surjective because $\bullet\tens W$ is right exact (see \cite[Proposition 4.26]{HLZ3}). Then the conclusions of the lemma follow from Lemma \ref{lem:KL_k_tensor_ideal}.
\end{proof}

\begin{thm}\label{thm:L1s_times_Lrs}
For $1\leq r,r'\leq p-1$, we have
\begin{equation*}
\cL_{r}\tens\cL_{r'}\cong
\bigoplus_{\substack{r''=\vert r-r'\vert+1\\ r+r'+r''\equiv 1\,(\mathrm{mod}\,2)}}^{\min(r+r'-1,2p-r-r'-1)} \cL_{r''} 
\end{equation*}
in $KL^k(\fsl_2)$, while $\cL_r\tens \cL_{r'}=0$ for $1\leq r\leq p-1$ and $r'\geq p$.
\end{thm}
\begin{proof}
The first conclusion follows from \eqref{eqn:KL_k_adm_level_fus_rules} and Theorem \ref{thm:inclusion_is_lax_monoidal}, while for $r<p\leq r'$,
\begin{equation*}
\cL_{r}\tens\cL_{r'}\cong (\cL_{r}\tens\cL_{1})\tens\cL_{r'}\cong\cL_{r}\tens(\cL_{1}\tens\cL_{r'})\cong\cL_{r}\tens 0\cong 0
\end{equation*}
by \eqref{eqn:KL_k_adm_level_fus_rules} and Lemma \ref{lem:L11_times_W}.
\end{proof}

\subsection{Intertwining operators among generalized Verma modules}\label{subsec:intw_op}

Here we discuss generalities on intertwining operators among $V^k(\mathfrak{sl}_2)$-modules, and especially generalized Verma modules, in $KL^k(\mathfrak{sl}_2)$. Let $W$ be an object of $KL^k(\mathfrak{sl}_2)$; then $W$ has a conformal weight grading of the form $W=\bigoplus_{i=1}^I\bigoplus_{n=0}^\infty W_{[\sigma_i+n]}$ where the complex numbers $\sigma_i$ for $1\leq i\leq I$ are non-congruent modulo $\ZZ$. We then fix a $\ZZ_{\geq 0}$-grading $W=\bigoplus_{n=0}^\infty W(n)$ such that $W(n)=\bigoplus_{i=1}^I W_{[\sigma_i+n]}$. Each $W(n)$ is a finite-dimensional $\mathfrak{sl}_2$-module.

Now take $W_1, W_2$, $W_3$ to be objects of $KL^k(\mathfrak{sl}_2)$, and suppose $\cY$ is a $V^k(\mathfrak{sl}_2)$-module  intertwining operator of type $\binom{W_3}{W_1\,W_2}$. The commutator formula \eqref{eqn:intw_op_comm} implies that $\cY$ induces an $\mathfrak{sl}_2$-module homomorphism
\[
\pi(\cY): W_1(0) \otimes W_2(0) \longrightarrow W_3(0)
\]
defined by
\[
\pi(\cY)(m_1 \otimes m_2) = \pi_0(\cY(m_1,1)m_2)
\]
for $m_1 \in W_1(0)$ and $m_2 \in W_2(0)$, where $\pi_0$ denotes projection onto $W_3(0)$ with respect to the $\ZZ_{\geq 0}$-grading of $W_3$. We say $\cY$ is surjective if $W_3$ is spanned by the coefficients of powers of $x$ and $\log x$ in $\cY(w_1,x)w_2$ as $w_1$ and $w_2$ range over $W_1$ and $W_2$, respectively. By \cite[Proposition~3.1.1]{CMY3} (which is based on \cite[Proposition 24]{TW}) and its proof, we have:
\begin{prop}\label{prop:surjective}
 If $W_1$ and $W_2$ are generated as $V^k(\mathfrak{sl}_2)$-modules by $W_1(0)$ and $W_2(0)$, respectively, and $\cY$ is a surjective intertwining operator, then $\pi(\cY)$ is a surjective $\mathfrak{sl}_2$-module homomorphism.
 \end{prop}
 
 Since the tensor product intertwining operator $\cY_\tens$ of type $\binom{W_1\tens W_2}{W_1\,W_2}$ is surjective \cite[Proposition 4.23]{HLZ3}, we get:
 \begin{cor}\label{cor:tens_prod_top_level}
 If $W_1$ and $W_2$ are generalized Verma modules or simple modules, then $(W_1\tens W_2)(0)$ is an $\mathfrak{sl}_2$-module quotient of $W_1(0)\otimes W_2(0)$.
 \end{cor}
 
 The preceding corollary gives upper bounds on the tensor product module $W_1\tens W_2$. For lower bounds, we need to construct intertwining operators. In fact, intertwining operators among generalized Verma modules have been constructed in \cite{MY1, McR}. We will use \cite[Theorems 3.9 and 5.1]{McR} in particular to determine how the generalized Verma module $\cV_{2}$ tensors with other generalized Verma modules in $KL^k(\mathfrak{sl}_2)$ when $k=-2+p/q$ is an admissible level. To make the notation in the following theorem more uniform, we set $\cV_0=0$ (that is, $\cV_0$ is the generalized Verma module induced from the $0$-dimensional $\fsl_2$-module):
 \begin{thm}\label{thm:V12_times_Vrs}
 For $r\in\ZZ_{\geq 1}$ such that $p\nmid r$,
 \begin{equation*}
 \cV_2\tens\cV_{r}\cong\cV_{r-1}\oplus\cV_{r+1},
 \end{equation*}
 while for $n\in\ZZ_{\geq 1}$, there is an exact sequence
 \begin{equation*}
 0\longrightarrow \cV_{np-1}/\cJ_n\longrightarrow \cV_{2}\tens\cV_{np}\longrightarrow\cV_{np+1}\longrightarrow 0
 \end{equation*}
 for some submodule $\cJ_n\subseteq\cV_{np-1}$.
 \end{thm}
 
 \begin{proof}
 The case $\cV_{2}\tens\cV_{1}\cong\cV_0\oplus\cV_2\cong\cV_2$ is immediate because $\cV_{1}$ is the unit object of $KL^k(\mathfrak{sl}_2)$. For all other cases, we recall the notation $M_r$, $r\in\ZZ_{\geq 1}$, for the $r$-dimensional irreducible $\mathfrak{sl}_2$-module. Thus by Corollary \ref{cor:tens_prod_top_level}, the degree-$0$ space $(\cV_{2}\tens\cV_{r})(0)$ is an $\mathfrak{sl}_2$-module quotient of
 \begin{equation*}
 M_2\otimes M_{r} \cong M_{r-1}\oplus M_{r+1}.
 \end{equation*}
In particular, the possible lowest conformal weight(s) of $\cV_{2}\tens\cV_{r}$ are
\begin{equation*}
h_{r\pm 1} =\frac{q}{4p}\left((r\pm 1)^2-1\right) = \frac{q}{4p}\left(r^2\pm 2r\right).
\end{equation*}
recalling \eqref{eqn:h_lambda}. Note that $h_{r+1}-h_{r-1} =\frac{qr}{p}$,
which is an integer if and only if $p\mid r$.

First suppose $p\nmid r$, so that the two possible lowest conformal weights of $\cV_{2}\tens\cV_{r}$ are non-congruent modulo $\ZZ$. Then the universal property of generalized Verma modules implies that the submodule $\langle(\cV_{2}\tens\cV_{r})(0)\rangle$ generated by the degree-$0$ space is a quotient of $\cV_{r-1}\oplus\cV_{r+1}$. Moreover, there is a surjective intertwining operator
\begin{equation*}
\cV_{2}\otimes\cV_{r}\longrightarrow (\cV_{2}\tens\cV_{r})/\langle(\cV_{2}\tens\cV_{r})(0)\rangle\lbrace x\rbrace
\end{equation*}
obtained by composing the tensor product intertwining operator with the quotient map. Thus Proposition \ref{prop:surjective} implies that the lowest conformal weight(s) of $(\cV_{2}\tens\cV_{r})/\langle(\cV_{2}\tens\cV_{r})(0)\rangle$ must be $h_{r\pm 1}$, but in fact these conformal weight spaces in the quotient are $0$. The only possibility is that the quotient module is $0$, that is, $\cV_{2}\tens\cV_{r}$ is generated by its degree-$0$ subspace. Consequently, there is a surjective map
\begin{equation*}
\cV_{r-1}\oplus\cV_{r+1}\longrightarrow\cV_{2}\tens\cV_{r}.
\end{equation*}

Now by \cite[Theorem 5.1]{McR}, there are non-zero intertwining operators
\begin{equation*}
\cY_\pm: \cV_{2}\otimes\cV_{r} \longrightarrow\cV_{r\pm 1}\lbrace x\rbrace
\end{equation*}
in the case $p\nmid r$. Moreover, from \cite[Theorem 3.9]{McR}, the $\mathfrak{sl}_2$-module homomorphisms
\begin{equation*}
\pi(\cY_\pm): M_2\otimes M_{r}\longrightarrow M_{r\pm 1}
\end{equation*}
are the unique (up to scale) surjections. Thus the images of $\cY_\pm$ contain the generating lowest conformal weight subspaces $M_{r\pm 1}\subseteq \cV_{r\pm 1}$, which means that $\cY_\pm$ are surjective. Then by the universal property of vertex algebraic tensor products, there are surjective maps
\begin{equation*}
\cV_{2}\tens\cV_{r}\longrightarrow\cV_{r\pm1}.
\end{equation*}
It follows that $\cV_{2}\tens\cV_{r}\cong\cV_{r-1}\oplus\cV_{r+1}$ when $p\nmid r$.

For $r=np$ with $n\in\ZZ_{\geq 1}$, we have $h_{np+1}-h_{np-1}=qn\in\ZZ_{\geq 1}$, so the lowest possible conformal weight of $\cV_{2}\tens\cV_{np}$ is $h_{np-1}$, and the corresponding conformal weight space generates a $V^k(\mathfrak{sl}_2)$-submodule of the form $\cV_{np-1}/\cJ_n$ for some $\cJ_n\subseteq\cV_{np-1}$. By Proposition \ref{prop:surjective}, the quotient $(\cV_{2}\tens\cV_{np})/(\cV_{np-1}/\cJ_n)$ is, if non-zero, generated by a lowest conformal weight space $M_{np+1}$ of conformal weight $h_{np+1}$. Thus there is an exact sequence
\begin{equation*}
0\longrightarrow\cV_{np-1}/\cJ_n\longrightarrow\cV_{2}\tens\cV_{np}\longrightarrow\cV_{np+1}/\til{\cJ}_n\longrightarrow 0
\end{equation*}
for some submodule $\til{\mathcal{J}}_n\subseteq\cV_{np+1}$. Again by \cite[Theorems 3.9 and 5.1]{McR}, there is a surjective intertwining operator of type $\binom{\cV_{np+1}}{\cV_{2}\,\cV_{np}}$ (although we are no longer guaranteed an intertwining operator into $\cV_{np-1}$). Thus there is a surjection $f: \cV_{2}\tens\cV_{np}\rightarrow\cV_{np+1}$. Since the submodule $\cV_{np-1}/\cJ_n\subseteq\cV_{2}\tens\cV_{np}$ is generated by its lowest conformal weight space, which has a conformal weight lower than $h_{np+1}$, it follows that $\cV_{np-1}/\cJ_n$ is contained in the kernel of $f$, and thus $f$ factors through $\cV_{np+1}/\til{\cJ}_n$. Consequently $\til{\cJ}_n=0$, proving the final case of the theorem.
 \end{proof}
 
 \begin{rem}
 After showing that $\cV_{2}$ is rigid in the next section, we will show that the submodule $\cJ_n\subseteq\cV_{np-1}$ is actually $0$.
 \end{rem}

\section{Rigidity of \texorpdfstring{$\cV_{2}$}{V2}}\label{sec:rigidity}

We continue to fix a level $k=-2+p/q$ for relatively prime $p\in\ZZ_{\geq 2}$ and $q\in\ZZ_{\geq 1}$, and we set $\kappa=k+2=p/q$ to be the shifted level. In this section, we will prove that the generalized Verma module $\cV_{2}$ is rigid in the tensor category $KL^k(\mathfrak{sl}_2)$ by using Knizhnik-Zamolodchikov (KZ) equations to derive explicit expressions for compositions of intertwining operators involving $\cV_{2}$. This is similar to rigidity proofs in \cite{TW, CMY-singlet, CMY3, MY2}.

\begin{thm}\label{thm:V12_rigid}
The generalized Verma module $\cV_{2}$ is rigid and self-dual in $KL^k(\mathfrak{sl}_2)$ for any level $k = -2+p/q$ with $p\in\ZZ_{\geq 2}$ and $q\in\ZZ_{\geq 1}$ relatively prime, with evaluation $e:\cV_{2}\tens\cV_{2}\rightarrow\cV_{1}$ and coevaluation $i:\cV_{1}\rightarrow\cV_{2}\tens\cV_{2}$ such that 
\begin{equation*}
e\circ i = (-e^{\pi i q/p}-e^{-\pi i q/p})\cdot\Id_{\cV_{1}}.
\end{equation*}
\end{thm}

The $p=1$ version of this theorem has already been proved in \cite{CY}, so we do not consider this case further here. The scalar $-e^{\pi iq/p}-e^{-\pi iq/p}$ coming from $e\circ i$ is called the \textit{intrinsic dimension} of $\cV_{2}$. It does not depend on the choice of evaluation and coevaluation, since rigidity will imply that there are vector space isomorphisms
\begin{equation*}
\hom(\cV_{2}\tens\cV_{2},\cV_{1})\cong\mathrm{End}(\cV_{2})\cong\hom(\cV_{1},\cV_{2}\tens\cV_{2}).
\end{equation*}
Thus because $\dim \mathrm{End}(\cV_{2})=1$, all possible evaluation-coevaluation pairs for $\cV_{2}$ are given by $(c\cdot e, c^{-1}\cdot i)$ for $c\in\CC^\times$, and all such possibilities yield the same intrinsic dimension.

\subsection{Knizhnik-Zamolodchikov equations}

Let $W$ be some $V^k(\mathfrak{sl}_2)$-module, $\cY_1$ an intertwining operator of type $\binom{\cV_{2}}{\cV_{2}\,W}$, and $\cY_2$ an intertwining operator of type $\binom{W}{\cV_{2}\,\cV_{2}}$. The $2$-dimensional simple $\mathfrak{sl}_2$-module $M_2$ is the lowest conformal weight space of both the generalized Verma module $\cV_{2}$ and its contragredient $\cV_{2}'$, so we can define
\begin{align*}
\phi(v_0,v_1,v_2,v_3; z) :=\langle v_0, \cY_1(v_1,1)\cY_2(v_2,z)v_3\rangle
\end{align*}
for $v_0\in M_2\subseteq\cV_{2}'$ and $v_1,v_2,v_3\in M_2\subseteq\cV_{2}$. We can view $\phi(v_0,v_1,v_2,v_3;z)$ as either a formal series in $z$ (with $1$ substituted into $\cY_1$ using the branch of logarithm $\log 1=0$), as a multivalued analytic function on the punctured disk $0<\vert z\vert <1$, or as a single-valued analytic function on the simply-connected domain
\begin{equation*}
 U=\lbrace z\in\CC\,\vert\,\vert z\vert<1\rbrace\setminus(-1,0].
\end{equation*}
In this last case, we fix a single-valued branch by setting $\log z=\log\vert z\vert+i\arg z$, where $-\pi<\arg z<\pi$.

It is well known that $\phi(v_0,v_1,v_2,v_3; z)$ satisfies a Knizhnik-Zamolodchikov (KZ) differential equation \cite{KZ}, which is derived using the $L(-1)$-derivative property for intertwining operators and the expression \eqref{eqn:L-1} for $L(-1)$. Specifically, taking equation (2.16) in \cite[Corollary 2.2]{HL}, substituting $x_1\mapsto 1$, $x_2\mapsto z$, and using the Casimir operator $\Omega = ef+\frac{1}{2}h^2+fe$ for $\mathfrak{sl}_2$, we have:
\begin{align}
 \kappa\dfrac{d}{dz} & \phi(v_0, v_1, v_2, v_3;z) \nonumber\\
 & = z^{-1}\phi(v_0, v_1, f\cdot v_2, e\cdot v_3;z) -(1-z)^{-1}\phi(v_0, e\cdot v_1, f\cdot v_2, v_3;z)  \nonumber\\
 & \quad + \frac{1}{2}z^{-1}\phi(v_0, v_1, h\cdot v_2, h\cdot v_3;z)\nonumber- \frac{1}{2}(1-z)^{-1}\phi(v_0, h\cdot v_1, h\cdot v_2, v_3; z) \\
\label{productz2} & \quad +z^{-1}\phi(v_0, v_1, e\cdot v_2, f\cdot v_3; z)-(1-z)^{-1}\phi(v_0, f\cdot v_1, e\cdot v_2, v_3; z)  .
\end{align}

To solve the KZ equations, we need relations among the different solutions for different $v_0,v_1,v_2,v_3\in M_2$. Such relations follow from the fact that for any fixed $z\in U$, the map
\begin{align*}
M_2\otimes M_2\otimes M_2\otimes M_2 & \rightarrow \CC\\
v_0\otimes v_1\otimes v_2\otimes v_3 & \mapsto \phi(v_0,v_1,v_2,v_3;z)
\end{align*}
is an $\mathfrak{sl}_2$-module homomorphism (this follows from the $n=0$ cases of the contragredient relation \eqref{eqn:sl2_contra_structure} and the commutator formula \eqref{eqn:intw_op_comm}). In particular, setting $v\in M_2\subseteq\cV_{2}$ to be a highest-weight vector, we have
\begin{equation*}
\phi(v_0,v,v,v;z) = 0
\end{equation*}
for all $v_0\in M_2$. Using this relation, we then get
\begin{equation}\label{reduction1s}
\phi(v_0,a\cdot v,v,v)+\phi(v_0, v,a\cdot v,v)+\phi(v_0, v,v,a\cdot v)=0
\end{equation}
for any $v_0\in M_2$, $a\in\mathfrak{sl}_2$.

Now for any $v_0\in M_2$, we derive second-order differential equations for
\begin{align*}
\phi_1(v_0; z):=\phi(v_0,f\cdot v,v,v),\qquad\phi_2(v_0;z):=\phi(v_0,v,f\cdot v,v).
\end{align*}
 We begin with two cases of the KZ equation \eqref{productz2}, in the second case also using \eqref{reduction1s}:
\begin{align}
 \kappa\dfrac{d}{dz}\phi_1(v_0;z) & = \frac{1}{2}\left[z^{-1}+(1-z)^{-1}\right]\phi_1(v_0;z)-(1-z)^{-1}\phi_2(v_0;z),
\label{eqn:solve1}\\
\kappa\dfrac{d}{dz}\phi_2(v_0;z) & =- \left[z^{-1}+(1-z)^{-1}+\right]\phi_1(v_0; z)-\frac{1}{2}\left[3z^{-1}-(1-z)^{-1}\right]\phi_2(v_0;z). \label{eqn:solve2}
\end{align}
We solve \eqref{eqn:solve1} for $\phi_2(v_0;z)$ in terms of $\phi_1(v_0;z)$ and its derivative and then plug into \eqref{eqn:solve2} to obtain a second-order equation for $\phi_1(v_0;z)$; we derive a second-order differential equation for $\phi_2(v_0;z)$ similarly:
\begin{thm}\label{thm:diff_eq}
For any $v_0\in M_2\subseteq \cV_{2}'$, the analytic function $\phi_1(v_0;z)$ is a solution on the region $U$ to the differential equation
\begin{align}\label{eqnfor1}
\kappa^2 z(1-z)\phi_1''(z) - \kappa[(\kappa+2)z-1]\phi_1'(z) 
+ \left[\frac{\kappa}{2}z^{-1} - \frac{3}{4}z^{-1}(1-z)^{-1}\right]\phi_1(z) = 0 ,
\end{align}
and $\phi_2(v_0; z)$ is a solution on $U$ to the differential equation
\begin{align}\label{eqnfor2}
\kappa^2 z(1-z)\phi_2''(z) + \kappa (\kappa+1)(1-2z)\phi_2'(z) - \left[2\kappa + \frac{3}{4}z^{-1}(1-z)^{-1}\right]\phi_2(z) = 0.
\end{align}
\end{thm}

Recall from \eqref{eqn:h_lambda} that the lowest conformal weight of $\cV_{2}$ is $h_2=\frac{3}{4\kappa}$. If we set
\[
f_1(z) = z^{2h_2-1}(1-z)^{2h_2}\phi_1(z),\qquad f_2(z) = z^{2h_2}(1-z)^{2h_2}\phi_2(z)
\]
where $\phi_1(z)$ is a solution to \eqref{eqnfor1} and $\phi_2(z)$ is a solution to \eqref{eqnfor2}, then $f_1(z)$ satisfies the hypergeometric equation
\begin{equation}\label{eqnfor11}
\kappa^2z(1-z)f_1''(z) + \kappa[\kappa (2-3z) - 2(1-2z)]f_1'(z) - (\kappa-1)(\kappa-3)f_1(z) = 0,
\end{equation}
and  $f_2(z)$ satisfies the hypergeometric equation
\begin{equation}\label{eqnfor21}
\kappa^2 z(1-z)f_2''(z) + \kappa(\kappa -2)(1-2z)f_2'(z) + (\kappa-3)f_2(z) = 0.
\end{equation}
The solutions to \eqref{eqnfor11} and \eqref{eqnfor21} can be found, for example, in \cite[Chapter~15]{AS} or \cite[Section 15.10]{DLMF}, and the solutions to \eqref{eqnfor11} and \eqref{eqnfor21} can then be used to write down solutions for the original equations \eqref{eqnfor1} and \eqref{eqnfor2}.

Before presenting the solutions of \eqref{eqnfor1} and \eqref{eqnfor2} in the next subsections, we briefly discuss iterates of intertwining operators involving $\cV_{2}$. Let $M$ be a $V^k(\mathfrak{sl}_2)$-module, $\cY^1$ an intertwining operator of type $\binom{\cV_{2}}{M\,\cV_{2}}$, and $\cY^2$ an intertwining operator of type $\binom{M}{\cV_{2}\,\cV_{2}}$. For $v_0\in M_2\subseteq\cV_{2}'$ and $v_1,v_2,v_3\in M_2\subseteq\cV_{2}$, we define
\begin{align*}
 \psi(v_0,v_1,v_2,v_3; z):=\langle v_0,\cY^1(\cY^2(v_1,1-z)v_2,z)v_3\rangle.
\end{align*}
Using the $L(0)$-conjugation property \cite[Proposition 3.36(b)]{HLZ2}, we view $\psi(v_0,v_1,v_2,v_3; z)$ as a series in powers of $\frac{1-z}{z}$:
\begin{align}\label{eqn:iterate_series_exp}
 \psi(v_0,v_1,v_2,v_3; z) & =z^{-2h_2}\left\langle v_0,\cY^1\left(\cY^2\left(v_1,\frac{1-z}{z}\right)v_2,1\right)v_3\right\rangle\nonumber\\
 &=\left(1+\frac{1-z}{z}\right)^{2h_2}\left\langle v_0,\cY^1\left(\cY^2\left(v_1,\frac{1-z}{z}\right)v_2,1\right)v_3\right\rangle.
\end{align}
If we substitute $\frac{1-z}{z}$ into $\cY^2$ using the branch of logarithm $\log(\frac{1-z}{z})=\log \vert\frac{1-z}{z}\vert+i\arg(\frac{1-z}{z})$ with $-\pi<\arg(\frac{1-z}{z})<\pi$, then $\psi(v_0,v_1,v_2,v_3;z)$ defines a single-valued analytic function on the simply-connected domain
\begin{equation*}
 \widetilde{U}=\lbrace z\in\CC\,\vert\,\vert z\vert>\vert 1-z\vert>0\rbrace\setminus[1,\infty)=\left\lbrace z\in\CC\mid\mathrm{Re}\,z>1/2\right\rbrace\setminus[1,\infty).
\end{equation*}
Associativity of intertwining operators \cite{HLZ6} shows that $\psi(v_0,v_1,v_2,v_3; z)$ is the analytic continuation to $\til{U}$ of a corresponding product of intertwining operators $\phi(v_0,v_1,v_2,v_3;z)$ defined on $U$. So the functions $\psi(v_0,v_1,v_2,v_3; z)$ satisfy the same differential equations as $\phi(v_0,v_1,v_2,v_3; z)$.

\subsection{The case \texorpdfstring{$p \geq 3$}{p>=3}}

In this subsection, we prove Theorem \ref{thm:V12_rigid} for levels $k=-2+p/q$ such that $p\geq 3$. In these cases, a basis of solutions for the differential equation \eqref{eqnfor1} on the region $U$ is:
\begin{align}
&\phi^{(1)}_1(z) = z^{-2h_2+1}(1-z)^{-2h_2} {}_2F_1\left(1-\frac{1}{\kappa}, 1-\frac{3}{\kappa}; 2-\frac{2}{\kappa}; z\right),\nonumber\\
&\phi^{(2)}_1(z) = z^{-2h_2+h_3}(1-z)^{-2h_2}{}_2F_1\left(\frac{1}{\kappa}, -\frac{1}{\kappa}; \frac{2}{\kappa}; z\right),
\end{align}
where $h_3=\frac{2}{\kappa}$ from \eqref{eqn:h_lambda} is the lowest conformal weight of $\cV_{3}$, and a basis of solutions for \eqref{eqnfor2} on $U$ is:
\begin{align}
&\phi^{(1)}_2(z) = z^{-2h_2}(1-z)^{-2h_2} {}_2F_1\left(1-\frac{3}{\kappa}, -\frac{1}{\kappa}; 1-\frac{2}{\kappa}; z\right),\nonumber\\
&\phi^{(2)}_2(z) = z^{-2h_2+h_3}(1-z)^{-2h_2} {}_2F_1\left(1-\frac{1}{\kappa}, \frac{1}{\kappa}; 1+\frac{2}{\kappa}; z\right).
\end{align}
We will use these explicit solutions to prove that $\cV_{2}$ is rigid, similar to the proof of \cite[Theorem~4.2.3]{CMY3}, as well as that of \cite[Theorem~4.3.7]{CMY-singlet}.

  We first fix candidates for the evaluation and coevaluation morphisms. For the evaluation, let $\langle\cdot,\cdot\rangle$ be the nondegenerate $\mathfrak{sl}_2$-invariant bilinear form on $M_2\subseteq\cV_{2}$
such that 
\begin{equation*}
\langle v,f\cdot v\rangle=-\langle f\cdot v, v\rangle =1,
\end{equation*}  
  where $v$ is a highest-weight vector. As in the proof of Theorem \ref{thm:V12_times_Vrs}, it follows from \cite[Theorems 3.9 and 5.1]{McR} that there is a (unique) intertwining operator $\cE$ of type $\binom{\cV_{1}}{\cV_{2}\,\cV_{2}}$ such that the $\mathfrak{sl}_2$-homomorphism $\pi(\cE): M_2\otimes M_2\rightarrow M_1$ is given by $\langle\cdot,\cdot\rangle$. In particular, for lowest-conformal-weight vectors $w, w' \in M_2\subseteq\cV_{2}$,
\begin{equation*}
 \cE(w',x)w \in x^{-2h_2}\big( \langle w',w\rangle {\bf 1}+x\,\cV_{1}[[x]]\big).
\end{equation*}
We define the evaluation candidate $\varepsilon:\cV_{2}\boxtimes\cV_{2}\rightarrow \cV_{1}$ to be the unique homomorphism such that $\varepsilon\circ\cY_\boxtimes=\cE$, where $\cY_\boxtimes$ is the tensor product intertwining operator of type $\binom{\cV_{2}\boxtimes\cV_{2}}{\cV_{2}\,\cV_{2}}$.

For the coevaluation, we compose the $\mathfrak{sl}_2$-homomorphism
$M_1 \rightarrow M_2 \otimes M_2$ defined by
\begin{equation*}
 {\bf 1} \mapsto f\cdot v\otimes v - v\otimes  f\cdot v
\end{equation*}
with the $\mathfrak{sl}_2$-homomorphism
\begin{equation*}
 \pi(\cY_\boxtimes): M_2 \otimes M_2 \rightarrow(\cV_{2}\boxtimes\cV_{2})(0),
\end{equation*} 
and apply the universal property of induced $\widehat{\fsl}_2$-modules to get a homomorphism
\begin{equation*}
 i: \cV_{1} \rightarrow \cV_{2}\boxtimes\cV_{2}
\end{equation*}
such that
\begin{equation}\label{eqn:coevaluation_candidate}
 i({\bf 1}) =\pi_0\left(\cY_\boxtimes(f\cdot v,1)v - \cY_\boxtimes(v,1)f\cdot v\right).
\end{equation}
Equivalently, $i({\bf 1})$ is the coefficient of $x^{-2h_2}$ in $\cY_\boxtimes(f\cdot v,x)v-\cY_\boxtimes(v,x)f\cdot v$.

To prove $\cV_{2}$ is rigid, we need to show that the compositions
\begin{align}
& \mathfrak{R}:  \cV_{2}\xrightarrow{l^{-1}} \cV_{1}\tens\cV_{2}\xrightarrow{i\tens\Id} (\cV_{2}\tens\cV_{2})\tens\cV_{2}\xrightarrow{\cA^{-1}}\cV_{2}\tens(\cV_{2}\tens\cV_{2})\xrightarrow{\Id\tens\varepsilon}\cV_{2}\tens\cV_{1}\xrightarrow{r}\cV_{2},\label{eqn:rigidity1}\\
& \mathfrak{R'}: \cV_{2}\xrightarrow{r^{-1}}\cV_{2}\tens\cV_{1}\xrightarrow{\Id\tens i}  \cV_{2}\tens(\cV_{2}\tens\cV_{2})
\xrightarrow{\cA}  (\cV_{2}\tens\cV_{2})\tens\cV_{2} \xrightarrow{\varepsilon\tens\Id} \cV_{1}\tens\cV_{2}\xrightarrow{l}\cV_{2}\label{eqn:rigidity2}
\end{align}
both equal the same non-zero scalar multiple of the identity on $\cV_{2}$, since we can then rescale either $i$ or $\varepsilon$ so that $\mathfrak{R}=\mathfrak{R}'=\Id_{\cV_{2}}$. Since $\mathrm{End}(\cV_{2})=\CC\mathrm{Id}_{\cV_{2}}$ (even though $\cV_{2}$ is not generally simple), it is enough to show that
\begin{equation*}
\langle f\cdot v,\mathfrak{R}(v)\rangle =\langle f\cdot v,\mathfrak{R}'(v)\rangle =c
\end{equation*}
for some highest-weight vector $v\in M_2\subseteq\cV_2$ and $c\in\CC^\times$.

Considering $\mathfrak{R}$ first, the definitions \eqref{eqn:unit_isos} and \eqref{eqn:coevaluation_candidate} imply that
\begin{align*}
(i\boxtimes \Id)\circ l^{-1}(v) & = (i\tens\Id)(\cY_\tens(\vac,1)v)\nonumber\\
& = \mathrm{Res}_x\,x^{2h_2-1}\cY_\tens\left(\cY_\tens(f\cdot v,x)v-\cY_\tens(v,x)f\cdot v, 1\right)v\nonumber\\
& =\mathrm{Res}_x\,x^{2h_2-1} (1+x)^{2h_2}\left(\cY_\tens(\cY_\tens(f\cdot v,x)v, 1)v_{1,2} -\cY_\tens(\cY_\tens (v,x)f\cdot v,1)v\right),
\end{align*}
where we use $\cY_\tens$ to denote all tensor product intertwining operators. The last equality above holds because $x^{-2h_2}$ is the lowest power of $x$ in $\cY_\tens(w',x)w$ for $w',w\in M_2\subseteq\cV_{2}$, which in turn holds because $0$ is the lowest conformal weight of $\cV_{2}\tens\cV_{2}$ (indeed, $0$ is the minimum among all conformal weights of all modules in $KL^k(\mathfrak{sl}_2)$). We now substitute $x\mapsto\frac{1-z}{z}$ for $z\in U\cap\til{U}$ using the principal branch of logarithm, and then recalling \eqref{eqn:iterate_series_exp} as well as \eqref{eqn:unit_isos} and \eqref{eqn:assoc_iso}, we find that  $\langle f\cdot v,\mathfrak{R}(v)\rangle$ is the coefficient of $\left(\frac{1-z}{z}\right)^{-2h_2}$ in the expansion of the following analytic function as a series in $\frac{1-z}{z}$ on $U\cap\til{U}$:
\begin{align}\label{eqn:rig_calc}
& \left\langle f\cdot v,\overline{r\circ(\Id\boxtimes\varepsilon)\circ\cA^{-1}}\left(\cY_\boxtimes(\cY_\boxtimes(f\cdot v,1-z)v,z)v-\cY_\boxtimes(\cY_\boxtimes(v,1-z)f\cdot v,z)v\right)\right\rangle\nonumber\\
 &\qquad =\left\langle f\cdot v,\overline{r\circ(\Id\boxtimes\varepsilon)}\left(\cY_\boxtimes(f\cdot v,1)\cY_\boxtimes(v,z)v-\cY_\boxtimes(v,1)\cY_\boxtimes(f\cdot v,z)v\right)\right\rangle\nonumber\\
 &\qquad =\left\langle f\cdot v,\Omega(Y_{\cV_{2}})(f\cdot v,1)\cE(v,z)v\right\rangle - \left\langle f\cdot v,\Omega(Y_{\cV_{2}})(v,1)\cE(f\cdot v,z)v\right\rangle,
\end{align}
where
\begin{equation*}
\Omega(Y_{\cV_{2}})(w,x)u = e^{xL(-1)} Y_{\cV_2}(u,-x)w
\end{equation*}
for $u\in V^k(\fsl_2)$, $w\in \cV_2$.

By Theorem \ref{thm:diff_eq}, the second term of \eqref{eqn:rig_calc} is a solution to the differential equation \eqref{eqnfor2}. As a series in $z$, this solution has lowest-degree term
\begin{align*}
\langle f\cdot v,\Omega(Y_{\cV_{2}}) & (v,1)\vac\rangle\langle f\cdot v,v\rangle z^{-2h_2} = \langle f \cdot v,v\rangle^2 z^{-2h_2}=z^{-2h_2},
\end{align*}
so the second term of \eqref{eqn:rig_calc}  is the fundamental basis solution
\begin{equation*}
\phi^{(1)}_2(z) = z^{-2h_2}(1-z)^{-2h_2}{}_2F_1\left(1-\frac{3}{\kappa}, -\frac{1}{\kappa}; 1-\frac{2}{\kappa}; z\right).
\end{equation*}
For the first term, we need the coefficient of $x^{-2h_2+1}$ in $\cE(v,x)v$. Set
\[
\cE(w',x)w = \sum_{m =0}^\infty \cE_m(w'\otimes w)x^{-2h_2 + m}
\]
for $w,w' \in M_2\subseteq\cV_{2}$. By the $L(0)$-conjugation formula and the commutator formula \eqref{eqn:intw_op_comm}, $\cE_m: M_2 \otimes M_2 \rightarrow \cV_{1}(m)$ is an $\mathfrak{sl}_2$-module homomorphism. Since $v\otimes v$ is an $\mathfrak{sl}_2$-highest weight vector, $\cE_1(v \otimes v) = c\cdot e(-1){\bf 1}$ for some $c \in \CC$. Thus
\begin{equation*}
f(1)\cE_1(v\tens v) = c\cdot f(1)e(-1)\vac=c(-h(0)+k\langle f,e\rangle)\vac=ck\cdot\vac;
\end{equation*}
this together with the commutator formula
\[
f(1)\cE(v,x)v = x\cE(f\cdot v,x)v
\]
implies $ck=\langle f\cdot v,v\rangle=-1$, or $c=-\frac{1}{k}$. Consequently, as a series in $z$, the first term of \eqref{eqn:rig_calc} has lowest-degree term
\begin{align*}
&-\frac{1}{k} \langle f\cdot v,e^{L(-1)}Y_{\cV_{2}}(e(-1){\bf 1}, -1)f\cdot v\rangle z^{-2h_2+1} = \frac{1}{k}\langle f\cdot v, e\cdot(f\cdot v)\rangle z^{-2h_2+1} = -\frac{1}{k}z^{-2h_2+1}.
\end{align*}
Since the first term of \eqref{eqn:rig_calc} is a solution to the differential equation \eqref{eqnfor1} by Theorem \ref{thm:diff_eq}, this solution is
\begin{equation*}
-\frac{1}{k}\phi^{(1)}_1(z) = -\frac{1}{k}z^{-2h_2+1}(1-z)^{-2h_2}{}_2F_1\left(1-\frac{1}{\kappa}, 1-\frac{3}{\kappa}; 2-\frac{2}{\kappa}; z\right).
\end{equation*}

Thus to calculate $\langle f\cdot v,\mathfrak{R}(v)\rangle$, we need to expand $-\frac{1}{k}\phi^{(1)}_1(z)-\phi^{(1)}_2(z)$ as a series in $\frac{1-z}{z}$ and extract the coefficient of $(\frac{1-z}{z})^{-2h_2}$. Since
\begin{align*}
&-\frac{1}{k}\phi^{(1)}_1(z)-\varphi^{(1)}_2(z)\nonumber\\
&\quad =-\left(\frac{1-z}{z}\right)^{-2h_2} z^{-4h_2}\left(\frac{z}{k}{}_2F_1\left(1-\frac{1}{\kappa}, 1-\frac{3}{\kappa}; 2-\frac{2}{\kappa}; z\right) + {}_2F_1\left(1-\frac{3}{\kappa}, -\frac{1}{\kappa}; 1-\frac{2}{\kappa}; z\right)\right)
\end{align*}
and since $z^{-4h_2(+1)}=(1+\frac{1-z}{z})^{4h_2(+1)}$ are power series in $\frac{1-z}{z}$ with constant term $1$, it is equivalent to find the constant term in the expansion of
 $$-\frac{1}{k}{}_2F_1\left(1-\frac{1}{\kappa}, 1-\frac{3}{\kappa}; 2-\frac{2}{\kappa}; z\right)-{}_2F_1\left(1-\frac{3}{\kappa}, -\frac{1}{\kappa}; 1-\frac{2}{\kappa}; z\right)$$ 
 as a series in $\frac{1-z}{z}$ on the region $U\cap\til{U}$. Using \cite[Equations 15.10.21, 5.5.1, and 5.5.3]{DLMF}, this constant term is 
\begin{align*}
-\frac{1}{\kappa-2}\frac{\Gamma(2-\frac{2}{\kappa})\Gamma(\frac{2}{\kappa})}{\Gamma(1-\frac{1}{\kappa})\Gamma(1+\frac{1}{\kappa})} -\frac{\Gamma(1-\frac{2}{\kappa})\Gamma(\frac{2}{\kappa})}{\Gamma(\frac{1}{\kappa})\Gamma(1-\frac{1}{\kappa})} & = \bigg(-\frac{1}{\kappa-2}\frac{1-\frac{2}{\kappa}}{\frac{1}{\kappa}}-1\bigg)\frac{\Gamma(1-\frac{2}{\kappa})\Gamma(\frac{2}{\kappa})}{\Gamma(\frac{1}{\kappa})\Gamma(1-\frac{1}{\kappa})}\nonumber\\
& = -2\cdot\frac{\pi/\sin(\frac{2\pi}{\kappa})}{\pi/\sin(\frac{\pi}{\kappa})} =-\frac{1}{\cos(\frac{\pi}{\kappa})}.
\end{align*}
This calculation proves that $\langle f\cdot v,\mathfrak{R}(v)\rangle =-\left[\cos\left(\frac{\pi}{\kappa}\right)\right]^{-1}$.
Then since $\mathfrak{R}$ is a scalar multiple of $\Id_{\cV_{2}}$ and $\langle f\cdot v,v\rangle =-1$, it follows that
\begin{equation}\label{eqn:rig_comp_scalar}
\mathfrak{R}=\left[\cos\left(\frac{\pi}{\kappa}\right)\right]^{-1}\cdot\Id_{\cV_{2}}\neq 0
\end{equation}
when $\kappa=p/q$ with $p\geq 3$ and $q$ relatively prime to $p$.

To show that the second rigidity composition $\mathfrak{R}'$ is the same scalar multiple of $\Id_{\cV_{2}}$, we could perform a similar calculation. Alternatively, we can apply the braiding isomorphisms to the composition $\mathfrak{R}$, obtaining the following commutative diagram; the middle rectangle commutes thanks to the naturality of the braiding and the hexagon axioms:
\begin{equation*}
\xymatrix{
\cV_{2} \ar[d]^{l^{-1}} \ar[rd]^{r^{-1}} & &\\
\cV_{1}\tens\cV_{2} \ar[d]^{i\tens\Id} \ar[r]^{\cR} & \cV_{2}\tens\cV_{1} \ar[d]^{\Id\tens i} &   \\
(\cV_{2}\tens\cV_{2})\tens\cV_{2} \ar[d]^{\cA^{-1}} 
\ar[r]^{\cR} & \cV_{2}\tens(\cV_{2}\tens\cV_{2}) \ar[r]^{\Id\tens\cR} & \cV_{2}\tens(\cV_{2}\tens\cV_{2}) \ar[d]^{\cA} \\
\cV_{2}\tens(\cV_{2}\tens\cV_{2}) \ar[d]^{\Id\tens\varepsilon} \ar[r]^{\cR} & (\cV_{2}\tens\cV_{2})\tens\cV_{2} \ar[d]^{\varepsilon\tens\Id} \ar[r]^{\cR\tens\Id} & (\cV_{2}\tens\cV_{2})\tens\cV_{2}\\
\cV_{2}\tens\cV_{1} \ar[r]^{\cR} \ar[d]^{r} & \cV_{1}\tens\cV_{2} \ar[ld]^l &\\
\cV_{2} & &\\
}
\end{equation*}
Thus we will get $\mathfrak{R}'=\mathfrak{R}$ as required if there is a non-zero scalar $c$ such that
\begin{equation}\label{eqn:twist_with_braiding}
\cR\circ i = c\cdot i,\qquad \varepsilon\circ\cR=c\cdot\varepsilon.
\end{equation}
 Note that $\cR\circ i$ and $\varepsilon\circ\cR$ are indeed scalar multiples of $i$ and $\varepsilon$, respectively, because the $p\geq 3$ case of Theorem \ref{thm:V12_times_Vrs} implies that $\hom(\cV_{1},\cV_{2}\tens\cV_{2})$ and $\hom(\cV_{2}\tens\cV_{2},\cV_{1})$ are both one-dimensional. To compute the scalar for $i$, the definitions \eqref{eqn:braiding_iso} and \eqref{eqn:coevaluation_candidate} yield
 \begin{align}\label{eqn:c_for_i}
 (\cR\circ i)(\vac) & = (\cR\circ\pi_0)\left(\cY_\tens(f\cdot v,1)v-\cY_\tens(v,1)f\cdot v\right)\nonumber\\
 & =\pi_0\left(e^{L(-1)}\cY_\tens(v,e^{\pi i})f\cdot v-e^{L(-1)}\cY_\tens(f\cdot v,e^{\pi i})v\right)\nonumber\\
 & = -\pi_0\left( e^{\pi i( L(0)-2h_2)}(\cY_\tens(f\cdot v,1)v-\cY_\tens(v,1)f\cdot v)\right)\nonumber\\
 & =-e^{\pi i(L(0)-2h_2)} i(\vac)= -e^{-2\pi i h_2}\cdot i(\vac),
 \end{align}
 so we get $c=-e^{-2\pi i h_2}$. For $\varepsilon$, the definitions yield
 \begin{align*}
 (\varepsilon\circ\cR)(\cY_\tens(w',x)w) & = e^{xL(-1)}\cE(w,e^{\pi i} x)w'\nonumber\\
 & \in e^{xL(-1)}(e^{\pi i} x)^{-2h_2}\left(\langle w,w'\rangle\vac+x\cV_{1}[[x]]\right)\nonumber\\
 & = -e^{-2\pi i h_2} x^{-2h_2}\left(\langle w',w\rangle\vac +x\cV_{1}[[x]]\right)
\end{align*}
for $w,w'\in M_2\subseteq\cV_{2}$. By the uniqueness of $\cE$ up to scalar multiples, this implies
\begin{equation*}
e^{xL(-1)}\cE(w,e^{\pi i} x)w'=-e^{-2\pi i h_2}\cE(w',x)w
\end{equation*}
for all $w,w'\in\cV_{2}$, which again yields $-e^{-2\pi i h_2}$ for the value of $c$ in \eqref{eqn:twist_with_braiding}. Thus $\mathfrak{R}'=\mathfrak{R}$, and we have proved that $\cV_{2}$ is rigid (and self-dual) in the case $p\geq 3$.

To complete the proof of Theorem \ref{thm:V12_rigid} for $p\geq 3$, we still need to compute the intrinsic dimension of $\cV_{2}$. By \eqref{eqn:rig_comp_scalar}, one choice of coevaluation and evaluation for $\cV_{2}$ is $i$ and $e:=\cos(\frac{\pi}{\kappa})\cdot\varepsilon$. Then since $\mathrm{End}(\cV_{1})=\CC\cdot\Id_{\cV_{1}}$, the intrinsic dimension is simply the scalar $d$ such that $(e\circ i)(\vac)=d\cdot\vac$. Using the definitions, we compute
\begin{align*}
(e\circ i)(\vac) & = \cos\left(\frac{\pi}{\kappa}\right)\cdot(\varepsilon\circ\pi_0)\left(\cY_\tens(f\cdot v,1)v-\cY_\tens(v,1)f\cdot v\right)\nonumber\\
& = \frac{e^{\pi i q/p}+e^{-\pi iq/p}}{2}\cdot \pi_0\left(\cE(f\cdot v,1)v-\cE(v,1)f\cdot v\right)\nonumber\\
& =\frac{e^{\pi i q/p}+e^{-\pi iq/p}}{2}\cdot\left(\langle f\cdot v,v\rangle\vac -\langle v,f\cdot v\rangle\vac\right)= (-e^{\pi i p/q}-e^{-\pi i p/q})\cdot\vac,
\end{align*}
as desired.

\subsection{The case \texorpdfstring{$p = 2$}{p=2}}
Now we prove Theorem \ref{thm:V12_rigid} for levels $k=-2+ 2/q$ for $q\in\ZZ_{\geq 1}$ odd. In this case, the differential equations \eqref{eqnfor1} and \eqref{eqnfor2} admit logarithmic solutions. A basis of solutions for \eqref{eqnfor1} on $U$ is
\begin{align}\label{eqn: solution1}
&\phi^{(1)}_1(z) = z^{-2h_2+q}(1-z)^{-2h_2}{}_2F_1\left(\frac{q}{2}, -\frac{q}{2}; q; z\right),\nonumber\\
&\phi^{(2)}_1(z) = z^{-2h_2+1}(1-z)^{-2h_2}\left[z^{q-1}{}_2F_1\left(\frac{q}{2}, -\frac{q}{2}; q; z\right)\log z + G_1(z)\right],
\end{align}
and a basis of solutions for \eqref{eqnfor2} on $U$ is
\begin{align}\label{eqn: solution2}
&\phi^{(1)}_2(z) = z^{-2h_2+q}(1-z)^{-2h_2}{}_2F_1\left(\frac{q}{2}, 1-\frac{q}{2}; 1+q; z\right)\nonumber\\
&\phi^{(2)}_2(z) = z^{-2h_2}(1-z)^{-2h_2}\left[z^q {}_2F_1\left(\frac{q}{2}, 1-\frac{q}{2}; 1+q; z\right)\log z + G_2(z)\right],
\end{align}
where $G_1(z)$ and $G_2(z)$ are power series that converge in the region $U$.

When $p=2$, we still have a coevaluation candidate $i:\cV_{1}\rightarrow\cV_{2}\tens\cV_{2}$ defined by \eqref{eqn:coevaluation_candidate}, but we must define the evaluation candidate $\varepsilon$ differently because there is no surjective intertwining operator of type $\binom{\cV_{1}}{\cV_{2}\,\cV_{2}}$. However, as in the proof of Theorem \ref{thm:V12_times_Vrs}, by \cite[Theorems 3.9 and 5.1]{McR}, there is a unique (up to scale) intertwining operator $\cY$ of type $\binom{\cV_{3}}{\cV_{2}\, \cV_{2}}$ such that the $\mathfrak{sl}_2$-homomorphism $\pi(\cY): M_2\otimes M_2\rightarrow M_3$ is surjective. The image of $\cY$ contains the generating lowest conformal weight space $M_3\subseteq\cV_{3}$, and thus $\cY$ is surjective. We then get a non-zero (but non-surjective) intertwining operator $\cE$ of type $\binom{\cV_{1}}{\cV_{2}\,\cV_{2}}$ by composing $\cY$ with the surjection $\cV_{3}\twoheadrightarrow\cL_{3}$ and then with the inclusion $\cL_{3}\hookrightarrow \cV_{1}$ (recall \eqref{exactseq:vrs}). We define $\varepsilon: \cV_{2} \boxtimes \cV_{2}  \rightarrow \cV_{1}$ to be the unique homomorphism such that $\varepsilon \circ \cY_{\boxtimes} = \cE$.

Since the lowest conformal weight of $\cL_{3}\subseteq\cV_{1}$ is $h_3=q$, we have
\[
\cE(w', x)w \in x^{-2h_2 + q}(b(w'\otimes w) + x\cL_{3}[[x]])
\]
for $w,w'\in M_2\subseteq \cV_2$, where $b$ is an $\mathfrak{sl}_2$-module surjection from $M_2\otimes M_2$ to the lowest conformal weight space $M_3$ of  $\cL_{3}\subseteq\cV_{1}$. To be concrete, for a highest-weight vector $v\in M_2$, let $b(v\otimes v) = \til{v}$ be a highest weight vector in $M_3$. Then 
\begin{equation*}
b(f\cdot v\otimes v) = b(v\otimes f\cdot v) = \frac{1}{2}f\cdot\til{v},\qquad b(f\cdot v\otimes f\cdot v) = \frac{1}{2}f\cdot(f\cdot\til{v}).
\end{equation*}

 As in the $p\geq 3$ case, we need to show that the rigidity compositions $\mathfrak{R}$ and $\mathfrak{R}'$ given in \eqref{eqn:rigidity1} and \eqref{eqn:rigidity2} are the same non-zero scalar multiple of $\Id_{\cV_{2}}$, and for $\mathfrak{R}$ it is again enough to show that $\langle f\cdot v,\mathfrak{R}(v)\rangle$ is non-zero. Then by the same calculations as in the $p \geq 3$ case, we are reduced to showing that the coefficient of 
$\left(\frac{1-z}{z}\right)^{-2h_2}$ in the expansion of
\begin{align}\label{eqn:rig_calc1}
&\left\langle v',\Omega(Y_{\cV_{2}})(f\cdot v,1)\cE(v,z)v\right\rangle - \left\langle v',\Omega(Y_{\cV_{2}})(v,1)\cE(f\cdot v,z)v\right\rangle
\end{align}
as a series in $\frac{1-z}{z}$ and $\log\left(\frac{1-z}{z}\right)$ on $U\cap\til{U}$ is non-zero for some $v'\in M_2\subseteq\cV_{2}'$. We prove this by contradiction; thus assume that the coefficient of 
$\left(\frac{1-z}{z}\right)^{-2h_2}$ in \eqref{eqn:rig_calc1} is $0$ for all $v'\in M_2$.

Taking $v' = f\cdot v$, then as a series in $z$, the lowest-degree term of the first summand in \eqref{eqn:rig_calc1} is
\begin{align*}
 \langle f\cdot v,\Omega(Y_{\cV_{2}})(f\cdot v,1)\til{v} \rangle x^{-2h_1+q}.
\end{align*}
We claim that the coefficient
\begin{equation}\label{eqn:claim1}
 c : = \langle f\cdot v,\Omega(Y_{\cV_{2}})(f\cdot v_{2},1)\til{v}\rangle = 0.
\end{equation}
By Theorem \ref{thm:diff_eq}, the first summand of \eqref{eqn:rig_calc1} is a solution to the differential equation \eqref{eqnfor1}. By comparing the coefficients of the lowest degree term $x^{-2h_2+q}$, it is a multiple of the non-logarithmic fundamental basis solution in \eqref{eqn: solution1}:
\[
c \cdot z^{-2h_2+q}(1-z)^{-2h_2}{}_2F_1\left(\frac{q}{2}, -\frac{q}{2}; q; z\right).
\]
From \cite[Equation 15.8.11]{DLMF}, the coefficient of $\left(\frac{1-z}{z}\right)^{-2h_2}$ when this solution is expanded as a series in $\frac{1-z}{z}$ on $\til{U}$ is 
\begin{equation}\label{eqn:coeff1}
\frac{c(q-1)!}{\Gamma(\frac{3q}{2})\Gamma(\frac{q}{2})}.
\end{equation}
The second summand in \eqref{eqn:rig_calc1} has lowest-degree term
\begin{align*}
&\frac{1}{2}\left\langle f\cdot v,\Omega(Y_{\cV_{2}})(v,1)f\cdot\til{v}\right\rangle z^{-2h_2+q}= -\frac{1}{2}\left\langle f\cdot v,\Omega(Y_{\cV_{2}})(f\cdot v,1)\til{v}\right\rangle z^{-2h_2+q} = -\frac{1}{2}c z^{-2h_2+q},
\end{align*}
where the first equality follows from the commutator formula \eqref{eqn:intw_op_comm}. By Theorem \ref{thm:diff_eq}, the second summand in \eqref{eqn:rig_calc1} is a solution to the differential equation \eqref{eqnfor2},
so it is a multiple of the non-logarithmic fundamental basis solution in \eqref{eqn: solution2}:
\begin{equation*}
-\frac{1}{2}c\cdot z^{-2h_2+q}(1-z)^{-2h_2}{}_2F_1\left(\frac{q}{2}, 1-\frac{q}{2}; 1+q; z\right). 
\end{equation*}
Again using \cite[Equation 15.8.11]{DLMF}, the coefficient of $\left(\frac{1-z}{z}\right)^{-2h_2}$ in the expansion of this solution as a series in $\frac{1-z}{z}$ on $\til{U}$ is 
\begin{equation}\label{eqn:coeff2}
-\frac{1}{2}\frac{c(q-1)!}{\Gamma(\frac{3q}{2})\Gamma(1+\frac{q}{2})} = -\frac{c(q-1)!}{q\Gamma(\frac{3q}{2})\Gamma(\frac{q}{2})}.
\end{equation}
By our assumption that the coefficient of $(\frac{1-z}{z})^{-2h_1}$ in \eqref{eqn:rig_calc1} vanishes, the difference of \eqref{eqn:coeff1} and \eqref{eqn:coeff2} must be $0$. This is only possible if $c=0$, as claimed.

Now taking $v' = v$ in \eqref{eqn:rig_calc1}, $\mathfrak{sl}_2$-weight considerations imply that 
\begin{align}
 \label{eqn:claim2}\langle v,\Omega(Y_{\cV_{2}})(f(0)v,1)\til{v} \rangle =0.
\end{align}
 Combining \eqref{eqn:claim1} and \eqref{eqn:claim2}, the $\mathfrak{sl}_2$-homomorphism
\begin{equation}
\pi_0(\Omega(Y_{\cV_{2}})|_{\cV_{2}\otimes \cL_3}): M_2 \otimes M_3 \longrightarrow M_2
\end{equation}
vanishes. Thus by Proposition \ref{prop:surjective}, the intertwining operator $\Omega(Y_{\cV_{2}})|_{\cV_{2}\otimes \cL_{3}}$  of type $\binom{\cV_{2}}{\cV_{2}\,\cL_3}$ is not surjective. Since $\cV_{2}$ is simple when $p=2$, this means $\Omega(Y_{\cV_{2}})|_{\cV_{2}\otimes \cL_{3}} = 0$, equivalently, $Y_{\cV_{2}}|_{\cL_{3}\otimes\cV_{2}} = 0$. Since $\cL_{3}$ is the maximal proper submodule of $\cV_{1}$, this implies $\cV_{2}$ is an $L_k(\fsl_2)$-module, which is a contradiction because the only simple grading-restricted $L_k(\fsl_2)$-module is $\cL_{1}$ when $k=-2+2/q$ is admissible. Consequently, the coefficient of 
$\left(\frac{1-z}{z}\right)^{-2h_2}$ in the expansion of \eqref{eqn:rig_calc1} on $\til{U}$ is not $0$  for some $v' \in M_2\subseteq\cV_{2}'$, and we conclude that the rigidity composition $\mathfrak{R}$ is a non-zero scalar multiple of $\Id_{\cV_{2}}$.

As in the $p\geq 3$ case, we can show that the second rigidity composition $\mathfrak{R}'$ agrees with $\mathfrak{R}$ by proving \eqref{eqn:twist_with_braiding} for some $c\in\CC^\times$. Since $i$ is defined the same way for $p=2$ as for $p\geq 3$, and since $\cV_{1}$ is generated by $\vac$, the calculation \eqref{eqn:c_for_i} shows that 
\begin{equation*}
\cR\circ i=-e^{-2\pi i h_2}\cdot i
\end{equation*}
when $p=2$. For $\varepsilon$, the definitions imply that it is enough to prove
\begin{equation*}
e^{xL(-1)}\cY(w,e^{\pi i} x)w' =-e^{-2\pi ih_2}\cY(w',x)w
\end{equation*}
for any intertwining operator $\cY$ of type $\binom{\cV_3}{\cV_{2}\,\cV_{2}}$ and all $w,w'\in\cV_{2}$. Since \cite[Theorem 3.9]{McR} implies that $\cY$ is completely determined by the $\mathfrak{sl}_2$-homomorphism $\pi(\cY): M_2\otimes M_2\rightarrow M_3$, it is sufficient to take $w,w'\in M_2$ and compute
\begin{align*}
\pi_0\big(e^{L(-1)}\cY(w,e^{\pi i})w'\big) & = \pi_0\big(e^{\pi i(L(0)-2h_2)}\cY(w,1)w'\big)\nonumber\\
& = e^{\pi i(h_3-2h_2)}\pi(\cY)(w\otimes w')\nonumber\\
& =e^{\pi i(q-2h_2)}\pi(\cY)(w\otimes w') =-e^{-2\pi i h_2}\pi(\cY)(w'\otimes w),
\end{align*}
where the last equality holds because $q$ is odd and $\mathfrak{sl}_2$-homomorphisms $M_2\otimes M_2\rightarrow M_3$ are symmetric. This completes the proof that $\cV_{2}$ is rigid and self-dual when $p=2$.

Finally, to complete the proof of Theorem \ref{thm:V12_rigid}, we need to show that when $p=2$ and $q$ is odd, the intrinsic dimension of $\cV_{2}$ is 
\begin{equation*}
-e^{\pi i q/2}-e^{-\pi i q/2} =-2\cos\left(\frac{q\pi}{2}\right) = 0.
\end{equation*}
Since $\vac$ generates $\cV_{1}$, it is enough to show that $(e\circ i)(\vac)=0$, where $e$ is a suitable non-zero multiple of $\varepsilon$. Indeed, this holds because the image of $e$ is contained in the maximal proper submodule $\cL_{3}\subseteq\cV_{1}$ whose minimum conformal weight is $q>0$, whereas the conformal weight of $(e\circ i)(\vac)$ is $0$.

\section{Projective objects in \texorpdfstring{$KL^k(\mathfrak{sl}_2)$}{KLk(sl2)}}\label{sec:proj}

By Corollary \ref{cor:rigid_is_projective}, every rigid object of $KL^k(\mathfrak{sl}_2)$ is projective. In this section, we will determine all rigid and thus all projective objects in $KL^k(\mathfrak{sl}_2)$. Moreover, we will show that $KL^k(\mathfrak{sl}_2)$ has enough projectives, that is, every simple object has a projective cover. We start by using the rigidity of $\cV_{2}$ proved in the previous section to determine how $\cV_{2}$ tensors with the simple objects of $KL^k(\mathfrak{sl}_2)$.

\subsection{Tensor products involving \texorpdfstring{$\cV_{2}$}{V2}}\label{subsec:tens_prod_V2}

In Theorem \ref{thm:V12_times_Vrs}, we determined the tensor products $\cV_{2}\tens\cV_{r}$ for $p\nmid r$, but did not completely determine the tensor products of $\cV_{2}$ with the simple generalized Verma modules $\cV_{np}$. In this subsection, we will completely determine $\cV_{2}\tens\cV_{np}$, but first we compute $\cV_{2}\tens\cL_{r}$ for $p\nmid r$. For more uniform formulas, we set $\cL_0=0$:
\begin{thm}\label{thm:V12_times_Lrs}
For $n\in\ZZ_{\geq 0}$ and $1\leq r\leq p-1$,
\begin{equation*}
\cV_2\tens\cL_{np+r} \cong\left\lbrace\begin{array}{lll}
\cL_{np+r-1}\oplus\cL_{np+r+1} & \text{if} & r\leq p-2\\
\cL_{(n+1)p-2} & \text{if} & r=p-1\\
\end{array}\right. .
\end{equation*}
\end{thm}
\begin{proof}
Since $\cV_{2}$ is rigid, the functor $\cV_{2}\tens\bullet$ is exact. Thus by Theorem \ref{thm:gen_Verma_structure}, there is an exact sequence
\begin{equation}\label{eqn:V12_times_Vrs_short_exact_seq}
0\longrightarrow \cV_{2}\tens\cL_{(n+2)p-r}\longrightarrow\cV_{2}\tens\cV_{np+r}\longrightarrow\cV_{2}\tens\cL_{np+r}\longrightarrow 0
\end{equation}
for $n\in\ZZ_{\geq 0}$ and $1\leq r\leq p-1$. We can combine these exact sequences into a resolution
\begin{equation*}
\cdots\longrightarrow\cV_{2}\tens\cV_{(n+2)p+r}\longrightarrow\cV_{2}\tens\cV_{(n+2)p-r}\longrightarrow\cV_{2}\tens\cV_{np+r}\longrightarrow\cV_{2}\tens\cL_{np+r}\longrightarrow 0.
\end{equation*}
By Theorem \ref{thm:V12_times_Vrs}, this resolution becomes
\begin{equation*}
\cdots\longrightarrow\cV_{(n+2)p-r-1}\oplus\cV_{(n+2)p-r+1}\longrightarrow\cV_{np+r-1}\oplus\cV_{np+r+1}\longrightarrow\cV_{2}\tens\cL_{np+r}\longrightarrow 0.
\end{equation*}
Because the conformal weights of the two summands in each term of the resolution are non-congruent mod $\ZZ$, $\cV_{2}\tens\cL_{np+r}$ is a direct sum $W_-\oplus W_+$ where $W_\pm$ have the following resolutions by generalized Verma modules:
\begin{equation}\label{eqn:V12_times_Lrs_resolution}
\cdots\longrightarrow\cV_{(n+2)p+r\pm1}\longrightarrow\cV_{(n+2)p-(r\pm1)}\longrightarrow\cV_{np+r\pm1}\longrightarrow W_\pm \longrightarrow 0.
\end{equation}
In the case $2\leq r\leq p-2$, Theorem \ref{thm:gen_Verma_structure} implies that both $\cL_{np+r\pm 1}$ have resolutions by the same generalized Verma modules as $W_\pm$, so in this case, $W_\pm$ are quotients of $\cV_{np+r\pm1}$ which have the same graded dimensions as $\cL_{np+r\pm 1}$. Thus $\cV_{2}\tens\cL_{np+r}\cong\cL_{np+r-1}\oplus\cL_{np+r+1}$ when $2\leq r\leq p-2$.

For the case $r=1$, we similarly obtain $W_+\cong\cL_{np+2}$ when $p\geq 3$, while the resolution \eqref{eqn:V12_times_Lrs_resolution} for $W_-$ becomes
\begin{equation*}
\cdots\longrightarrow\cV_{(n+2)p}\longrightarrow\cV_{(n+2)p}\longrightarrow\cV_{np}\longrightarrow W_-\longrightarrow 0.
\end{equation*}
The map $\cV_{(n+2)p}\rightarrow\cV_{np}$ is $0$ because $\cV_{(n+2)p}$ and $\cV_{np}$ are non-isomorphic and simple (or because $\cV_{np}=0$ in the $n=0$ case), so $W_-\cong\cV_{np}\cong\cL_{np}$ when $r=1$. This completes the proof of the $r\leq p-2$ case of the theorem.

For the case $r=p-1$, we obtain $W_-\cong\cL_{(n+1)p-2}$ when $p\geq 3$, while for $p=2$, the $r=1$ case discussed above yields $W_-\cong\cL_{2n}=\cL_{(n+1)p-2}$ as well. The resolution \eqref{eqn:V12_times_Lrs_resolution} for $W_+$ becomes
\begin{equation*}
\cdots\longrightarrow\cV_{(n+3)p}\longrightarrow\cV_{(n+1)p}\longrightarrow\cV_{(n+1)p}\longrightarrow W_+\longrightarrow 0,
\end{equation*}
where the map $\cV_{(n+3)p}\rightarrow\cV_{(n+1)p}$ is $0$ as before. Thus the map $\cV_{(n+1)p}\rightarrow\cV_{(n+1)p}$ is an isomorphism, and it follows that $W_+=0$. This proves the $r=p-1$ case of the theorem.
\end{proof}

We can now fully determine $\cV_{2}\tens\cV_{np}=\cV_{2}\tens\cL_{np}$ for $n\in\ZZ_{\geq 1}$. We introduce the notation $\cP_{np+1}:=\cV_{2}\tens\cV_{np}$ for this module because it will turn out to be a projective cover of $\cL_{np+1}$.
\begin{thm}\label{thm:V12_times_Vrp}
For $n\in\ZZ_{\geq 1}$, there is a non-split exact sequence
\begin{equation*}
0\longrightarrow\cV_{np-1}\longrightarrow\cP_{np+1}\longrightarrow\cV_{np+1}\longrightarrow 0.
\end{equation*}
Moreover, $\cP_{np+1}$ is an indecomposable $V^k(\mathfrak{sl}_2)$-module with Loewy diagram:
\begin{equation*}
\begin{matrix}
  \begin{tikzpicture}[->,>=latex,scale=1.5]
\node (b1) at (1,0) {$\cL_{np+1}$};
\node (c1) at (-1, 1){$\cP_{np+1}:$};
   \node (a1) at (0,1) {$\cL_{np-1}$};
   \node (b2) at (2,1) {$\cL_{(n+2)p-1}$};
    \node (a2) at (1,2) {$\cL_{np+1}$};
\draw[] (b1) -- node[left] {} (a1);
   \draw[] (b1) -- node[left] {} (b2);
    \draw[] (a1) -- node[left] {} (a2);
    \draw[] (b2) -- node[left] {} (a2);
\end{tikzpicture}
\end{matrix} .
\end{equation*}
\end{thm}

\begin{proof}
We have already seen in Theorem \ref{thm:V12_times_Vrs} that there is an exact sequence
\begin{equation*}
0\longrightarrow\cV_{np-1}/\cJ_n\longrightarrow\cP_{np+1}\longrightarrow\cV_{np+1}\longrightarrow 0
\end{equation*}
where $\cJ_n$ is one of $\cV_{np-1}$, $\cL_{np+1}$, or $0$ (using Theorem \ref{thm:gen_Verma_structure}). The first two are impossible because rigidity of $\cV_{2}$ and Theorem \ref{thm:V12_times_Vrs} imply that
\begin{equation*}
\hom(\cV_{np-1},\cP_{np+1}) =\hom(\cV_{np-1},\cV_{2}\tens\cV_{np}) \cong\hom(\cV_{2}\tens\cV_{np-1},\cV_{np}) \neq 0,
\end{equation*}
while rigidity of $\cV_2$ and Theorem \ref{thm:V12_times_Lrs} imply that
\begin{equation*}
\hom(\cL_{np-1},\cP_{np+1}) =\hom(\cL_{np-1},\cV_{2}\tens\cV_{np})\cong\hom(\cV_{2}\tens\cL_{np-1},\cL_{np}) = 0.
\end{equation*}
Thus $\cJ_n=0$, yielding the desired exact sequence. The exact sequence does not split because a surjection $\cP_{np+1}\rightarrow\cV_{np-1}$ would imply a surjection $\cP_{np+1}\rightarrow\cL_{np-1}$, but in fact
\begin{equation*}
\hom(\cP_{np+1},\cL_{np-1}) =\hom(\cV_{2}\tens\cV_{np},\cL_{np-1})\cong\hom(\cL_{np},\cV_{2}\tens\cL_{np-1}) = 0.
\end{equation*}

To verify the Loewy diagram of $\cP_{np+1}$, note that Theorem \ref{thm:gen_Verma_structure} shows that $\cP_{np+1}$ has the four indicated composition factors. The socle of $\cP_{np+1}$ is isomorphic to $\cL_{np+1}$ because
\begin{equation*}
\dim\hom(\cL_{np+1},\cP_{np+1}) = \dim\hom(\cL_{np+1},\cV_{2}\tens\cV_{np})=\dim \hom(\cV_{2}\tens\cL_{np+1},\cL_{np}) =1
\end{equation*}
by the rigidity of $\cV_{2}$ and Theorem \ref{thm:V12_times_Lrs}, while $\hom(\cL_{np-1},\cP_{np+1})=0$ and similarly
\begin{equation*}
\hom(\cL_{(n+2)p-1},\cP_{np+1}) =\hom(\cL_{(n+2)p-1},\cV_{2}\tens\cV_{np})\cong\hom(\cV_{2}\tens\cL_{(n+2)p-1},\cL_{np})=0.
\end{equation*}
Next, we need to show that the socle of $\cP_{np+1}/\cL_{np+1}$ is isomorphic to $\cL_{np-1}\oplus\cL_{(n+2)p-1}$. For this, note that since the submodule $\cL_{np+1}\subseteq\cP_{np+1}$ is contained in the generalized Verma submodule $\cV_{np-1}$, we have a surjection
\begin{equation*}
\pi: \cP_{np+1}/\cL_{np+1}\twoheadrightarrow\cV_{np+1}\twoheadrightarrow\cL_{np+1},
\end{equation*}
and there is a short exact sequence
\begin{equation*}
0\longrightarrow\cL_{np-1}\longrightarrow\ker\pi\longrightarrow\cL_{(n+2)p-1}\longrightarrow 0.
\end{equation*}
Taking contragredients yields a short exact sequence
\begin{equation*}
0\longrightarrow\cL_{(n+2)p-1}\longrightarrow(\ker\pi)'\longrightarrow\cL_{np-1}\longrightarrow 0.
\end{equation*}
Since the minimal conformal weight of $\cL_{np-1}$ is lower than that of $\cL_{(n+2)p-1}$, the lowest conformal weight space of $(\ker\pi)'$ generates a quotient of the generalized Verma module $\cV_{np-1}$; since $\cV_{np-1}$ does not contain $\cL_{(n+2)p-1}$ as a composition factor, this quotient must be $\cL_{np-1}$. That is, $\cL_{np-1}$ is a submodule of $(\ker\pi)'$ and thus
\begin{equation*}
(\ker\pi)'\cong\cL_{np-1}\oplus\cL_{(n+2)p-1}\cong\ker\pi.
\end{equation*}
This shows that $\cL_{np-1}\oplus\cL_{(n+2)p-1}\subseteq\mathrm{Soc}(\cP_{np+1}/\cL_{np+1})$. In fact, this is the full socle of $\cP_{np+1}/\cL_{np+1}$ because the only remaining composition factor of $\cP_{np+1}/\cL_{np+1}$ is $\cL_{np+1}$, and a non-zero map $\cL_{np+1}\rightarrow\cP_{np+1}/\cL_{np+1}$ would imply a non-zero map
\begin{equation*}
\cL_{np+1}\longrightarrow\cP_{np+1}/\cL_{np+1}\longrightarrow\cV_{np+1},
\end{equation*}
which is impossible. (The above composition would be non-zero because the kernel of the second map in the composition is $\cL_{np-1}\neq\cL_{np+1}$.) We have now verified the row structure of the Loewy diagram of $\cP_{np+1}$.

To verify the arrows in the Loewy diagram of $\cP_{np+1}$, we need to check that the length-$2$ subquotients of $\cP_{np+1}$ indicated by the arrows are indecomposable. Indeed, if the length-$2$ submodules indicated by the lower arrows were decomposable, then $\cL_{np-1}$ and/or $\cL_{(n+2)p-1}$ would be submodules of $\cP_{np+1}$, which is not the case. Similarly, if the length-$2$ quotients indicated by the upper arrows were decomposable, then $\cL_{np-1}$ and/or $\cL_{(n+2)p-1}$ would be quotients of $\cP_{np+1}$, which is also not the case. Indeed, we already saw that $\hom(\cP_{np+1},\cL_{np-1})=0$, and similarly
\begin{equation*}
\hom(\cP_{np+1},\cL_{(n+2)p-1})\cong\hom(\cL_{np},\cV_{2}\tens\cL_{(n+2)p-1})=0
\end{equation*}
by Theorem \ref{thm:V12_times_Lrs} and the rigidity of $\cV_{2}$. This completes the proof that $\cP_{np+1}$ has the indicated Loewy diagram. Moreover, $\cP_{np+1}$ is indecomposable because if $\cP_{np+1}=W_1\oplus W_2$, then one of $W_1$ and $W_2$ must be $0$, because otherwise $\mathrm{Soc}(\cP_{np+1})$ would contain at least two irreducible submodules, whereas in fact $\mathrm{Soc}(\cP_{np+1})\cong\cL_{np+1}$ is irreducible.
\end{proof}

In the case $p=2$, we can now determine how $\cV_{2}$ tensors with $\cP_{np+1}$:
\begin{thm}\label{thm:V12_times_Pr1_p=2}
If $p=2$, then for $n\in\ZZ_{\geq 1}$,
\begin{equation*}
\cV_{2}\tens\cP_{2n+1}\cong\cV_{2(n-1)}\oplus2\cdot\cV_{2n}\oplus\cV_{2(n+1)}.
\end{equation*}
\end{thm}
\begin{proof}
Since $\cV_{2}$ is rigid, the functor $\cV_{2}\tens\bullet$ is exact, and then Theorems \ref{thm:V12_times_Vrp} and \ref{thm:V12_times_Vrs} show there is an exact sequence
\begin{equation*}
0\longrightarrow\cV_{2(n-1)}\oplus\cV_{2n}\longrightarrow\cV_{2}\tens\cP_{2n+1}\longrightarrow\cV_{2n}\oplus\cV_{2(n+1)}\longrightarrow 0;
\end{equation*}
for the case $n=1$, recall the convention $\cV_0=0$. Since the lowest conformal weight of $\cV_{2n}$ is not congruent to those of $\cV_{2(n\pm 1)}$ mod $\ZZ$, $\cV_{2}\tens\cP_{2n+1}$ decomposes as a direct sum $W_1\oplus W_2$ such that there are short exact sequences
\begin{equation*}
0\longrightarrow\cV_{2(n-1)}\longrightarrow W_1\longrightarrow \cV_{2(n+1)}\longrightarrow 0
\end{equation*}
and
\begin{equation*}
0\longrightarrow\cV_{2n}\longrightarrow W_2\longrightarrow\cV_{2n}\longrightarrow 0.
\end{equation*}
The lowest conformal weight space of $W_2$ is a finite-dimensional $\mathfrak{sl}_2$-module which decomposes as $M_{2n}\oplus M_{2n}$, so the universal property of generalized Verma modules implies $W_2$ contains a homomorphic image of $\cV_{2n}\oplus\cV_{2n}$. It then follows that $W_2\cong\cV_{2n}\oplus\cV_{2n}$ since both modules have the same graded dimension.

If $n=1$, we get $W_1\cong\cV_{2(n+1)}$, proving the $n=1$ case of the theorem. For $n\geq 2$, taking contragredients yields an exact sequence
\begin{equation*}
0\longrightarrow\cV_{2(n+1)}\longrightarrow W_1'\longrightarrow\cV_{2(n-1)}\longrightarrow 0.
\end{equation*}
The lowest conformal weight space of $W_1'$ is the irreducible $\mathfrak{sl}_2$-module $M_{2(n-2)}$, so the universal property of generalized Verma modules implies that $W_1'$ contains a submodule isomorphic to the simple module $\cV_{2(n-1)}$. It follows that $W_1'\cong\cV_{2(n-1)}\oplus\cV_{2(n+1)}$, and thus $W_1\cong\cV_{2(n-1)}\oplus\cV_{2(n+1)}$ as well. This proves the $n\geq 2$ case of the theorem.
\end{proof}

When $p\geq 3$, the tensor product $\cV_{2}\tens\cP_{np+1}$ will contain a new indecomposable module, as we discuss next.

\subsection{Further indecomposable modules}\label{subsec:further_indecomp}

We take $p\geq 3$ and consider the tensor product $\cV_{2}\tens\cP_{np+1}$ for $n\in\ZZ_{\geq 1}$. The exactness of $\cV_{2}\tens\bullet$ and Theorems \ref{thm:V12_times_Vrp} and \ref{thm:V12_times_Vrs} imply that there is an exact sequence
\begin{equation*}
0\longrightarrow\cV_{np-2}\oplus\cV_{np} \longrightarrow\cV_{2}\tens\cP_{np+1}\longrightarrow\cV_{np}\oplus\cV_{np+2}\longrightarrow 0
\end{equation*}
As in the proof of Theorem \ref{thm:V12_times_Pr1_p=2}, conformal weight considerations and the absence of non-split self-extensions of $\cV_{np}$ imply that
\begin{equation}\label{eqn:V2_times_Pnp+1}
\cV_{2}\tens\cP_{np+1}\cong 2\cdot\cV_{np}\oplus\cP_{np+2},
\end{equation}
where the direct summand $\cP_{np+2}$ has an exact sequence
\begin{equation*}
0\longrightarrow\cV_{np-2}\longrightarrow\cP_{np+2}\longrightarrow\cV_{np+2}\longrightarrow 0.
\end{equation*}
Now when $p\geq 4$, we assume inductively that we have obtained $\cP_{np+1},\ldots,\cP_{np+r-1}$ for some $r\in\lbrace 3,\ldots p-1\rbrace$, such that there is a short exact sequence
\begin{equation*}
0\longrightarrow \cV_{np-r+1}\longrightarrow\cP_{np+r-1}\longrightarrow\cV_{np+r-1}\longrightarrow 0.
\end{equation*}
Then by exactness of $\cV_{2}\tens\bullet$ and Theorem \ref{thm:V12_times_Vrs}, there is an exact sequence
\begin{equation*}
0\longrightarrow\cV_{np-r}\oplus\cV_{np-r+2}\longrightarrow\cV_{2}\tens\cP_{np+r-1}\longrightarrow\cV_{np+r-2}\oplus\cV_{np+r}\longrightarrow 0.
\end{equation*}
By conformal weight considerations again, $\cV_{2}\tens\cP_{np+r-1}$ contains a direct summand $\cP_{np+r}$ such that there is an exact sequence
\begin{equation}\label{eqn:Prs_exact_seq}
0\longrightarrow\cV_{np-r}\longrightarrow\cP_{np+r}\longrightarrow\cV_{np+r}\longrightarrow 0.
\end{equation}
Properties of the modules $\cP_{np+r}$ constructed by this recursive procedure can be determined in a manner similar to the proof of Theorem \ref{thm:V12_times_Vrp}:
\begin{thm}\label{thm:Prs}
For $n\in\ZZ_{\geq 1}$ and $1\leq r\leq p-1$, the exact sequence \eqref{eqn:Prs_exact_seq} for $\cP_{np+r}$ does not split, and the module $\cP_{np+r}$ is indecomposable with Loewy diagram
\begin{equation*}
\begin{matrix}
  \begin{tikzpicture}[->,>=latex,scale=1.5]
\node (b1) at (1,0) {$\cL_{np+r}$};
\node (c1) at (-1, 1){$\cP_{np+r}:$};
   \node (a1) at (0,1) {$\cL_{np-r}$};
   \node (b2) at (2,1) {$\cL_{(n+2)p-r}$};
    \node (a2) at (1,2) {$\cL_{np+r}$};
\draw[] (b1) -- node[left] {} (a1);
   \draw[] (b1) -- node[left] {} (b2);
    \draw[] (a1) -- node[left] {} (a2);
    \draw[] (b2) -- node[left] {} (a2);
\end{tikzpicture}
\end{matrix} .
\end{equation*}
\end{thm}
\begin{proof}
We prove the theorem by induction on $r$, with the base case $r=1$ being Theorem \ref{thm:V12_times_Vrp}. For $r\geq 2$, we then have
\begin{align*}
\hom(\cL_{(n+1\pm1)p-r},\cP_{np+r})  \hookrightarrow &\,\hom(\cL_{(n+1\pm1)p-r},\cV_{2}\tens\cP_{np+r-1})\nonumber\\
\cong &\,\hom(\cV_{2}\tens\cL_{(n+1\pm1)p-r},\cP_{np+r-1})=0
\end{align*}
using the fact that $\cP_{np+r}$ is a direct summand of $\cV_{2}\tens\cP_{np+r-1}$, the rigidity of $\cV_2$, Theorem \ref{thm:V12_times_Lrs}, and the fact that $\cL_{np+r-1}$ is the only irreducible submodule of $\cP_{np+r-1}$ (which holds by induction). Similarly,
\begin{align*}
\hom(\cP_{np+r},\cL_{(n+1\pm1)p-r})  \hookrightarrow &\,\hom(\cV_{2}\tens\cP_{np+r-1},\cL_{(n+1\pm1)p-r})\nonumber\\
\cong &\,\hom(\cP_{np+r-1},\cV_{2}\tens\cL_{(n+1\pm1)p-r})=0
\end{align*}
since by induction $\cL_{np+r-1}$ is also the only irreducible quotient of $\cP_{np+r-1}$. We
also get
\begin{align*}
\dim\hom(\cL_{np+r},\cP_{np+r}) & \leq\dim\hom(\cL_{np+r},\cV_{2}\tens\cP_{np+r-1})\nonumber\\
& =\dim\hom(\cV_{2}\tens\cL_{np+r},\cP_{np+r-1})=1
\end{align*}
and
\begin{align*}
\dim\hom(\cP_{np+r},\cL_{np+r}) & \leq\dim\hom(\cV_{2}\tens\cP_{np+r-1},\cL_{np+r})\nonumber\\
& =\dim\hom(\cP_{np+r-1},\cV_{2}\tens\cL_{np+r})=1.
\end{align*}
Then in fact
\begin{equation*}
\dim\hom(\cL_{np+r},\cP_{np+r})=1=\dim\hom(\cP_{np+r},\cL_{np+r})
\end{equation*}
since the exact sequence \eqref{eqn:Prs_exact_seq} (and Theorem \ref{thm:gen_Verma_structure}) shows that $\cL_{np+r}$ is both submodule and quotient of $\cP_{np+r}$.

Using these dimensions of homomorphism spaces involving $\cP_{np+r}$, we can now repeat the proof of Theorem \ref{thm:V12_times_Vrp} practically verbatim to show that the exact sequence \eqref{eqn:Prs_exact_seq} does not split and that $\cP_{np+r}$ is indecomposable with the indicated Loewy diagram.
\end{proof}

We next establish the rigidity and projectivity of the modules $\cP_{np+r}$. But first, we need a general lemma on rigidity in tensor categories:
\begin{lem}\label{lem:direct_summand_rigid}
Suppose that $X=W\oplus\til{W}$ in a tensor category, where $X$ and $\til{W}$ are rigid and self-dual, and $\hom(W,\til{W})=0$. Then $W$ is also rigid and self-dual.
\end{lem}
\begin{proof}
Let $e_X: X\tens X\rightarrow\vac$ and $i_X: \vac\rightarrow X\tens X$ be the evaluation and coevaluation morphisms for $X$, let $\pi: X\rightarrow W$ and $\til{\pi}: X\rightarrow \til{W}$ be the projection morphisms, and let $\iota: W\rightarrow X$ and $\til{\iota}: \til{W}\rightarrow X$ be the inclusion morphisms. We define
\begin{equation*}
e_W=e_X\circ(\iota\tens\iota),\qquad  i_W= (\pi\tens\pi)\circ i_X
\end{equation*}
to be evaluation and coevaluation morphisms for $W$.

To show that $W$ is rigid, we use naturality of the unit and associativity isomorphisms, the identity $\Id_X=\iota\circ\pi+\til{\iota}\circ\til{\pi}$, and the rigidity of $X$ to rewrite the first rigidity composition for $W$:
\begin{align*}
r_W & \circ(\Id_W\tens e_W) \circ\cA_{W,W,W}^{-1}\circ(i_W\tens\Id_W)\circ l_W^{-1}\nonumber\\
& = \pi\circ r_X\circ(\Id_X\tens e_X)\circ(\Id_X\tens((\iota\circ\pi)\tens\Id_X))\circ\cA_{X,X,X}^{-1}\circ(i_X\tens\Id_X)\circ l_X^{-1}\circ\iota\nonumber\\
& = \pi\circ r_X\circ(\Id_X\tens e_X)\circ\cA_{X,X,X}^{-1}\circ(i_X\tens\Id_X)\circ l_X^{-1}\circ\iota\nonumber\\
&\hspace{2em}  - \pi\circ r_X\circ(\Id_X\tens e_X)\circ(\Id_X\tens((\til{\iota}\circ\til{\pi})\tens\Id_X))\circ\cA_{X,X,X}^{-1}\circ(i_X\tens\Id_X)\circ l_X^{-1}\circ\iota\nonumber\\
& = \pi\circ\iota - r_W\circ(\Id_W\tens f)\circ\cA_{W,\til{W},W}^{-1}\circ(g\tens\Id_W)\circ l_W^{-1}
\end{align*}
where $f=e_X\circ(\til{\iota}\tens\iota)$ and $g=(\pi\tens\til{\pi})\circ i_X$. But $f=0$ since
\begin{equation*}
f\in\hom(\til{W}\tens W,\vac)\cong\hom(W,\til{W})=0,
\end{equation*}
using the self-duality of $\til{W}$. Thus the first rigidity composition for $W$ is $\pi\circ\iota=\Id_W$. Similarly,
\begin{equation*}
l_W\circ(e_W\tens\Id_W)\circ\cA_{W,W,W}\circ(\Id_W\tens i_W)\circ r_W^{-1}=\Id_W,
\end{equation*}
so $W$ is rigid and self-dual.
\end{proof}

Using the preceding lemma, we first prove:
\begin{prop}\label{prop:V1s_rigid}
For $1\leq r\leq p$, the generalized Verma module $\cV_{r}$ is rigid, self-dual, and projective in $KL^k(\mathfrak{sl}_2)$.
\end{prop}
\begin{proof}
By Corollary \ref{cor:rigid_is_projective}, it is enough to show that $\cV_{r}$ is rigid and self-dual. The $r=1$ case follows because $\cV_{1}$ is the unit object of $KL^k(\mathfrak{sl}_2)$, and the $r=2$ case is Theorem \ref{thm:V12_rigid}. Now we assume by induction that $\cV_{r-1}$ is rigid and self-dual for some $r\in\lbrace 3,\ldots, p\rbrace$. Then $\cV_{2}\tens\cV_{r-1}$ is rigid and self-dual because
\begin{equation*}
(\cV_{2}\tens\cV_{r-1})^*\cong \cV_{r-1}^*\tens\cV_{2}^*\cong\cV_{r-1}\tens\cV_{2}\cong\cV_{2}\tens\cV_{r-1}.
\end{equation*}
Moreover, $\cV_{2}\tens\cV_{r-1}\cong\cV_{r-2}\oplus\cV_{r}$ with $\cV_{r-2}$ rigid and self-dual by induction. Thus because $\hom(\cV_{r},\cV_{r-2})=0$, Lemma \ref{lem:direct_summand_rigid} implies that $\cV_{r}$ is rigid and self-dual.
\end{proof}

We also need an elementary lemma about projective covers in abelian categories:
\begin{lem}\label{lem:proj_cover}
If $W$ is a simple object in an abelian category and $p_W: P_W\twoheadrightarrow W$ is a surjection where $P_W$ is a finite-length indecomposable projective object, then $(P_W,p_W)$ is a projective cover of $W$.
\end{lem}
\begin{proof}
Suppose $\pi: P\twoheadrightarrow W$ is any surjection where $P$ is projective. Because both $P_W$ and $P$ are projective, we have morphisms $f: P\rightarrow P_W$ and $g: P_W\rightarrow P$ such that the diagrams
\begin{equation*}
\xymatrix{
& P_W \ar[d]^{p_W} &\\
P \ar[ru]^{f} \ar[r]_\pi & W\\
}\qquad\xymatrix{
& P_W \ar[ld]_g \ar[d]^{p_W} &\\
P \ \ar[r]_\pi & W\\
}
\end{equation*}
commute. We need to show that $f$ is surjective, and for this it is sufficient to show that $f\circ g\in\mathrm{End}(P_W)$ is an isomorphism. Indeed, by Fitting's Lemma, $f\circ g$ is either an isomorphism or nilpotent, and the latter is impossible because
\begin{equation*}
p_W\circ (f\circ g)^N = p_W \neq 0
\end{equation*}
for all $N\in\ZZ_{\geq 0}$.
\end{proof}

We now define an indecomposable module $\cP_{r}$ for all $r\in\ZZ_{\geq 1}$ by setting $\cP_{r}=\cV_{r}$ if either $r<p$ or $p\mid r$. We also adopt the convention $\cP_r=0$ if $r\leq 0$ in the tensor product formulas of the next theorem.
\begin{thm}\label{thm:Prs_properties}
For all $r\in\ZZ_{\geq 1}$, $\cP_{r}$ is rigid, self-dual, and a projective cover of $\cL_{r}$ in $KL^k(\mathfrak{sl}_2)$. In particular, $KL^k(\mathfrak{sl}_2)$ has enough projectives. Moreover, assuming $p\geq 3$, for $n\in\ZZ_{\geq 0}$ and $1\leq r\leq p$,
\begin{equation*}
\cV_{2}\tens\cP_{np+r}\cong\left\lbrace\begin{array}{lll}
2\cdot\cP_{np}\oplus\cP_{np+2} & \text{if} & r=1\\
\cP_{np+r-1}\oplus\cP_{np+r+1} & \text{if} & 2\leq r\leq p-2\\
\cP_{(n-1)p}\oplus\cP_{(n+1)p-2}\oplus\cP_{(n+1)p} & \text{if} & r=p-1\\
\cP_{(n+1)p+1} & \text{if} & r=p
\end{array}\right. .
\end{equation*}
\end{thm}

\begin{proof}
Proposition \ref{prop:V1s_rigid} and Lemma \ref{lem:proj_cover} show that $\cP_r$ is rigid, self-dual, and a projective cover of $\cL_r$ when $1\leq r\leq p$. We prove the corresponding general result for $\cP_{np+r}$, $n\in\ZZ_{\geq 0}$, by induction on $n$. Thus for any fixed $n\in\ZZ_{\geq 1}$, we assume by induction that $\cP_{(n-1)p+p}=\cV_{np}$ is rigid and self-dual. We will now prove that $\cP_{np+r}$ is rigid, self-dual, and a projective cover of $\cL_{np+r}$ by induction on $r$.

First, $\cP_{np+1}=\cV_{2}\tens\cV_{np}$ is rigid and self-dual, and thus also projective by Corollary \ref{cor:rigid_is_projective}, because $\cV_{2}$ and $\cV_{np}$ are rigid and self-dual. Then since $\cP_{np+1}$ is indecomposable and surjects onto $\cL_{np+1}$ by Theorem \ref{thm:V12_times_Vrp}, $\cP_{np+1}$ is a projective cover of $\cL_{np+1}$ by Lemma \ref{lem:proj_cover}. Next, when $p\geq 3$, we defined $\cP_{np+2}$ so that $\cV_{2}\tens\cP_{np+1}\cong 2\cdot\cV_{np}\oplus\cP_{np+2}$. Both $\cV_{2}\tens\cP_{np+1}$ and $2\cdot\cV_{np}$ are rigid and self-dual, and $\hom(\cP_{np+2},2\cdot\cV_{np})=0$ since $\cV_{np}$ is not a composition factor of $\cP_{np+2}$ by Theorem \ref{thm:Prs}, so $\cP_{np+2}$ is rigid and self-dual (and thus also projective) by Lemma \ref{lem:direct_summand_rigid}. Then Theorem \ref{thm:Prs} and Lemma \ref{lem:proj_cover} imply that $\cP_{np+2}$ is a projective cover of $\cL_{np+2}$.

Now for any $r\in\lbrace 2,\ldots, p-2\rbrace$, assume by induction that $\cP_{np+s}$ is rigid, self-dual, and a projective cover of $\cL_{np+s}$ for all $s\in\lbrace 1,\ldots, r\rbrace$. Then by our construction of $\cP_{np+r+1}$, 
\begin{equation*}
\cV_{2}\tens\cP_{np+r}\cong\til{\cP}_{np+r-1}\oplus\cP_{np+r+1}
\end{equation*}
where $\til{\cP}_{np+r-1}$ is a direct summand such that there is an exact sequence
\begin{equation*}
0\longrightarrow\cV_{np-r+1}\longrightarrow\til{\cP}_{np+r-1}\longrightarrow\cV_{np+r-1}\longrightarrow 0.
\end{equation*}
Since $\cV_{2}$ and $\cP_{np+r}$ are both rigid and self-dual, $\cV_{2}\tens\cP_{np+r}$ is also rigid and self-dual, and thus projective. Thus its direct summand $\til{\cP}_{np+r-1}$ is projective and surjects onto $\cL_{np+r-1}$. Then because $\cP_{np+r-1}$ is a projective cover of $\cL_{np+r-1}$,  $\til{\cP}_{np+r-1}$ surjects onto $\cP_{np+r-1}$. So because both modules have the same length, $\til{\cP}_{np+r-1}\cong\cP_{np+r-1}$, and we get 
\begin{equation*}
\cV_{2}\tens\cP_{np+r}\cong\cP_{np+r-1}\oplus\cP_{np+r+1}.
\end{equation*}
 Moreover, $\cP_{np+r-1}$ is rigid and self dual, and $\hom(\cP_{np+r+1},\cP_{np+r-1})=0$ since their composition factors are disjoint by Theorem \ref{thm:Prs}, so Lemma \ref{lem:direct_summand_rigid} implies that $\cP_{np+r+1}$ is rigid and self-dual. Then $\cP_{np+r+1}$ is a projective cover of $\cL_{np+r+1}$ by Corollary \ref{cor:rigid_is_projective}, Theorem \ref{thm:Prs}, and Lemma \ref{lem:proj_cover}.

We have now proved that $\cP_{np+r}$ is rigid, self-dual, and a projective cover of $\cL_{np+r}$ for $1\leq r\leq p-1$, but we still need to prove that $\cP_{(n+1)p}=\cL_{(n+1)p}$ is rigid and self-dual. By the exact sequence \eqref{eqn:Prs_exact_seq}, the exactness of $\cV_{2}\tens\bullet$, and Theorem \ref{thm:V12_times_Vrs}, there is an exact sequence
\begin{equation*}
0\longrightarrow\cV_{(n-1)p}\oplus\cV_{(n-1)p+2} \longrightarrow\cV_{2}\tens\cP_{(n+1)p-1}\longrightarrow\cV_{(n+1)p-2}\oplus\cV_{(n+1)p}\longrightarrow 0.
\end{equation*}
Conformal weight considerations imply $\cV_{2}\tens\cP_{(n+1)p-1}=W_1\oplus W_2$, with exact sequences
\begin{equation*}
0\longrightarrow \cV_{(n-1)p}\longrightarrow W_1\longrightarrow\cV_{np}\longrightarrow 0
\end{equation*}
and
\begin{equation*}
0\longrightarrow \cV_{(n-1)p+2}\longrightarrow W_2\longrightarrow \cV_{(n+1)p-2}\longrightarrow 0.
\end{equation*}
We get $W_1\cong\cV_{(n-1)p}\oplus\cV_{np}$ exactly as in the proof of Theorem \ref{thm:V12_times_Pr1_p=2}, while $W_2$ is a direct summand of the rigid, self-dual, and thus also projective module $\cV_{2}\tens\cP_{(n+1)p-1}$. Thus $W_2$ is a projective length-$4$ module that surjects onto $\cL_{(n+1)p-2}$, and therefore $W_2$ is isomorphic to the length-$4$ projective cover $\cP_{(n+1)p-2}$ (assuming $p\geq 3$). This shows that
\begin{equation*}
\cV_{2}\tens\cP_{(n+1)p-1}\cong\cP_{(n-1)p}\oplus\cP_{(n+1)p-2}\oplus\cP_{(n+1)p}
\end{equation*}
(where $\cP_{(n-1)p}=0$ in the case $n=1$). Moreover, both $\cV_{2}\tens\cP_{(n+1)p-1}$ and $\cP_{(n-1)p}\oplus\cP_{(n+1)p-2}$ are rigid and self-dual, and 
\begin{equation*}
\hom(\cP_{(n+1)p},\cP_{(n-1)p}\oplus\cP_{(n+1)p-2})=0,
\end{equation*}
 so $\cP_{(n+1)p}$ is rigid and self-dual by Lemma \ref{lem:direct_summand_rigid} when $p\geq 3$. When $p=2$, $\cP_{2(n+1)}$ is rigid and self-dual similarly from Theorem \ref{thm:V12_times_Pr1_p=2}. Then $\cP_{(n+1)p}$ is projective (by Corollary \ref{cor:rigid_is_projective}) and simple, so it is its own projective cover in $KL^k(\mathfrak{sl}_2)$.
 
 We have now proved that $\cP_r$ is rigid, self-dual, and a projective cover of $\cL_r$ for all $r\in\ZZ_{\geq 1}$. As for the formulas for $\cV_2\tens\cV_{np+r}$, the $n=0$, $r<p$ case is Theorem \ref{thm:V12_times_Vrs}, the $r=p$ case is the definition of $\cP_{(n+1)p}$, the $r=1$ case is \eqref{eqn:V2_times_Pnp+1}, and the $2\leq r\leq p-1$ cases were proved in the course of showing that $\cP_{np+r+1}$ is rigid and self-dual.
\end{proof}

Since $KL^k(\mathfrak{sl}_2)$ has enough projectives and all objects of $KL^k(\mathfrak{sl}_2)$ have finite length, we easily conclude:
\begin{cor}\label{cor:all_projectives}
Every projective object in $KL^k(\mathfrak{sl}_2)$ is a finite direct sum of $\cP_{r}$, $r\in\ZZ_{\geq 1}$.
\end{cor}

Corollary \ref{cor:rigid_is_projective} shows that all rigid objects in $KL^k(\mathfrak{sl}_2)$ are projective, while conversely Theorem \ref{thm:Prs_properties} and Corollary \ref{cor:all_projectives} imply that all projective objects in $KL^k(\mathfrak{sl}_2)$ are rigid (and self-dual). We conclude:
\begin{cor}\label{cor:projective_is_rigid}
The subcategory of all rigid objects in $KL^k(\mathfrak{sl}_2)$ is equal to the subcategory of all projective objects in $KL^k(\mathfrak{sl}_2)$.
\end{cor}

\subsection{More properties of \texorpdfstring{$\cP_{r}$}{Pr}}\label{subsec:more_prop}

Here we derive a few more properties of the indecomposable projective objects $\cP_{r}$. In Theorem \ref{thm:Prs_properties}, we showed that all $\cP_{r}$ are self-dual. The simple projective modules $\cP_{np}=\cV_{np}$ are also self-contragredient, although the generalized Verma modules $\cP_{r}=\cV_{r}$ for $r<p$ are not. We first show that the modules $\cP_{np+r}$ for $n\geq 1$ and $1\leq r\leq p-1$ are self-contragredient:
\begin{prop}\label{prop:Pr_self_contra}
For $r\geq p$, the projective module $\cP_{r}$ is self-contragredient.
\end{prop}
\begin{proof}
It remains to consider $\cP_{np+r}$ for $n\geq 1$ and $1\leq r\leq p-1$.
Since taking contragredients is an exact contravariant functor that fixes all simple objects of $KL^k(\mathfrak{sl}_2)$, the inclusion $\iota:\cL_{np+r}\hookrightarrow\cP_{np+r}$ induces a surjection $\iota':\cP_{np+r}'\twoheadrightarrow\cL_{np+r}$. Then fixing a surjection $\pi:\cP_{np+r}\twoheadrightarrow\cL_{np+r}$, projectivity of $\cP_{np+r}$ implies there is a map $f:\cP_{np+r}\rightarrow\cP_{np+r}'$ such that $\iota'\circ f=\pi$. We would like to show that $f$ is an isomorphism; since both $\cP_{np+r}$ and $\cP_{np+r}'$ have length $4$, it is enough to show that $f$ is surjective.

If $f$ were not surjective, there would be a surjection $\cP_{np+r}'/\im f\twoheadrightarrow \cL_{s}$ for some $s\in\ZZ_{\geq 1}$; in particular, $\hom(\cP_{np+r}',\cL_{s})\neq 0$. But
\begin{equation*}
\dim\hom(\cP_{np+r}',\cL_{s})=\dim\hom(\cL_{s}',\cP_{np+r}'')=\dim\hom(\cL_{s},\cP_{np+r})=\delta_{np+r,s},
\end{equation*}
so $s=np+r$, and $\hom(\cP_{np+r}',\cL_{np+r})=\CC\iota'$. Thus if $\cP_{np+r}'/\im f\neq 0$, then $\iota'$ would factor through $\cP_{np+r}'/\im f$. But this is impossible because $\iota'\circ f\neq 0$, so in fact $\cP_{np+r}'=\im f$, that is, $f$ is surjective.
\end{proof}

Our next goal is to show that the projective covers $\cP_{np+r}$ for $n\geq 1$ and $1\leq r\leq p-1$ are logarithmic, that is, $L(0)$ acts non-semisimply on these modules. Before doing so, however, we need to determine the self-braiding $\cR_{\cV_{2},\cV_{2}}$:
\begin{prop}\label{prop:V12_self_braiding}
Let $k=-2+p/q$ for relatively prime $p\in\ZZ_{\geq 2}$ and $q\in\ZZ_{\geq 1}$, and define $f_{\cV_{2}}= i_{\cV_{2}}\circ e_{\cV_{2}}$ where $e_{\cV_{2}}$ and $i_{\cV_{2}}$ are an evaluation and coevaluation, respectively, for $\cV_{2}$ in $KL^k(\mathfrak{sl}_2)$. Then the self-braiding of $\cV_{2}$ in $KL^k(\mathfrak{sl}_2)$ is given by
\begin{equation*}
\cR_{\cV_{2},\cV_{2}} = e^{\pi i q/2p}\cdot\Id_{\cV_{2}\tens\cV_{2}}+e^{-\pi i q/2p}\cdot f_{\cV_{2}}.
\end{equation*}
\end{prop}
\begin{proof}
Note that $f_{\cV_{2}}$ is independent of the choice of evaluation and coevaluation for $\cV_{2}$ since
\begin{equation*}
\dim\hom(\cV_{2}\tens\cV_{2},\cV_{1}) =\dim\hom(\cV_{1},\cV_{2}\tens\cV_{2}) =\dim\mathrm{End}(\cV_{2})=1
\end{equation*}
implies that any other choice of evaluation and coevaluation has the form $(c\cdot e_{\cV_{2}}, c^{-1}\cdot i_{\cV_{2}})$ for some $c\in\CC^\times$. We claim that $f_{\cV_{2}}$ is non-zero and not an isomorphism. Indeed, it is an endomorphism of
\begin{equation*}
\cV_{2}\tens\cV_{2}\cong\left\lbrace\begin{array}{lll}
\cV_{1}\oplus \cV_{3} & \text{if} & p\geq 3\\
\cP_{3} & \text{if} & p=2
\end{array}\right. ,
\end{equation*}
using Theorems \ref{thm:V12_times_Vrs} and \ref{thm:V12_times_Vrp}, and the structure of this tensor product module shows that $i_{\cV_{2}}$ maps the vacuum vector $\vac\in\cV_{1}$ to a non-zero generating vector of a submodule of $\cV_{2}\tens\cV_{2}$ isomorphic to $\cV_{1}$. Thus $i_{\cV_{2}}$ is injective but not surjective, and then $f_{\cV_{2}}$ is also not surjective. Also, $f_{\cV_{2}}$ is non-zero because $e_{\cV_{2}}$ is non-zero, proving the claim.

The structure of $\cV_{2}\tens\cV_{2}$ also shows that $\dim\mathrm{End}(\cV_{2}\tens\cV_{2})=2$ and thus this endomorphism space is spanned by $\Id_{\cV_{2}\tens\cV_{2}}$ and the non-zero non-isomorphism $f_{\cV_{2}}$.  So
\begin{equation*}
\cR_{\cV_{2},\cV_{2}}=a\cdot\Id_{\cV_{2}\tens\cV_{2}}+b\cdot f_{\cV_{2}}
\end{equation*}
for some $a,b\in\CC$. Then the same proof as in \cite[Lemma 6.1]{GN} shows that only four possibilities for $(a,b)$ are compatible with the hexagon axiom:
\begin{equation*}
(a,b)\in\lbrace\pm(e^{\pi iq/2p}, e^{-\pi iq/2p}), \pm(e^{-\pi iq/2p},e^{\pi iq/2p})\rbrace.
\end{equation*}
Then also similar to \cite{GN}, the constraints
\begin{equation*}
e_{\cV_{2}}\circ\cR_{\cV_{2},\cV_{2}}^2 = e^{-3\pi i q/p}\cdot e_{\cV_{2}},\qquad e_{\cV_{2}}\circ\cR_{\cV_{2},\cV_{2}} =-e^{-3\pi i q/2p}\cdot e_{\cV_{2}}
\end{equation*}
(which we proved in \eqref{eqn:twist_with_braiding} and \eqref{eqn:c_for_i} for $p\geq 3$, and we showed similarly for $p=2$), together with the intrinsic dimension relation
\begin{equation*}
e_{\cV_{2}}\circ i_{\cV_{2}} = -(e^{\pi iq/p}+e^{-\pi iq/p})\cdot\Id_{\cV_{1}}
\end{equation*}
of Theorem \ref{thm:V12_rigid}, force first $a^2=e^{\pi i q/p}$, and then $a=e^{\pi i q/2p}$. 
\end{proof}

We will use the formula for $\cR_{\cV_{2},\cV_{2}}$ in the next theorem for showing that $\cP_{np+r}$ is logarithmic in the case $r\geq 2$.

\begin{thm}\label{thm:Pr_log}
For $n\geq 1$ and $1\leq r\leq p-1$, the projective module $\cP_{np+r}$ is logarithmic.
\end{thm}
\begin{proof}
We prove the $r=1$ case of the theorem first. Since all conformal weights of $\cP_{np+1}$ are congruent to $h_{np+1}=\frac{qn}{4}(np+2)$ mod $\ZZ$, it is enough to show that the twist $\theta_{\cP_{np+1}}=e^{2\pi i L(0)}$ is not equal to the scalar $e^{\pi i qn(np+2)/2}$. If it were, then the double braiding
\begin{equation*}
\cR_{\cV_{2},\cV_{np}}^2 =\cR_{\cV_{np},\cV_{2}}\circ\cR_{\cV_{2},\cV_{np}}: \cV_{2}\tens\cV_{np}\longrightarrow\cV_{2}\tens\cV_{np}
\end{equation*}
would, by the balancing equation, equal
\begin{equation*}
\cR_{\cV_{2},\cV_{np}}^2 = \theta_{\cP_{np+1}}\circ(\theta_{\cV_{2}}^{-1}\tens\theta_{\cV_{np}}^{-1}) = e^{2\pi i(h_{np+1}-h_2-h_{np})}\Id_{\cV_{2}\tens\cV_{np}}=(-1)^{qn} e^{\pi i q/p}\Id_{\cV_{2}\tens\cV_{np}}.
\end{equation*}
The hexagon axiom would then imply that
\begin{equation*}
\cR^2_{\cV_{2}\tens\cV_{2},\cV_{np}} = e^{2\pi iq/p}\Id_{(\cV_{2}\tens\cV_{2})\tens\cV_{np}}.
\end{equation*}
Recall the injective coevaluation $i_{\cV_{2}}:\cV_{1}\rightarrow\cV_{2}\tens\cV_{2}$. Because $\cV_{np}$ is rigid, the functor $\bullet\tens\Id_{\cV_{np}}$ is exact and thus $i_{\cV_{2}}\tens\Id_{\cV_{np}}$ is still injective. Then we would have
\begin{align*}
e^{2\pi i q/p}(i_{\cV_{2}}\tens\Id_{\cV_{np}}) & = \cR^2_{\cV_{2}\tens\cV_{2},\cV_{np}}\circ(i_{\cV_{2}}\tens\Id_{\cV_{np}})\nonumber\\
& = (i_{\cV_{2}}\tens\Id_{\cV_{np}})\circ\cR^2_{\cV_{1},\cV_{np}} =i_{\cV_{2}}\tens\Id_{\cV_{np}},
\end{align*}
which is a contradiction since $p>1$ and $\mathrm{gcd}(p,q)=1$. Thus $\theta_{\cP_{np+1}}$ is not a scalar, equivalently, $\cP_{np+1}$ is logarithmic.

Now take $r\geq 2$; the tensor product formulas of Theorem \ref{thm:Prs_properties} imply that $\cP_{np+r}$ occurs as an indecomposable direct summand of $\cV_{2}^{\boxtimes r}\tens\cV_{np}$ with multiplicity one. Let $\iota: \cP_{np+r}\rightarrow\cV_{2}^{\boxtimes r}\tens\cV_{np}$ and $\pi: \cV_{2}^{\boxtimes r}\tens\cV_{np}\rightarrow\cP_{np+r}$ be such that $\pi\circ\iota =\Id_{\cP_{np+r}}$. Then naturality of the twist implies
\begin{equation*}
\theta_{\cP_{np+r}} =\pi\circ\theta_{\cV_{2}^{\tens r}\tens\cV_{np}}\circ\iota;
\end{equation*}
we need to show that this endomorphism does not equal the scalar $e^{2\pi i h_{np+r}}$.

By repeated applications of the balancing equation
\begin{equation*}
\theta_{\cV_{2}\tens X} =\cR_{\cV_{2},X}^2\circ(\theta_{\cV_{2}}\tens\theta_X)
\end{equation*}
and the hexagon axioms, we see that $\theta_{\cV_{2}^{\tens r}\tens\cV_{np}}$ is equal to a long composition involving associativity isomorphisms, $r(r-1)$ instances of the braiding $\cR_{\cV_{2},\cV_{2}}$, $r$ instances of the double braiding $\cR^2_{\cV_{2},\cV_{np}}$, and one instance of
\begin{equation*}
\theta_{\cV_{2}}^{\tens r}\tens\theta_{\cV_{np}} = e^{2\pi i (rh_2+h_{np})}\Id_{\cV_{2}^{\tens s}\tens\cV_{np}}.
\end{equation*}
Now by Proposition \ref{prop:V12_self_braiding}, every instance of $\cR_{\cV_{2},\cV_{2}}$ can be replaced by a linear combination of $\Id_{\cV_{2}\tens\cV_{2}}$ and $f_{\cV_{2}}$. Thus since the long composition for $\theta_{\cV_{2}^{\tens r}\tens\cV_{np}}$ contains $r(r-1)$ instances of $\cR_{\cV_{2},\cV_{2}}$, it can be written as a linear combination of $2^{r(r-1)}$ compositions, all but one of which contains at least one instance of $f_{\cV_{2}}$. The only term in the linear combination that does not contain $f_{\cV_{2}}$ is just
\begin{align*}
e^{\pi i r(r-1)q/2p} e^{2\pi i(rh_2+h_{np})} & \cdot\Id_{\cV_{2}}^{\tens(r-1)}\tens(\cR_{\cV_{2},\cV_{np}}^{2})^r\nonumber\\
& = e^{2\pi i r(r-1)q/4p} e^{-2\pi i(r-1)h_{np}}\cdot\Id_{\cV_{2}}^{\tens(r-1)}\tens\theta_{\cP_{np+1}}^r\nonumber\\
& = e^{2\pi i(h_{np+r}-rh_{np+1})}\cdot\Id_{\cV_{2}}^{\tens(r-1)}\tens\theta_{\cP_{np+1}}^r,
\end{align*}
where we have used the balancing equation and the definition \eqref{eqn:h_lambda} of the conformal weight $h_r$, $r\in\ZZ_{\geq 1}$. The remaining $2^{r(r-1)}-1$ terms of the composition that do involve $f_{\cV_{2}}$ factor through $\cV_{2}^{\tens(r-2)}\tens\cV_{np}$. The images of these terms are contained in the kernel of $\pi: \cV_{2}^{\tens r}\tens\cV_{np}\rightarrow\cP_{np+r}$ because by rigidity of $\cV_{2}$,
\begin{equation*}
\hom(\cV_{2}^{\tens(r-2)}\tens\cV_{np},\cP_{np+r}) \cong \hom(\cV_{np},\cV_{2}^{\tens(r-2)}\tens\cP_{np+r})=0,
\end{equation*}
where the last equality follows because Theorem \ref{thm:Prs_properties} shows that the simple projective module $\cV_{np}$ is not a direct summand of $\cV_{2}^{\tens(r-2)}\tens\cP_{np+r}$.

We have now shown that when $r\geq 2$,
\begin{equation*}
\theta_{\cP_{np+r}}=e^{2\pi i(h_{np+r}-rh_{np+1})}\cdot\pi\circ(\Id_{\cV_{2}}^{\tens(r-1)}\tens\theta_{\cP_{np+1}}^r)\circ\iota.
\end{equation*}
Moreover, we have shown that $\theta_{\cP_{np+1}}$ is not semisimple, so that we can write
\begin{equation*}
\theta_{\cP_{np+1}} =e^{2\pi i h_{np+1}}\cdot(\Id_{\cP_{np+1}}+g\circ f),
\end{equation*}
where $f:\cP_{np+1}\rightarrow\cL_{np+1}$ is a surjection and $g:\cL_{np+1}\rightarrow\cP_{np+1}$ is an injection, with $f\circ g=0$. It follows that
\begin{equation*}
\theta_{\cP_{np+r}}=e^{2\pi ih_{np+r}}(\Id_{\cP_{np+r}}+r\cdot\pi\circ(\Id_{\cV_{2}}^{\tens(r-1)}\tens(g\circ f))\circ\iota).
\end{equation*}
Thus we need to show that $\pi\circ(\Id_{\cV_{2}}^{\tens(r-1)}\tens(g\circ f))\circ\iota$ is the unique (up to scale) non-zero nilpotent endomorphism of $\cP_{np+r}$, whose image is the irreducible submodule $\cL_{np+r}$.

By exactness of $\cV_{2}^{\tens(r-1)}\tens\bullet$,
\begin{align*}
& \Id_{\cV_{2}}^{\tens(r-1)}\tens f: \cV_{2}^{\tens(r-1)}\tens\cP_{np+1}\longrightarrow \cV_{2}^{\tens(r-1)}\tens\cL_{np+1},\\
& \Id_{\cV_{2}}^{\tens(r-1)}\tens g: \cV_{2}^{\tens(r-1)}\tens\cL_{np+1}\longrightarrow \cV_{2}^{\tens(r-1)}\tens\cP_{np+1}
\end{align*}
are a surjection and injection, respectively, so the image of $\Id_{\cV_{2}}^{\tens(r-1)}\tens(g\circ f)$ is isomorphic to $\cV_{2}^{(r-1)}\tens\cL_{np+1}$. By Theorems \ref{thm:V12_times_Lrs} and \ref{thm:Prs_properties},
\begin{equation*}
\cV_{2}^{\tens(r-1)}\tens\cL_{np+1}\cong\cL_{np+r}\oplus\til{\cL}_{np+r-2}
\end{equation*}
where the indecomposable direct summands of $\til{\cL}_{np+r-2}$ come from $\cV_{np}$ and the $\cL_{np+s}$ and $\cP_{np+s}$ with $1\leq s\leq r-2$. Similarly,
\begin{equation*}
\cV_{2}^{\tens(r-1)}\tens\cP_{np+1}\cong\cP_{np+r}\oplus\til{\cP}_{np+r-2}
\end{equation*}
where the indecomposable summands of $\til{\cP}_{np+r-2}$ come from $\cV_{np}$ and the $\cP_{np+s}$ for $1\leq s\leq r-2$. Thus we are considering the composition
\begin{align*}
\cP_{np+r}\xrightarrow{\iota}\cP_{np+r}\oplus\til{\cP}_{np+r-2}&\xrightarrow{\Id_{\cV_{2}}^{\tens(r-1)}\tens f} \cL_{np+r}\oplus\til{\cL}_{np+r-2}\nonumber\\
&\xrightarrow{\Id_{\cV_{2}}^{\tens(r-1)}\tens g} \cP_{np+r}\oplus\til{\cP}_{np+r-2}\xrightarrow{\pi} \cP_{np+r}.
\end{align*}
Now,
\begin{align*}
\hom(\cP_{np+r},\til{\cL}_{np+r-2})=0=\hom(\til{\cP}_{np+r-2},\cL_{np+r})
\end{align*}
because $\cL_{np+r}$ is not a composition factor of $\cV_{np}$, $\cL_{np+s}$, or $\cP_{np+s}$, $1\leq s\leq r-2$. Thus $\Id_{\cV_{2}}^{\tens(r-1)}\tens f$ is a direct sum of two surjections $f_1:\cP_{np+r}\twoheadrightarrow\cL_{np+r}$ and $f_2:\til{\cP}_{np+r-2}\twoheadrightarrow\til{\cL}_{np+r-2}$. Similarly, $\Id_{\cV_{2}}^{\tens(r-1)}\tens g$ is a direct sum of two injections $g_1:\cL_{np+r}\hookrightarrow\cP_{np+r}$ and $g_2:\til{\cL}_{np+r-2}\hookrightarrow\til{\cP}_{np+r-2}$. Then
\begin{equation*}
\pi\circ(\Id_{\cV_{2}}^{\tens(r-1)}\tens(g\circ f))\circ\iota = g_1\circ f_1,
\end{equation*}
which is the unique (up to scale) non-zero nilpotent endomorphism of $\cP_{np+r}$, as desired. This completes the proof that $\theta_{\cP_{np+r}}$ is not semisimple, and thus $\cP_{np+r}$ is logarithmic.
\end{proof}

\begin{rem}
The existence of indecomposable logarithmic $V^k(\mathfrak{sl}_2)$-modules was previously conjectured in \cite[Section 5.3]{Ra}. In particular, the conjectural module denoted $\cS^{a,0;+}_{\ell p,0}$ in Conjecture 1 of \cite[Section 5.3]{Ra} seems to be our module $\cP_{\ell p+a}$. Thus we have rigorously constructed such modules here using the tensor category structure on $KL^k(\mathfrak{sl}_2)$.
\end{rem}

\begin{rem}
For $n\geq 1$ and $1\leq r\leq p-1$, the lowest conformal weight $h_{np+r}$ of $\cL_{np+r}$ is related to the lowest conformal weight $h_{np-r}$ of $\cL_{np-r}$ by $h_{np+r}=h_{np-r}+nqr$. Thus in the $\ZZ_{\geq 0}$-gradable module $\cP_{np+r}=\bigoplus_{m=0}^\infty \cP_{np+r}(m)$, the space $\cP_{np+r}(m)$ of lowest degree on which $L(0)$ acts non-semisimply has degree $nqr$. Except for low values of $q$, $n$, and $r$, it seems difficult to calculate explicitly, using \eqref{eqn:L0} for example, that $L(0)$ acts non-semisimply on this space.
\end{rem}

\section{Cocycle twist and braidings of \texorpdfstring{$KL^k(\mathfrak{sl}_2)$}{KLk(sl2)}}\label{sec:braiding}

We continue to fix $k=-2+p/q$ for relatively prime $p\in\ZZ_{\geq 2}$ and $q\in\ZZ_{\geq 1}$. In the proof of Proposition \ref{prop:V12_self_braiding}, we observed that at most four automorphisms $\cR_{\cV_{2},\cV_{2}}$ of $\cV_{2}\tens\cV_{2}$ are compatible with the hexagon axioms for a braiding on $KL^k(\mathfrak{sl}_2)$. Although only one of these automorphisms is the official braiding $\cR_{\cV_{2},\cV_{2}}$ as specified by the construction in \cite{HLZ8}, we will show in this section that all four extend to braidings on $KL^k(\mathfrak{sl}_2)$. But first, we will discuss the $3$-cocycle twist of the tensor category structure on $KL^k(\mathfrak{sl}_2)$.

\subsection{\texorpdfstring{$\ZZ/2\ZZ$}{Z/2Z}-grading and the cocycle twist}\label{subsec:cocycle_twist}

We first observe that as a category,
\begin{equation*}
KL^k(\mathfrak{sl}_2) = KL^k_{\even}(\mathfrak{sl}_2)\oplus KL^k_{\odd}(\mathfrak{sl}_2),
\end{equation*}
where for $i\in\ZZ/2\ZZ$, $KL^k_i(\mathfrak{sl}_2)$ is the full subcategory of objects whose $h(0)$-eigenvalues lie in $i+2\ZZ$. That is, modules in $KL^k_{\even}(\mathfrak{sl}_2)$ have $\mathfrak{sl}_2$-weights from the root lattice of $\mathfrak{sl}_2$, while modules in $KL^k_{\odd}(\mathfrak{sl}_2)$ have $\mathfrak{sl}_2$-weights from the non-zero coset of the root lattice in the weight lattice of $\mathfrak{sl}_2$. Since every object of $KL^k(\mathfrak{sl}_2)$ is the direct sum of finite-dimensional $\mathfrak{sl}_2$-submodules, and since $V^k(\mathfrak{sl}_2)$-module homomorphisms preserve $h(0)$-eigenvalues, every object of $KL^k(\mathfrak{sl}_2)$ is uniquely the direct sum of an object in $KL^k_{\even}(\mathfrak{sl}_2)$ and an object in $KL^k_{\odd}(\mathfrak{sl}_2)$, and there are no non-zero homomorphisms from objects in $KL^k_{\even}(\mathfrak{sl}_2)$ to objects in $KL^k_{\odd}(\mathfrak{sl}_2)$ or vice versa (this is what it means for $KL^k(\mathfrak{sl}_2)$ to decompose as the direct sum of two subcategories). The $n=0$ case of \eqref{eqn:intw_op_comm} shows that if $W$ is an object of $KL^k_i(\mathfrak{sl}_2)$ and $X$ is an object of $KL^k_j(\mathfrak{sl}_2)$ for $i,j\in\ZZ/2\ZZ$, then $W\tens X$ is an object of $KL^k_{i+j}(\mathfrak{sl}_2)$.

We can use the above $\ZZ/2\ZZ$-grading of $KL^k(\mathfrak{sl}_2)$ to modify the tensor category structure of $KL^k(\mathfrak{sl}_2)$ by the $3$-cocycle $\tau$ on $\ZZ/2\ZZ$ defined by 
\begin{equation*}
\tau(i_1,i_2,i_3)=(-1)^{i_1 i_2 i_3}
\end{equation*}
for $i_1,i_2,i_3\in\ZZ/2\ZZ$. Namely, $KL^k(\mathfrak{sl}_2)^\tau$ is the tensor category with the same tensor product bifunctor and unit isomorphisms as $KL^k(\mathfrak{sl}_2)$, but with new associativity isomorphisms
\begin{equation*}
\cA_{W_1,W_2,W_3}^\tau =\tau(i_1,i_2,i_3)\cdot\cA_{W_1,W_2,W_3}
\end{equation*}
for objects $W_1$, $W_2$, and $W_3$ in $KL^k_{i_1}(\mathfrak{sl}_2)$, $KL^k_{i_2}(\mathfrak{sl}_2)$, and $KL^k_{i_3}(\mathfrak{sl}_2)$, respectively. It is easy to see that $\cA^\tau$ still satisfies the triangle and pentagon axioms.

Note that $\cV_{2}$ is still rigid in the cocycle twist tensor category $KL^k(\mathfrak{sl}_2)^\tau$, but we need to change either the evaluation or coevaluation by a sign since
\begin{equation*}
\cA_{\cV_{2},\cV_{2},\cV_{2}}^\tau = -\cA_{\cV_{2},\cV_{2},\cV_{2}}.
\end{equation*}
Thus the intrinsic dimension of $\cV_{2}$ in $KL^k(\mathfrak{sl}_2)^\tau$ is $e^{\pi i q/p}+e^{-\pi i q/p}$, which suggests that $KL^k(\mathfrak{sl}_2)^\tau$ could be tensor equivalent to $KL^{-2+p/(q+p)}(\mathfrak{sl}_2)$. Later, we shall prove that this is indeed the case.

\begin{rem}
Cocycle twists of $G$-graded tensor categories, where $G$ is an abelian group, were previously used by Kazhdan and Wenzl in their classification of rigid semisimple tensor categories with type $A$ fusion rules \cite{KWe}. There is also a generalization of cocycle twists called zesting introduced in \cite{DGPRZ}. In zesting, not only the associativity isomorphisms but also the tensor product bifunctor of a $G$-graded braided tensor category $\cC$ can be modified (by a normalized $2$-cocycle from $G$ to the abelian group of invertible objects in the component of $\cC$ graded by the identity of $G$).
\end{rem}

\subsection{Braidings and twists on \texorpdfstring{$KL^k(\mathfrak{sl}_2)$}{KLk(sl2)} and \texorpdfstring{$KL^k(\mathfrak{sl}_2)^\tau$}{KLk(sl2)tau}}

We now determine all braidings and ribbon twists on $KL^k(\mathfrak{sl}_2)$ and its cocycle twist $KL^k(\mathfrak{sl}_2)^\tau$. First we need to show that any braiding or twist is determined by $\cR_{\cV_{2},\cV_{2}}$ and $\theta_{\cV_{2}}$. For future use, we state this result more generally as follows:
\begin{prop}\label{prop:braid_and_twist_from_V12}
Let $\cC$ be one of the tensor categories $KL^k(\mathfrak{sl}_2)$ or $KL^k(\mathfrak{sl}_2)^\tau$, equipped with any braiding $\cR$ and twist $\theta$, and let $\cF:\cC\rightarrow\cD$ be a right exact tensor functor, equipped with natural isomorphism
\begin{equation*}
F: \tens\circ(\cF\times\cF)\longrightarrow\cF\circ\tens,
\end{equation*}
where $\cD$ is a braided tensor category with a right exact tensor product and twist.
\begin{enumerate}
\item If $F_{\cV_{2},\cV_{2}}\circ\cR_{\cF(\cV_{2}),\cF(\cV_{2})} = \cF(\cR_{\cV_{2},\cV_{2}})\circ F_{\cV_{2},\cV_{2}}$, then $\cF$ is a braided tensor functor.

\item If also $\theta_{\cF(\cV_{2})}=\cF(\theta_{\cV_{2}})$, then $\theta_{\cF(W)}=\cF(\theta_W)$ for all $W$ in $\cC$.
\end{enumerate}
\end{prop}
\begin{proof}
To prove (1), we need to show that for all objects $W$, $X$ in $\cC$,
\begin{equation}\label{eqn:F_braided_tensor}
F_{X,W}\circ\cR_{\cF(W),\cF(X)} =\cF(\cR_{W,X})\circ F_{W,X}.
\end{equation}
Suppose objects $W_1$, $W_2$, and $X$ in $\cC$ satisfy this relation for $W=W_i$, $i=1,2$. Then a straightforward calculation using the hexagon axiom and compatibility of the natural isomorphism $F$ with the associativity isomorphisms implies that \eqref{eqn:F_braided_tensor} holds for $W=W_1\tens W_2$ and $X$ as well. Similarly, \eqref{eqn:F_braided_tensor} holds for $W$ and $X=X_1\tens X_2$ if it holds for $W$ and $X=X_j$, $j=1,2$. Thus induction on $m$ and $n$ shows that \eqref{eqn:F_braided_tensor} holds for $W=\cV_{2}^{\tens m}$ and $X=\cV_{2}^{\tens n}$ for all $m,n\in\ZZ_{\geq 0}$ (the base cases $m=0$ and $n=0$ are proved using compatibility of the tensor functor $\cF$ with units together with the triviality of braiding isomorphisms involving units).

Next, suppose $W$ and $X$ are indecomposable projective objects of $\cC$. Theorem \ref{thm:Prs_properties} shows that there are surjections $p_W: \cV_{2}^{\tens m}\rightarrow W$ and $p_X: \cV_{2}^{\tens n}\rightarrow X$ for suitable $m,n\in\ZZ_{\geq 0}$. Then \eqref{eqn:F_braided_tensor} in this case holds due to the commutative diagrams
\begin{equation}\label{eqn:F_braid_tens_surj_1}
\begin{matrix}
\xymatrixcolsep{5pc}
\xymatrix{
\cF(\cV_{2}^{\tens m})\tens\cF(\cV_{2}^{\tens n}) \ar[r]^{\cR_{\cF(\cV_{2}^{\tens m}),\cF(\cV_{2}^{\tens n})}} \ar[d]^{\cF(p_W)\tens\cF(p_X)} & \cF(\cV_{2}^{\tens n})\tens\cF(\cV_{2}^{\tens m}) \ar[r]^{F_{\cV_{2}^{\tens n},\cV_{2}^{\tens m}}} \ar[d]^{\cF(p_X)\tens\cF(p_W)} & \cF(\cV_{2}^{\tens n}\tens\cV_{2}^{\tens m}) \ar[d]^{\cF(p_X\tens p_W)} \\
\cF(W)\tens\cF(X) \ar[r]^{\cR_{\cF(W),\cF(X)}} & \cF(X)\tens\cF(W) \ar[r]^{F_{X,W}} & \cF(X\tens W) \\
}
\end{matrix}
\end{equation}
and
\begin{equation}\label{eqn:F_braid_tens_surj_2}
\begin{matrix}
\xymatrixcolsep{5pc}
\xymatrix{
\cF(\cV_{2}^{\tens m})\tens\cF(\cV_{2}^{\tens n}) \ar[r]^{F_{\cV_{2}^{\tens m},\cV_{2}^{\tens n}}} \ar[d]^{\cF(p_W)\tens\cF(p_X)} & \cF(\cV_{2}^{\tens m}\tens\cV_{2}^{\tens n}) \ar[r]^{\cF(\cR_{\cV_{2}^{\tens m},\cV_{2}^{\tens n}})}  \ar[d]^{\cF(p_W\tens p_X)} & \cF(\cV_{2}^{\tens n}\tens\cV_{2}^{\tens m}) \ar[d]^{\cF(p_X\tens p_W)} \\
\cF(W)\tens\cF(X) \ar[r]^{F_{W,X}} & \cF(W\tens X) \ar[r]^{\cF(\cR_{W,X})} & \cF(X\tens W) \\
}
\end{matrix}
\end{equation}
together with the surjectivity of $\cF(p_W)\tens\cF(p_X)$ (which holds because $\cF$ and the tensor product on $\cD$ are right exact). Then for any projective objects $W\cong\bigoplus_i P_i$ and $X\cong\bigoplus_j Q_j$ where each $P_i$ and $Q_j$ is indecomposable, \eqref{eqn:F_braided_tensor} holds due to commutative diagrams
\begin{equation*}
\xymatrixcolsep{5.5pc}
\xymatrix{
\cF(W)\tens \cF(X) \ar[r]^{\cR_{\cF(W),\cF(X)}} \ar[d]^{\cong} & \cF(X)\tens\cF(W) \ar[r]^{F_{X,W}} \ar[d]^{\cong} & \cF(X\tens W) \ar[d]^{\cong} \\
\bigoplus_{i,j} \cF(P_i)\tens\cF(Q_j) \ar[r]^{\bigoplus_{i,j}\cR_{\cF(P_i),\cF(Q_j)}} & \bigoplus_{i,j} \cF(Q_j)\tens\cF(P_i) \ar[r]^{\bigoplus_{i,j} F_{Q_j,P_i}} & \bigoplus_{i,j} \cF(Q_j\tens P_i)\\
}
\end{equation*}
and
\begin{equation*}
\xymatrixcolsep{5.5pc}
\xymatrix{
\cF(W)\tens \cF(X) \ar[r]^{F_{W,X}} \ar[d]^{\cong} & \cF(W\tens X) \ar[r]^{\cF(\cR_{W,X})} \ar[d]^{\cong} & \cF(X\tens W) \ar[d]^{\cong} \\
\bigoplus_{i,j} \cF(P_i)\tens\cF(Q_j) \ar[r]^{\bigoplus_{i,j} F_{P_i,Q_j}} & \bigoplus_{i,j} \cF(P_i\tens Q_j) \ar[r]^{\bigoplus_{i,j}\cF(\cR_{P_i,Q_j})}  & \bigoplus_{i,j} \cF(Q_j\tens P_i)\\
}
\end{equation*}
Finally, \eqref{eqn:F_braided_tensor} holds for all $W$ and $X$ in $\cC$ thanks to diagrams similar to \eqref{eqn:F_braid_tens_surj_1} and \eqref{eqn:F_braid_tens_surj_2}, because every object in $\cC$ is a quotient of some projective object. This proves (1).

The proof of (2) is similar. The main difference is that we need to use the balancing equation and part (1) to show that $\theta_{\cF(W\tens X)}=\cF(\theta_{W\tens X})$ if the same holds for $W$ and $X$:
\begin{align*}
\theta_{\cF(W\tens X)} & =\theta_{\cF(W\tens X)}\circ F_{W,X}\circ F_{W,X}^{-1}\nonumber\\ & = F_{W,X}\circ\theta_{\cF(W)\tens\cF(X)}\circ F_{W,X}^{-1}\nonumber\\
& = F_{W,X}\circ\cR_{\cF(W),\cF(X)}^2\circ(\theta_{\cF(W)}\tens\theta_{\cF(X)})\circ F_{W,X}^{-1}\nonumber\\
& =\cF(\cR_{W,X}^2)\circ F_{W,X}\circ(\cF(\theta_W)\tens\cF(\theta_X))\circ F_{W,X}^{-1}\nonumber\\
& =\cF(\cR_{W,X}^2\circ(\theta_W\tens\theta_X))\nonumber\\
& =\cF(\theta_{W\tens X}),
\end{align*}
as required.
\end{proof}

Taking $\cF=\Id_\cC$ in Proposition \ref{prop:braid_and_twist_from_V12}, we get:
\begin{cor}\label{cor:braid_and_twist_from_V12}
If $(\cR,\theta)$ and $(\til{\cR},\til{\theta})$ are two choices of braiding and twist on $KL^k(\mathfrak{sl}_2)$ or $KL^k(\mathfrak{sl}_2)^\tau$ such that $\cR_{\cV_{2},\cV_{2}}=\til{\cR}_{\cV_{2},\cV_{2}}$ and $\theta_{\cV_{2}}=\til{\theta}_{\cV_{2}}$, then $\cR=\til{\cR}$ and $\theta=\til{\theta}$.
\end{cor}

Now we can classify braidings and twists on $KL^k(\mathfrak{sl}_2)$ and its cocycle twist:
\begin{thm}\label{thm:KLk_braidings}
Let $k=-2+p/q$ for relatively prime $p\in\ZZ_{\geq 2}$ and $q\in\ZZ_{\geq 1}$. The tensor category $KL^k(\mathfrak{sl}_2)$ admits four braidings, and for each braiding there are two compatible twists, characterized as follows:
\begin{enumerate}

\item $\cR_{\cV_{2},\cV_{2}} =\hphantom{-} e^{\pi i q/2p}\cdot\Id_{\cV_{2}\tens\cV_{2}} + e^{-\pi i q/2p}\cdot f_{\cV_{2}},\qquad \theta_{\cV_{2}}=\pm e^{3\pi i q/2p}\cdot\Id_{\cV_{2}}$

\item $\cR_{\cV_{2},\cV_{2}} = -e^{\pi i q/2p}\cdot\Id_{\cV_{2}\tens\cV_{2}} - e^{-\pi i q/2p}\cdot f_{\cV_{2}},\qquad \theta_{\cV_{2}}=\pm e^{3\pi i q/2p}\cdot\Id_{\cV_{2}}$

\item $\cR_{\cV_{2},\cV_{2}} =\hphantom{-} e^{-\pi i q/2p}\cdot\Id_{\cV_{2}\tens\cV_{2}} + e^{\pi i q/2p}\cdot f_{\cV_{2}},\qquad \theta_{\cV_{2}}=\pm e^{-3\pi i q/2p}\cdot\Id_{\cV_{2}}$

\item $\cR_{\cV_{2},\cV_{2}} = -e^{-\pi i q/2p}\cdot\Id_{\cV_{2}\tens\cV_{2}} - e^{\pi i q/2p}\cdot f_{\cV_{2}},\qquad \theta_{\cV_{2}}=\pm e^{-3\pi i q/2p}\cdot\Id_{\cV_{2}}$

\end{enumerate}
The tensor category $KL^k(\mathfrak{sl}_2)^\tau$ also admits four braidings, and for each braiding there are two compatible twists. Specifically, for each braiding and twist $(\cR,\theta)$ of $KL^k(\mathfrak{sl}_2)$, there is a braiding and twist $(\cR^\tau, \theta^\tau)$ of $KL^k(\mathfrak{sl}_2)^\tau$ characterized by
\begin{equation*}
\cR_{\cV_{2},\cV_{2}}^\tau =e^{\pi i/2}\cdot\cR_{\cV_{2},\cV_{2}},\qquad\theta_{\cV_{2}}^\tau =e^{\pi i/2}\cdot\theta_{\cV_{2}}.
\end{equation*}
\end{thm}

\begin{proof}
As mentioned in the proof of Proposition \ref{prop:V12_self_braiding}, there are four possible braiding isomorphisms $\cR_{\cV_{2},\cV_{2}}$ compatible with the hexagon axioms, namely the four listed in the theorem. Thus by Corollary \ref{cor:braid_and_twist_from_V12}, $KL^k(\mathfrak{sl}_2)$ admits at most four braidings, and we still need to show that all for possibilities for $\cR_{\cV_{2},\cV_{2}}$ extend to braidings on $KL^k(\mathfrak{sl}_2)$.

By Proposition \ref{prop:V12_self_braiding}, braiding (1) is the official $\cR_{\cV_{2},\cV_{2}}$ specified by \cite{HLZ8}, and braiding (3) is its inverse since
\begin{equation*}
f_{\cV_{2}}\circ f_{\cV_{2}} =-(e^{\pi i q/p}+e^{-\pi i q/p})\cdot f_{\cV_{2}}
\end{equation*}
by Theorem \ref{thm:V12_rigid}. Thus (1) extends to the official braiding on $KL^k(\mathfrak{sl}_2)$ and (3) extends to the reverse braiding. For braidings (2) and (4), we use the $\ZZ/2\ZZ$-grading
\begin{equation*}
KL^k(\mathfrak{sl}_2)=KL^k_{\even}(\mathfrak{sl}_2)\oplus KL^k_{\odd}(\mathfrak{sl}_2)
\end{equation*}
introduced in the previous subsection. Given any braiding $\cR$ of $KL^k(\mathfrak{sl}_2)$, we can define a new braiding $\til{\cR}$ by
\begin{equation*}
\til{\cR}_{W_1,W_2} =(-1)^{i_1 i_2} \cR_{W_1,W_2}
\end{equation*}
for objects $W_1$ in $KL^k_{i_1}(\mathfrak{sl}_2)$ and $W_2$ in $KL^k_{i_2}(\mathfrak{sl}_2)$, $i_1, i_2\in\ZZ/2\ZZ$. It is easy to see that this braiding still satisfies the hexagon axioms, and braiding (2) is obtained in this way from braiding (1), while braiding (4) is obtained in this way from braiding (3).

For the twists, Corollary \ref{cor:braid_and_twist_from_V12} shows that given any braiding on $KL^k(\mathfrak{sl}_2)$, any compatible twist is completely determined by $\theta_{\cV_{2}}$, which must be a non-zero scalar multiple of $\Id_{\cV_{2}}$ since $\dim\mathrm{End}(\cV_{2})=1$. Assuming $\theta_{\cV_{2}}=c\cdot\Id_{\cV_{2}}$ for some $c\in\CC^\times$, the possible values of $c$ are determined by
\begin{align*}
c^2\cdot e_{\cV_{2}} & = e_{\cV_{2}}\circ(\theta_{\cV_{2}}\tens\theta_{\cV_{2}}) =e_{\cV_{2}}\circ\theta_{\cV_{2}\tens\cV_{2}}\circ\cR_{\cV_{2},\cV_{2}}^{-2} = \theta_{\cV_{1}}\circ e_{\cV_{2}}\circ\cR_{\cV_{2},\cV_{2}}^{-2} = e^{\pm 3\pi i q/p}\cdot e_{\cV_{2}},
\end{align*}
where the last equation comes from calculating $\cR_{\cV_{2},\cV_{2}}^{-2}$ for all four braidings and composing with $e_{\cV_{2}}$ (using the definition $f_{\cV_{2}}=i_{\cV_{2}}\circ e_{\cV_{2}}$ and Theorem \ref{thm:V12_rigid}); we take the positive sign for braidings (1) and (2) and the negative sign for braidings (3) and (4). This yields the two possible values of $\theta_{\cV_{2}}$ for each braiding indicated in the theorem.

We need to check that for each braiding, both possibilities for $\theta_{\cV_{2}}$ extend to twists on $KL^k(\mathfrak{sl}_2)$. Since the lowest conformal weight of $\cV_{2}$ is $h_1=3q/4p$, taking the positive sign for the twist for braidings (1) and (2) in the statement of the theorem yields the official twist $e^{2\pi i L(0)}$, while taking the positive sign for the twist in braidings (3) and (4) yields the inverse twist $e^{-2\pi i L(0)}$, which is compatible with the reverse braiding. Note that $e^{2\pi i L(0)}$ is also compatible with braiding (2) since braidings (1) and (2) yield identical double braiding isomorphisms, and therefore $e^{2\pi i L(0)}$ obeys the balancing equation in both cases. Similarly, $e^{-2\pi i L(0)}$ is compatible with both braidings (3) and (4). Thus each braiding of $KL^k(\mathfrak{sl}_2)$ is compatible with at least one twist. To get the second compatible twist for each braiding, note that if $\cR$ is any braiding on $KL^k(\mathfrak{sl}_2)$ with twist $\theta$, we can define a second twist $\til{\theta}$ by
\begin{equation*}
\til{\theta}_W=(-1)^i\theta_W
\end{equation*}
for objects $W$ in $KL^k_i(\mathfrak{sl}_2)$, $i\in\ZZ/2\ZZ$. This new twist obeys the balancing equation because
\begin{align*}
\til{\theta}_{W_1\tens W_2} =(-1)^{i_1+i_2}\theta_{W_1\tens W_2}=\cR_{W_1,W_2}^2\circ((-1)^{i_1}\theta_{W_1}\tens(-1)^{i_2}\theta_{W_2}) =\cR_{W_1,W_2}^2\circ(\til{\theta}_{W_1}\tens\til{\theta}_{W_2})
\end{align*}
for $W_1$ in $KL^k_{i_1}(\mathfrak{sl}_2)$ and $W_2$ in $KL^k_{i_2}(\mathfrak{sl}_2)$, $i_1,i_2\in\ZZ/2\ZZ$.

For $KL^k(\mathfrak{sl}_2)^\tau$, suppose that $(\mathcal{R},\theta)$ is a braiding and twist for $KL^k(\mathfrak{sl}_2)$. Then we can define a braiding $\cR^\tau$ and twist $\theta^\tau$ for $KL^k(\mathfrak{sl}_2)^\tau$ by
\begin{equation*}
\cR_{W_1,W_2}^\tau =e^{\pi i j_1 j_2/2}\cdot\cR_{W_1,W_2},\qquad\theta_W^\tau =e^{\pi i j/2}\cdot\theta_W
\end{equation*}
where $j_1,j_2,j\in\lbrace 0,1\rbrace$ and $W_1$, $W_2$, and $W$ are objects of $KL^k_{\bar{j}_1}(\mathfrak{sl}_2)$, $KL^k_{\bar{j}_2}(\mathfrak{sl}_2)$, and $KL^k_{\bar{j}}(\mathfrak{sl}_2)$, respectively. It is straightforward to show that $\cR^\tau$ satisfies the hexagon axioms and that $\theta^\tau$ is a twist. In particular $\theta^\tau$ satisfies the balancing equation because
\begin{align*}
(\cR^\tau_{W_1,W_2})^2\circ(\theta_{W_1}^\tau\tens\theta_{W_2}^\tau) = e^{\pi i(2j_1 j_2+j_1+j_2)/2}\cdot\theta_{W_1\tens W_2} =\theta^\tau_{W_1\tens W_2}
\end{align*}
for $j_1,j_2\in\lbrace 0,1\rbrace$ and $W_1$, $W_2$ as above. Conversely, if $(\cR^\tau,\theta^\tau)$ is a braiding and twist for $KL^k(\mathfrak{sl}_2)^\tau$, then we can similarly define a braiding $\cR$ and twist $\theta$ for $KL^k(\mathfrak{sl}_2)$ by
\begin{equation*}
\cR_{W_1,W_2}=e^{-\pi i j_1 j_2/2}\cdot\cR^\tau_{W_1,W_2},\qquad\theta_W=e^{-\pi i j/2}\cdot\theta^\tau_W 
\end{equation*}
for $j_1,j_2,j\in\lbrace 0,1\rbrace$ and $W_1$, $W_2$, $W$ as above. Thus there is a bijection between braiding and twist pairs for $KL^k(\mathfrak{sl}_2)$ and $KL^k(\mathfrak{sl}_2)^\tau$ as indicated in the statement of the theorem.
\end{proof}

\begin{rem}
The construction in Theorem \ref{thm:KLk_braidings} of a braiding and twist on $KL^k(\fsl_2)^\tau$ from a given braiding and twist on $KL^k(\fsl_2)$ can be viewed as an easy special case of the more general braided zesting construction for braided tensor categories graded by an abelian group introduced in \cite{DGPRZ}.
\end{rem}

\begin{rem}
In the next sections, we will compare the tensor structure on $KL^k(\mathfrak{sl}_2)$ with that on modules for the big quantum group of $\mathfrak{sl}_2$ at the root of unity $\zeta=e^{\pi i/(k+2)}$. As observed for example in \cite[Remark 3.1]{GN}, the big quantum group tensor category also admits exactly four braidings, which can be characterized similarly as in Theorem \ref{thm:KLk_braidings}. This contrasts with the case of the small quantum group quotient $u_\zeta(\fsl_2)$ of the big quantum group: if $\zeta$ has odd order, then the category of finite-dimensional $u_\zeta(\fsl_2)$-modules has only two braidings \cite{DEN}, while if $\zeta$ has even order, then it has no braidings \cite{KS, GR}, although the representation category of a quasi-Hopf modification of $u_\zeta(\fsl_2)$ does \cite{CGR}.
\end{rem}

\section{The universal property of \texorpdfstring{$KL^k(\fsl_2)$}{KLk(sl2)}}\label{sec:univ_prop}

In Theorem \ref{thm:V12_rigid}, we showed that for $k=-2+p/q$ an admissible level for $\fsl_2$, the generalized Verma module $\cV_2$ in $KL^k(\fsl_2)$ is self-dual with intrinsic dimension $-e^{\pi i q/p}-e^{-\pi i q/p}$. In this section, we will show that if $\cC$ is any (not necessarily rigid) tensor category with a rigid self-dual object $X$ of the same intrinsic dimension, then there is a unique right exact tensor functor $\cF: KL^k(\fsl_2)\rightarrow\cC$ such that $\cF(\cV_2,e_{\cV_2},i_{\cV_2})=(X,e_X,i_X)$. To prove this result, the key step is to relate the subcategory of projective objects in $KL^k(\fsl_2)$ to the category of tilting modules for quantum $\fsl_2$ at the root of unity $\zeta =e^{\pi i q/p}$.

\subsection{Tilting modules for quantum \texorpdfstring{$\mathfrak{sl}_2$}{sl2}}\label{subsec:tilting}

Let $\cP^k$ denote the full subcategory of projective objects in $KL^k(\fsl_2)$. By Corollary \ref{cor:projective_is_rigid}, $\cP^k$ is also the subcategory of all rigid objects in $KL^k(\fsl_2)$, and it is a monoidal subcategory which is closed under finite direct sums and direct summands. However, $\cP^k$ is not an abelian category because it is not closed under subquotients in general. We will show that $\cP^k$ is tensor equivalent to the rigid monoidal category $\cT_\zeta$ of tilting modules \cite{AP} for quantum $\fsl_2$ at the root of unity $\zeta=e^{\pi i/(k+2)}=e^{\pi i q/p}$ (as usual, $k=-2+p/q$ for relatively prime $p\in\ZZ_{\geq 2}$ and $q\in\ZZ_{\geq 1}$).

 First, let $\cC(\zeta,\mathfrak{sl}_2)$ denote the category of finite-dimensional comodules for the Hopf algebra $SL_\zeta(2)$ (see for example \cite[Chapter IV]{Ka} for the definitions). As in \cite[Sections VII.4 and VII.5]{Ka}, we can also view $\cC(\zeta,\mathfrak{sl}_2)$ as a category of finite-dimensional weight modules for the Hopf algebra $U_\zeta(\mathfrak{sl}_2)$; this category is a braided ribbon tensor category. Its simple objects are irreducible highest-weight modules $L_\lambda$ labeled by highest weights $\lambda\in\ZZ_{\geq 0}$, with $L_0$ the unit object. The category $\cC(\zeta,\mathfrak{sl}_2)$ has enough projectives, and we let $P_\lambda$ denote the projective cover of $L_\lambda$ for $\lambda\in\ZZ_{\geq 0}$. Set $\mathbf{X}:=L_1$, the two-dimensional \textit{standard object} of $\cC(\zeta,\mathfrak{sl}_2)$; it is self-dual with intrinsic dimension $-\zeta-\zeta^{-1}$ (see for example \cite[Exercise 8.18.8]{EGNO}).

The subcategory $\cT_\zeta\subseteq\cC(\zeta,\fsl_2)$ of tilting modules was first defined in \cite{An} in terms of certain filtrations and dual filtrations; however, as in \cite{Os}, one can also define $\cT_\zeta$ to be the smallest full monoidal subcategory of $\cC(\zeta,\mathfrak{sl}_2)$ which contains $\mathbf{X}$ and is closed under direct sums and direct summands. The indecomposable objects of $\cT_\zeta$ are labeled by their highest weights: for any weight $\lambda\in\ZZ_{\geq 0}$, there is an indecomposable module $T_\lambda$ such that $\lambda$ is the highest weight of $T_\lambda$; moreover, $\lambda$ occurs as a weight of $T_\lambda$ with multiplicity one, and every indecomposable module in $\cT_\zeta$ is isomorphic to some $T_\lambda$ \cite[Theorem 2.5]{An}. Since tensoring with $\mathbf{X}$ raises the highest weight of $T_\lambda$ by one, we have a decomposition
\begin{equation}\label{eqn:X_times_T_lambda}
\mathbf{X}\otimes T_\lambda\cong\bigoplus_{\mu=0}^{\lambda+1} n_\mu\cdot T_\mu
\end{equation}
for certain $n_\mu\in\ZZ_{\geq 0}$, with $n_{\lambda+1}=1$.

To show that $\cT_\zeta$ is tensor equivalent to $\cP^k$, we will need the composition series structure (Loewy diagram) for each  $T_\lambda$ and $P_\lambda$, as well as explicit formulas for $\mathbf{X}\otimes T_\lambda$. If the root of unity $\zeta$ has even order, these results can be found in \cite[Sections 3 and 4]{BFGT}, though the notation used there is somewhat different from that used here. For general $\zeta$, we mainly use \cite{AT}, \cite{STWZ}, and \cite{Ne} as references.

\begin{thm}\label{thm:tilting_structures}
The indecomposable tilting and projective modules in $\cC(\zeta,\mathfrak{sl}_2)$ are as follows:
\begin{enumerate}
\item $T_m=L_m$ for $0\leq m\leq p-1$.

\item $T_{(\ell+1)p-1}=P_{(\ell+1)p-1} = L_{(\ell+1)p-1}$ for $\ell\in\ZZ_{\geq 0}$.

\item $T_{\ell p+m} = P_{\ell p-m-2}$ for $\ell\in\ZZ_{\geq 1}$ and $0\leq m\leq p-2$, and the Loewy diagram of this indecomposable module is the following for $\ell =1$ and $\ell\geq 2$, respectively:
\begin{equation*}
\begin{matrix}
  \begin{tikzpicture}[->,>=latex,scale=1.5]
\node (b1) at (1,0) {$L_{p-m-2}$};
\node (c1) at (-1, 1){$T_{p+m}:$};
   \node (a1) at (0,1) {$L_{p+m}$};
    \node (a2) at (1,2) {$L_{p-m-2}$};
\draw[] (b1) -- node[left] {} (a1);
    \draw[] (a1) -- node[left] {} (a2);
\end{tikzpicture}
\end{matrix} , \qquad\qquad \begin{matrix}
  \begin{tikzpicture}[->,>=latex,scale=1.5]
\node (b1) at (1,0) {$L_{\ell p-m-2}$};
\node (c1) at (-1, 1){$T_{\ell p+m}:$};
   \node (a1) at (0,1) {$L_{\ell p+m}$};
   \node (b2) at (2,1) {$L_{(\ell-2)p+m}$};
    \node (a2) at (1,2) {$L_{\ell p-m-2}$};
\draw[] (b1) -- node[left] {} (a1);
   \draw[] (b1) -- node[left] {} (b2);
    \draw[] (a1) -- node[left] {} (a2);
    \draw[] (b2) -- node[left] {} (a2);
\end{tikzpicture}
\end{matrix} .
\end{equation*}
\end{enumerate}
\end{thm}

\begin{proof}
For any $\lambda\in\ZZ_{\geq 0}$, we write $\lambda=\ell p+m$ for unique $\ell\in\ZZ_{\geq 0}$ and $m\in\lbrace 0,1,\ldots,p-1\rbrace$. Then by \cite[Proposition 2.20(a)]{AT}, $T_{\ell p+m} = L_{\ell p+m}$ if and only if $\ell = 0$ or $m=p-1$. If $\ell\geq 1$ and $0\leq m\leq p-2$, then the Loewy diagram for $T_{\ell p+m}$ follows from the non-split exact sequences in Proposition 2.7(b), Corollary 2.8(b), and Proposition 2.20(b) of \cite{AT} (see also \cite[Corollary 4.6]{APW}). In particular, the arrows in the Loewy diagrams represent the indecomposable Weyl $U_\zeta(\fsl_2)$-submodule $\Delta_{\ell p+m}$ and quotient $\Delta_{\ell p-m-2}$ and dual Weyl submodule $\nabla_{\ell p -m-2}$ and quotient $\nabla_{\ell p+m}$, where $\Delta_{p-m-2}=\nabla_{p-m-2}=L_{p-m-2}$ if $\ell =1$.

For the projectivity statements, $T_{p-1} = L_{p-1}$ is projective in $\cC(\zeta,\fsl_2)$ and hence is its own projective cover by \cite[Theorem 10.12(1)]{Ne} (see also \cite[Theorem 9.8]{APW} for the odd order case and \cite[Lemma 3.2.1]{BFGT} for the even order case). Then for $\lambda \geq p$, \eqref{eqn:X_times_T_lambda} implies that $T_\lambda$ is a direct summand of $\mathbf{X}^{\otimes(\lambda-p+1)}\otimes T_{p-1}$. Since projective objects of $\cC(\zeta,\fsl_2)$ are closed under tensoring with rigid objects and taking direct summands, $T_\lambda$ is projective for $\lambda\geq p$. Since $T_\lambda$ is also indecomposable, it is then the projective cover of its unique simple quotient, which can be determined from the Loewy diagrams.
\end{proof}

The next theorem is the characteristic $0$ special case of \cite[Proposition 4.4]{STWZ}:
\begin{thm}\label{thm:tilting_tensor_products}
The tensor products of $\mathbf{X}$ with the indecomposable tilting modules in $\cC(\zeta,\mathfrak{sl}_2)$ are the following, where we use the notational convention $T_\lambda=0$ for $\lambda<0$:
\begin{enumerate}
\item If $p=2$, then for $\ell\in\ZZ_{\geq 0}$ and $m=0,1$,
\begin{equation*}
\mathbf{X}\otimes T_{2\ell +m}\cong\left\lbrace\begin{array}{lll}
T_{2\ell-3}\oplus 2\cdot T_{2\ell-1}\oplus T_{2\ell+1} & \text{if} & m=0\\
T_{2(\ell+1)} & \text{if} & m=1
\end{array}\right. .
\end{equation*}

\item If $p=3$, then for $\ell\in\ZZ_{\geq 0}$ and $0\leq m\leq p-1$,
\begin{equation*}
\mathbf{X}\otimes T_{\ell p +m}\cong\left\lbrace\begin{array}{lll}
2\cdot T_{\ell p-1}\oplus T_{\ell p +1} & \text{if} & m=0\\
T_{\ell p+m-1}\oplus T_{\ell p +m +1} & \text{if} & 1\leq m\leq p-3\\
T_{(\ell-1)p-1}\oplus T_{(\ell+1)p-3}\oplus  T_{(\ell+1)p-1} & \text{if} & m=p-2\\
T_{(\ell+1)p} & \text{if} & m=p-1
\end{array}\right. .
\end{equation*}
\end{enumerate}
\end{thm}

We can determine all morphisms between indecomposable tilting modules from their Loewy diagrams. In particular, since the non-simple tilting modules $T_{\ell p+m}=P_{\ell p-m-2}$ for $\ell\geq 1$, $0\leq m\leq p-2$ are projective, $\hom(T_{\ell p+m},\bullet)$ is exact, and thus the dimension of $\hom(T_{\ell p+m},T_\mu)$ is equal to the multiplicity of $L_{\ell p-m-2}$ as a composition factor of $T_\mu$. Thus the following are all the non-zero morphism spaces for indecomposable tilting modules:
\begin{enumerate}

\item For $0\leq m\leq p-2$,
\begin{equation*}
\hom(T_m, T_m)=\CC\cdot\Id_{L_m},\qquad \hom(T_m, T_{2p-m-2}) =\CC\cdot f_m^+,
\end{equation*}
where $f_m^+$ is the inclusion of $T_m=L_m$ as the socle of $T_{2p-m-2}=P_m$.

\item For $\ell\geq 1$,
\begin{equation*}
\hom(T_{\ell p-1},T_{\ell p-1}) =\CC\cdot\Id_{L_{\ell p-1}}.
\end{equation*}

\item For $\ell\geq 1$ and $0\leq m\leq p-2$,
\begin{align*}
&\hom(T_{\ell p+m}, T_{(\ell+1\pm 1)p-m-2})  =\CC\cdot f_{\ell p+m}^\pm,\nonumber\\
&\hom(T_{\ell p+m},T_{\ell p+m})  =\CC\cdot\Id_{T_{\ell p+m}}\oplus\CC\cdot f^+_{\ell p-m-2}\circ f_{\ell p+m}^-.
\end{align*}
\end{enumerate}
Here we have introduced notation for the non-zero, non-identity morphisms; for all other $\lambda,\mu\in\ZZ_{\geq 0}$, we have $\hom(T_\lambda, T_\mu)=0$.

In \cite{Os}, Ostrik proved a universal property of the monoidal category $\cT_\zeta$:
\begin{thm}[\cite{Os}, Theorem 2.4]\label{thm:tilt_univ_prop}
Let $\cC$ be an additive monoidal category which is closed under direct summands, and let $(X, e_X, i_X)$ be a rigid self-dual object of $\cC$ with intrinsic dimension $-\zeta-\zeta^{-1}$. Then there is a unique (up to natural isomorphism) additive tensor functor $\cF: \cT_\zeta\rightarrow\cC$ such that $\cF(\mathbf{X}, e_{\mathbf{X}}, i_{\mathbf{X}}) =(X,e_X,i_X)$.
\end{thm}

Now for level $k=-2+p/q$, the full subcategory $\cP^k$ of projective (equivalently rigid) objects in $KL^k(\mathfrak{sl}_2)$ is a monoidal category which contains $\cV_{2}$ and is closed under finite direct sums and direct summands. Moreover, Theorem \ref{thm:V12_rigid} says that $\cV_{2}$ is self-dual with intrinsic dimension $-\zeta-\zeta^{-1}$. Thus by Theorem \ref{thm:tilt_univ_prop}, there is an additive tensor functor $\cF: \mathcal{T}_\zeta\rightarrow\cP^k$ such that $\cF(\mathbf{X})=\cV_{2}$. We will show that $\cF$ is an equivalence of categories and thus a tensor equivalence. We prove essential surjectivity first:
\begin{prop}\label{prop:F_essential_surj}
If $\cF: \cT_\zeta\rightarrow\cP^k$ is an additive tensor functor such that $\cF(\mathbf{X})=\cV_2$, then $\cF(T_\lambda)\cong\cP_{\lambda+1}$ for any $\lambda\in\ZZ_{\geq 0}$, and $\cF$ is essentially surjective.
\end{prop}
\begin{proof} 
We prove that $\cF(T_\lambda)\cong\cP_{\lambda+1}$ by induction on $\lambda$. The base case $\lambda =0$ holds because $T_0=L_0$ and $\cP_1=\cV_1$ are the units objects of $\cT_\zeta$ and $\cP^k$, respectively. For the inductive step, Theorem \ref{thm:tilting_tensor_products} shows that for any $\lambda\in\ZZ_{\geq 0}$,
\begin{equation*}
\mathbf{X}\otimes T_\lambda\cong \til{T}_{\lambda-1}\oplus T_{\lambda+1}
\end{equation*}
where $\til{T}_{\lambda-1}$ is a direct sum of $T_\mu$ for $\mu\leq \lambda -1$. Similarly, by Theorems \ref{thm:V12_times_Pr1_p=2} and \ref{thm:Prs_properties}, 
\begin{equation*}
\cV_{2}\tens\cP_{\lambda+1} \cong \til{\cP}_{\lambda}\oplus\cP_{\lambda+2}
\end{equation*}
where $\til{\cP}_{\lambda}$ is the direct sum of $\cP_\mu$ corresponding to $\til{T}_{\lambda-1}$ under the correspondence $T_\mu\mapsto\cP_{\mu+1}$. Thus because $\cF$ is an additive monoidal functor and using the inductive hypothesis,
\begin{align*}
\til{\cP}_{\lambda}\oplus\cF(T_{\lambda+1}) & \cong\cF(\til{T}_{\lambda-1}\oplus T_{\lambda+1}) \cong \cF(\mathbf{X}\otimes T_\lambda) \nonumber\\
&\cong\cF(\mathbf{X})\tens\cF(T_\lambda)\cong\cV_{2}\tens\cP_{\lambda+1}\cong\til{\cP}_{\lambda}\oplus\cP_{\lambda+2}.
\end{align*}
Thus because the indecomposable summands of a finite-length module are unique up to isomorphism by the Krull-Schmidt Theorem, we get $\cF(T_{\lambda+1})\cong\cP_{\lambda+2}$, completing the induction. 

Now because $\cF$ is additive and every object of $\cP^k$ is a direct sum of $\cP_{r}$ for $r\in\ZZ_{\geq 1}$, it follows that $\cF$ is essentially surjective.
\end{proof}

To prove that $\cF:\cT_\zeta\rightarrow\cP^k$ is also fully faithful, and thus an equivalence of categories, we will need a general lemma. Thus let $\cC$ and $\cD$ be monoidal categories and let $\cF: \cC\rightarrow\cD$ be a monoidal functor equipped with isomorphism $\varphi:\cF(\vac_\cC)\rightarrow\vac_\cD$ and natural isomorphism
\begin{equation*}
F: \tens_\cD\circ(\cF\times\cF)\longrightarrow\cF\circ\tens_\cC.
\end{equation*}
Let $X$ be a rigid self-dual object of $\cC$ with evaluation $e_X: X\tens_\cC X\rightarrow\vac_\cC$ and coevaluation $i_X:\vac_\cC\rightarrow X\tens_\cC X$. Then $\cF(X)$ is also rigid and self-dual with evaluation and coevaluation
\begin{equation*}
e_{\cF(X)}=\varphi\circ\cF(e_X)\circ F_{X,X},\quad\qquad i_{\cF(X)}=F_{X,X}^{-1}\circ\cF(i_X)\circ\varphi^{-1},
\end{equation*}
respectively. For objects $W_1$, $W_2$ of $\cC$, we have an isomorphism
\begin{equation*}
\eta_{W_1,W_2}^X: \hom_\cC(W_1, X\tens_\cC W_2)\longrightarrow\hom_\cC(X\tens_\cC W_1, W_2)
\end{equation*}
where for a morphism $f: W_1\rightarrow X\tens_\cC W_2$ in $\cC$, $\eta^X_{W_1,W_2}$ is defined to be the composition
\begin{equation*}
X\tens_\cC W_1\xrightarrow{\Id_X\tens_\cC f} X\tens_\cC(X\tens_\cC W_2)\xrightarrow{\cA_{X,X,W_2}} (X\tens_\cC X)\tens_\cC W_2\xrightarrow{e_X\tens_\cC\Id_{W_2}} \vac_\cC\tens_\cC W_2\xrightarrow{l_{W_2}} W_2.
\end{equation*}
We have a similar isomorphism $\eta^{\cF(X)}_{\cF(W_1),\cF(W_2)}$ between morphism spaces in $\cD$.

\begin{lem}\label{lem:F_and_eta}
For any morphism $f: W_1\rightarrow X\tens_\cC W_2$ in $\cC$,
\begin{equation*}
\cF(\eta^X_{W_1,W_2}(f)) =\eta^{\cF(X)}_{\cF(W_1),\cF(W_2)}(F_{X,W_2}^{-1}\circ\cF(f))\circ F_{X,W_1}^{-1}.
\end{equation*}
In particular, $\cF(f)=0$ if and only if $\cF(\eta^X_{W_1,W_2}(f))=0$.
\end{lem}
\begin{proof}
The desired identity follows from the following commutative diagram (where we suppress subscripts to save space):
\begin{equation*}
\xymatrixcolsep{4pc}
\xymatrix{
\cF(X\tens W_1) \ar[r]^{F^{-1}} \ar[d]^{\cF(\Id\tens f)} & \cF(X)\tens\cF(W_1) \ar[d]^{\Id\tens\cF(f)} & \\
\cF(X\tens(X\tens W_2)) \ar[r]^{F^{-1}} \ar[d]^{\cF(\cA)} & \cF(X)\tens\cF(X\tens W_2) \ar[r]^(.45){\Id\tens F^{-1}} & \cF(X)\tens(\cF(X)\tens\cF(W_2)) \ar[d]^{\cA} \\
\cF((X\tens X)\tens W_2) \ar[r]^{F^{-1}} \ar[d]^{\cF(e_X\tens\Id)} & \cF(X\tens X)\tens\cF(W_2) \ar[r]^(.45){F^{-1}\tens\Id} \ar[d]^{\cF(e_X)\tens\Id} & (\cF(X)\tens\cF(X))\tens\cF(W_2) \ar[ldd]^{e_{\cF(X)}\tens\Id}\\
\cF(\vac_\cC\tens W_2) \ar[r]^{F^{-1}} \ar[d]^{\cF(l)} & \cF(\vac_\cC)\tens\cF(W_2) \ar[d]^{\varphi\tens\Id} & \\
\cF(W_2) & \vac_\cD\tens\cF(W_2) \ar[l]_{l} & \\
}
\end{equation*}
as well as from the definitions of $\eta^X_{W_1,W_2}$ and $\eta^{\cF(X)}_{\cF(W_1),\cF(W_2)}$.
\end{proof}

Using the preceding lemma, we now show that $\cT_\zeta$ and $\cP^k$ are tensor equivalent:
\begin{thm}\label{thm:Tzeta_Pk_equiv}
The additive tensor functor $\cF: \cT_\zeta\rightarrow\cP^k$ such that $\cF(\mathbf{X})=\cV_2$ is an equivalence of categories.
\end{thm}
\begin{proof}
In view of Proposition \ref{prop:F_essential_surj}, it remains to show that $\cF$ is fully faithful. Thus for any tilting modules $W$ and $X$, we need to show that
\begin{equation*}
\cF: \hom_{\cT_\zeta}(W,X)\longrightarrow\hom_{\cP^k}(\cF(W),\cF(X))
\end{equation*}
is an isomorphism. Since $W$ and $X$ are both isomorphic to finite direct sums of indecomposable tilting modules, we fix isomorphisms $f: W\rightarrow\bigoplus_{i} T_{\lambda_i}$ and $g: X\rightarrow\bigoplus_{j} T_{\mu_j}$. For each $i$, we use $q_i: T_{\lambda_i}\rightarrow\bigoplus_{i} T_{\lambda_i}$ and $p_i:\bigoplus_{i} T_{\lambda_i}\rightarrow T_{\lambda_i}$ to denote the inclusion and projection morphisms. We also use $\lbrace q_j, p_j\rbrace$, $\lbrace \til{q}_i,\til{p}_i\rbrace$, and $\lbrace \til{q}_j,\til{p}_j\rbrace$ to denote the inclusion and projection morphisms for $\bigoplus_{j} T_{\mu_j}$, $\bigoplus_{i} \cF(T_{\lambda_i})$, and $\bigoplus_{j} \cF(T_{\mu_j})$, respectively. Then for any morphism $F: W\rightarrow X$ in $\cT_\zeta$, we have a commutative diagram
\begin{equation*}\label{eqn:direct_sum_diagram}
\begin{matrix}
\xymatrixcolsep{4.5pc}
\xymatrixrowsep{3pc}
\xymatrix{
\cF(W) \ar[r]^(.42){\cF(f)} \ar[d]^{\cF(F)} & \cF\big(\bigoplus_i T_{\lambda_i}\big) \ar[r]^{\sum_i \til{q}_i\circ\cF(p_i)} \ar[d]^{\cF(g\circ F\circ f^{-1})} & \bigoplus_i \cF(T_{\lambda_i}) \ar[d]^{\sum_{i,j} \til{q}_j\circ\cF(p_j\circ g\circ F\circ f^{-1}\circ q_i)\circ \til{p}_i}\\
\cF(X) \ar[r]^(.42){\cF(g)} & \cF\big(\bigoplus_j T_{\mu_j}\big) \ar[r]^{\sum_j \til{q}_j\circ\cF(p_j)} & \bigoplus_j \cF(T_{\mu_j}) \\
}
\end{matrix}
\end{equation*}
where the horizontal arrows are isomorphisms. In particular, $\cF(F)=0$ if and only if $\cF(p_j\circ g\circ F\circ f^{-1}\circ q_i)=0$ for all $i$ and $j$. It follows that $\cF$ is faithful if and only if
\begin{equation}\label{eqn:F_on_indecomposables}
\cF: \hom_{\cT_\zeta}(T_\lambda, T_\mu)\longrightarrow\hom_{\cP^k}(\cF(T_\lambda),\cF(T_\mu))
\end{equation}
is injective for all $\lambda,\mu\in\ZZ_{\geq 0}$. 

For determining whether $\cF$ is full, consider any morphism $G:\cF(W)\rightarrow\cF(X)$ in $\cP^k$. If for all $i$ and $j$ there exists $g_{i,j}: T_{\lambda_i}\rightarrow T_{\mu_j}$ such that 
\begin{equation*}
\cF(g_{i,j})=\cF(p_j\circ g)\circ G\circ\cF(f^{-1}\circ q_i),
\end{equation*}
then it is straightforward to show that
\begin{equation*}
G=\sum_{i,j}\cF\left(g^{-1}\circ q_j\circ g_{i,j}\circ p_i\circ f\right).
\end{equation*}
It follows that $\cF$ is full if and only if \eqref{eqn:F_on_indecomposables} is surjective for all $\lambda,\mu\in\ZZ_{\geq 0}$.

We are now reduced to showing that \eqref{eqn:F_on_indecomposables} is an isomorphism for all $\lambda,\mu\in\ZZ_{\geq 0}$. We have already listed all the non-zero spaces $\hom_{\cT_\zeta}(T_\lambda, T_\mu)$ in the paragraph following Theorem \ref{thm:tilting_tensor_products}. Similarly, projectivity of the $V^k(\mathfrak{sl}_2)$-modules $\cP_{r}$ in $KL^k(\mathfrak{sl}_2)$ imply that the following are all non-zero morphism spaces $\hom_{\cP^k}(\cP_r,\cP_{r'})$:
\begin{enumerate}

\item For $1\leq r\leq p-1$,
\begin{equation*}
\hom(\cP_{r},\cP_{r}) =\CC\cdot\Id_{\cV_{r}},\qquad\hom(\cP_{r},\cP_{2p-r})=\CC\cdot F_{r}^+,
\end{equation*}
where $F_{r}^+: \cV_{r}\hookrightarrow\cP_{2p-r}$ is the inclusion from \eqref{eqn:Prs_exact_seq}.

\item For $n\geq 1$,
\begin{equation*}
\hom(\cP_{np},\cP_{np})=\CC\cdot\Id_{\cV_{np}}.
\end{equation*}

\item For $n\geq 1$ and $1\leq r\leq p-1$,
\begin{align*}
&\hom(\cP_{np+r}, \cP_{(n+1\pm 1)p-r})=\CC\cdot F_{np+r}^\pm,\nonumber\\
&\hom(\cP_{np+r},\cP_{np+r})=\CC\cdot\Id_{\cP_{np+r}}\oplus\CC\cdot F^+_{np-r}\circ F_{np+r}^-.
\end{align*}
\end{enumerate}
Since $\cF(T_{\ell p+m})\cong\cP_{\ell p+m+1}$ for $\ell\in\ZZ_{\geq 0}$ and $0\leq m\leq p-2$, we thus get
\begin{equation*}
\dim\hom_{\cT_\zeta}(T_\lambda, T_\mu) =\dim\hom_{\cP^k}(\cF(T_\lambda),\cF(T_\mu))
\end{equation*}
for all $\lambda,\mu\in\ZZ_{\geq 0}$. Thus we just need to show that $\cF$ acts injectively on all one-dimensional spaces $\hom_{\cT_\zeta}(T_\lambda,T_\mu)$, and that on the two-dimensional spaces $\mathrm{End}_{\cT_\zeta}(T_{\ell p+m})$ for $\ell\geq 1$ and $0\leq m\leq p-2$, $\cF$ maps the non-isomorphism $f_{\ell p-m-2}^+\circ f_{\ell p+m}^-$ to a non-zero non-isomorphism in $\mathrm{End}_{\cP^k}(\cF(T_{\ell p+m}))$.

First, since $\cF(\Id_{T_\lambda})=\Id_{\cF(T_\lambda)}$, $\cF$ is an isomorphism on all one-dimensional endomorphism spaces. We also need to show that $\cF(f_{\ell p+m}^+)\neq 0$ for $\ell\geq 0$, $0\leq m\leq p-2$, and that $\cF(f_{\ell p+m}^-)\neq 0$ for $\ell \geq 1$, $0\leq m\leq p-2$. For the case of $f_{\ell p+m}^+$, the rigidity and self-duality of $\mathbf{X}$ yield an isomorphism of one-dimensional spaces:
\begin{align*}
\hom(T_{\ell p+m}, &\, T_{(\ell+2)p-m-2})  \xrightarrow{\cong} \hom(T_{\ell p+m}, \mathbf{X}^{\otimes(p-m-1)}\otimes T_{(\ell+1)p-1})\nonumber\\
& \xrightarrow{\eta_{T_{\ell p+m},T_{(\ell+1)p-1}}^{\mathbf{X}^{\otimes(p-m-1)}}}\hom(\mathbf{X}^{\otimes(p-m-1)}\otimes T_{\ell p+m}, T_{(\ell+1)p-1}) \xrightarrow{\cong} \mathrm{End}(T_{(\ell+1)p-1}),
\end{align*}
where the first and third isomorphisms come from identifying $T_{(\ell+2)p-m-2}$ as a direct summand of $\mathbf{X}^{\otimes(p-m-1)}\otimes T_{(\ell+1)p-1}$ and $T_{(\ell+1)p-1}$ as a direct summand of $\mathbf{X}^{\otimes(p-m-1)}\otimes T_{\ell p+m}$, respectively, both with multiplicity one. By rescaling if necessary, we may assume $f_{\ell p+m}^+$ maps to $\Id_{T_{(\ell+1)p-1}}$ under this isomorphism. Then by Lemma \ref{lem:F_and_eta}, if $\cF(f_{\ell p+m}^+)$ were $0$, then $\cF(\Id_{T_{(\ell+1)p-1}})$ would also be $0$, which is not the case. Thus $\cF(f_{\ell p+m}^+)\neq 0$. Similarly, $\cF(f_{\ell p+m}^-)\neq 0$ for $\ell\geq 1$, $0\leq m\leq p-2$ as a consequence of the isomorphism of one-dimensional spaces
\begin{align*}
\mathrm{End}(T_{\ell p-1}) \xrightarrow{\cong} &\,\hom(T_{\ell p-1}, \mathbf{X}^{\otimes(m+1)}\otimes T_{\ell p-m-2})\nonumber\\
& \xrightarrow{\eta_{T_{\ell p-1}, T_{\ell p-m-2}}^{\mathbf{X}^{\otimes(m+1)}}} \hom(\mathbf{X}^{\otimes(m+1)}\otimes T_{\ell p-1}, T_{\ell p-m-2}) \xrightarrow{\cong}\hom(T_{\ell p+m}, T_{\ell p-m-2}),
\end{align*}
together with Lemma \ref{lem:F_and_eta}.

We have now shown that $\cF$ defines an isomorphism on all one-dimensional (and all zero-dimensional)  morphism spaces $\hom_{\cT_\zeta}(T_\lambda,T_\mu)$. It remains to consider the two-dimensional endomorphism spaces $\mathrm{End}_{\cT_\zeta}(T_{\ell p+m})$ for $\ell\geq 1$, $0\leq m\leq p-2$. First, we have $\cF(\Id_{T_{\ell p+m}})=\Id_{\cF(T_{\ell p+m)}}$, and then $\cF(f_{\ell p-m-2}^+\circ f_{\ell p+m}^-)$ is the composition of two non-zero morphisms
\begin{align*}
\cF(T_{\ell p+m}) \longrightarrow\cF( T_{\ell p-m-2})\longrightarrow\cF(T_{\ell p+m}).
\end{align*}
So identifying $\cF(T_{\ell p+m})\cong \cP_{\ell p+m+1}$ and $\cF(T_{\ell p -m -2})\cong \cP_{\ell p-m-1}$, $\cF(f_{\ell p-m-2}^+\circ f_{\ell p+m}^-)$ corresponds to a non-zero multiple of the second basis element
\begin{equation*}
F_{\ell p-m-1}^+\circ F_{\ell p+m+1}^-\in\mathrm{End}_{\cP^k}(\cP_{\ell p+m+1}).
\end{equation*}
Thus $\cF$ maps a basis of $\mathrm{End}_{\cT_\zeta}(T_{\ell p+m})$ to a basis of $\mathrm{End}_{\cP^k}(\cF(T_{\ell p+m}))$, completing the proof that $\cF: \cT_\zeta\rightarrow\cP^k$ is an equivalence of categories, and therefore a tensor equivalence.
\end{proof}

\subsection{Tensor functors out of \texorpdfstring{$KL^k(\fsl_2)$}{KLk(sl2)}}\label{subsec:univ_prop}

We will now use Theorems \ref{thm:tilt_univ_prop} and \ref{thm:Tzeta_Pk_equiv} to derive the universal property of $KL^k(\fsl_2)$. But first, we need a general theorem about right exact extensions of functors. For a proof, see for example Theorems A.1 and 4.10 in \cite{McR-Del} (in the notation of these theorems from \cite{McR-Del}, we take $\til{\cD}=\cP$, $P_X=X$, $Q_X=0$, and $\cD_{X_1,X_2}=\cD$):
\begin{thm}\label{thm:extend_to_right_exact}
Let $\cD$ be a $\CC$-linear abelian category with enough projectives and let $\cP$ be the full subcategory of projective objects in $\cD$. Then for any $\CC$-linear functor $\cG: \cP\rightarrow\cC$ where $\cC$ is a $\CC$-linear abelian category, there is a unique (up to natural isomorphism) right exact $\CC$-linear functor $\cF: \cD\rightarrow\cC$ such that $\cF\vert_{\cP}\cong \cG$. If in addition $\cD$ and $\cC$ are (not necessarily rigid) tensor categories with right exact tensor products, $\cP$ is a monoidal subcategory of $\cD$ (in particular, the unit object of $\cD$ is projective), and $\cG:\cP\rightarrow\cC$ is a tensor functor, then the unique right exact extension $\cF:\cD\rightarrow\cC$ is also a tensor functor.
\end{thm}

Now the following theorem gives the universal property of $KL^k(\mathfrak{sl}_2)$:

\begin{thm}\label{thm:KLk_univ_prop}
Let $k=-2+p/q$ for relatively prime $p\in\ZZ_{\geq 2}$ and $q\in\ZZ_{\geq 1}$, let $\cC$ be a (not necessarily rigid) tensor category with right exact tensor product $\tens_\cC$, and let $X$ be a rigid self-dual object of $\cC$ with evaluation $e_X: X\tens_\cC X\rightarrow\vac_\cC$ and coevaluation $i_X: \vac_\cC\rightarrow X\tens_\cC X$ such that
\begin{equation*}
e_X\circ i_X = -(e^{\pi i q/p}+e^{-\pi i q/p})\cdot\Id_{\vac_\cC}.
\end{equation*}
Then there is a unique up to natural isomorphism right exact tensor functor $\cF: KL^k(\mathfrak{sl}_2)\rightarrow\cC$, equipped with isomorphism $\varphi:\cF(\cV_{1})\rightarrow\vac_\cC$ and natural isomorphism
\begin{equation*}
F: \tens_\cC\circ(\cF\times \cF)\longrightarrow\cF\circ\tens,
\end{equation*}
 such that $\cF(\cV_{2})=X$ and
 \begin{equation*}
\varphi\circ \cF(e_{\cV_{2}})\circ F_{\cV_{2},\cV_{2}} = e_X, \qquad F_{\cV_{2},\cV_{2}}^{-1}\circ\cF(i_{\cV_{2}})\circ\varphi^{-1}= i_X.
 \end{equation*}
\end{thm}
\begin{proof}
Since the full subcategory $\cP^k\subseteq KL^k(\mathfrak{sl}_2)$ is tensor equivalent to $\cT_\zeta$ with $\zeta=e^{\pi i q/p}$ by Theorem \ref{thm:Tzeta_Pk_equiv}, the category $\cP^k$ satisfies the universal property of Theorem \ref{thm:tilt_univ_prop}. Thus there is a unique (up to natural isomorphism) tensor functor $\cG: \cP^k\rightarrow\cC$ sending $(\cV_{2}, e_{\cV_{2}}, i_{\cV_{2}})$ to $(X, e_X, i_X)$ as specified in the theorem statement. Then by Theorem \ref{thm:extend_to_right_exact}, $\cG$ extends uniquely to a right exact tensor functor $\cF: KL^k(\mathfrak{sl}_2)\rightarrow\cC$.
\end{proof}

In the setting of Theorem \ref{thm:KLk_univ_prop}, we would like to determine conditions under which the right exact tensor functor $\cF: KL^k(\mathfrak{sl}_2)\rightarrow\cC$ is additionally left exact, braided or braid-reversed, and preserves ribbon twists. The last three properties are easy to determine thanks to Proposition \ref{prop:braid_and_twist_from_V12}:
\begin{thm}\label{thm:univ_prop_F_braided}
In the setting of Theorem \ref{thm:KLk_univ_prop}, if $\cC$ is a braided tensor category and 
\begin{equation}\label{eqn:braided_RXX}
\cR_{X,X} =  e^{\pi i q/2p}\cdot\Id_{X\tens_\cC X} + e^{-\pi i q/2p}\cdot (i_X\circ e_X),
\end{equation}
respectively
\begin{equation}\label{eqn:braid-rev_RXX}
\cR_{X,X} =  e^{-\pi i q/2p}\cdot\Id_{X\tens_\cC X} + e^{\pi i q/2p}\cdot (i_X\circ e_X),
\end{equation}
then the tensor functor $\cF: KL^k(\fsl_2)\rightarrow\cC$ is braided, respectively braid-reversed. If in addition $\cC$ has a ribbon twist $\theta$ such that $\theta_X= e^{3\pi iq/2p}\cdot\Id_X$, then $\cF$ preserves twists.
\end{thm}
\begin{proof}
Using the braiding in $KL^k(\fsl_2)$ from Proposition \ref{prop:V12_self_braiding} and its inverse, we obtain
\begin{align*}
\cF(\cR_{\cV_2,\cV_2}^{\pm 1})\circ F_{\cV_2,\cV_2} & =\left(e^{\pm\pi i q/2p}\cdot\cF(\Id_{\cV_2\tens \cV_2}) + e^{\mp\pi i q/2p}\cdot \cF(i_{\cV_2})\circ \cF(e_{\cV_2})\right)\circ F_{\cV_2,\cV_2}\nonumber\\
& = F_{\cV_2,\cV_2}\circ\left(e^{\pm\pi iq/2p}\circ\Id_{X\tens_\cC X} + e^{\mp \pi iq/2p}\cdot(i_X\circ e_X)\right).
\end{align*}
Thus
\begin{equation*}
 \cF(\cR_{\cV_2,\cV_2})\circ F_{\cV_2,\cV_2}=F_{\cV_2,\cV_2}\circ\cR_{\cF(\cV_2),\cF(\cV_2)},
\end{equation*}
respectively
\begin{equation*}
 \cF(\cR_{\cV_2,\cV_2}^{-1})\circ F_{\cV_2,\cV_2}=F_{\cV_2,\cV_2}\circ\cR_{\cF(\cV_2),\cF(\cV_2)},
\end{equation*}
if and only if $\cR_{X,X}$ satisfies \eqref{eqn:braided_RXX}, respectively \eqref{eqn:braid-rev_RXX}. Then by Proposition \ref{prop:braid_and_twist_from_V12}, $\cF$ is braided if \eqref{eqn:braided_RXX} holds and braid-reversed if \eqref{eqn:braid-rev_RXX} holds. Similarly, $\cF$ preserves twists if
\begin{equation*}
\theta_X=e^{2\pi i h_2}\cdot\Id_X = e^{3\pi iq/2p}\cdot\Id_X,
\end{equation*}
since $\theta_{\cV_2} =e^{2\pi i L(0)}$.
\end{proof}

We now consider when the functor $\cF:KL^k(\fsl_2)\rightarrow\cC$ in Theorem \ref{thm:KLk_univ_prop} is exact. If $KL^k(\mathfrak{sl}_2)$ and $\cC$ were rigid, we could use the exact contravariant duality functor to show that right exactness of $\cF$ implies left exactness. However, as in \cite[Theorem 2.12]{ALSW}, contragredient modules only give $KL^k(\mathfrak{sl}_2)$ the weaker duality structure of a Grothendieck-Verdier category with dualizing object $\cV_{1}'$. This means that for any object $W$ in $KL^k(\mathfrak{sl}_2)$ there is a map $\varepsilon_W: W'\tens W\rightarrow\cV_{1}'$ (not to be confused with the evaluation in case $W$ happens to be rigid) such that for any homomorphism $f: X\tens W\rightarrow\cV_{1}'$ in $KL^k(\mathfrak{sl}_2)$, there is a unique homomorphism $\varphi: X\rightarrow W'$ such that the diagram
\begin{equation*}
\xymatrixcolsep{3pc}
\xymatrix{
X\tens W \ar[d]_{\varphi\tens\Id_W} \ar[rd]^{f} & \\
W' \tens W \ar[r]_(.55){\varepsilon_W} & \cV_{1}'\\
}
\end{equation*}
commutes. The homomorphisms $\varepsilon_W$ and $\varphi$ can be constructed from symmetries of intertwining operators. In particular, we can take $\varepsilon_W$ to be the unique homomorphism such that $\varepsilon_W\circ\cY_\tens=\cE_W$ where $\cY_\tens$ is the tensor product intertwining operator of type $\binom{W'\tens W}{W\,W'}$ and $\cE_W$ is the intertwining operator of type $\binom{\cV_{1}'}{W\,W'}$ defined by
\begin{equation*}
\cE_W=A_0(\Omega_0(Y_{W'}))\circ(\Id_{W'}\otimes\delta_W).
\end{equation*}
Here we use the notation of \cite[Equations 3.77 and 3.87]{HLZ2}, and $\delta_W: W\rightarrow W''$ is the natural isomorphism defined by
\begin{equation*}
\langle\delta_W(w), w'\rangle =\langle w',w\rangle
\end{equation*}
for $w\in W$, $w'\in W'$. Specifically,
\begin{align}\label{eqn:contragredient_EW}
\langle\cE_W(w',x)w, v\rangle & = \big\langle \delta_W(w), \Omega_0(Y_{W'})(e^{xL(1)} e^{\pi i L(0)} (x^{-L(0)})^2 w',x^{-1})v\big\rangle\nonumber\\
& = \big\langle Y_{W'}(v,-x^{-1}) e^{x L(1)} e^{\pi i L(0)}(x^{-L(0)})^2 w', w\big\rangle
\end{align}
for $w\in W$, $w'\in W'$, and $v\in\cV_1$. Using this formula, we prove:
\begin{lem}\label{eqn:contra_EW_surjective}
The homomorphism $\varepsilon_W: W'\tens W\rightarrow\cV_{1}'$ is surjective if and only if $W$ is not an object of the subcategory $KL_k(\mathfrak{sl}_2)\subseteq KL^k(\mathfrak{sl}_2)$ of $L_k(\mathfrak{sl}_2)$-modules.
\end{lem}
\begin{proof}
The homomorphism $\varepsilon_W$ is not surjective if and only if $\im\varepsilon_W$ is contained in the maximal proper submodule $\cL_{1}\subseteq\cV_{1}'$. Equivalently, since the tensor product intertwining operator of type $\binom{W'\tens W}{W'\,W}$ is surjective, the definitions imply $\varepsilon_W$ is not surjective if and only if
\begin{equation*}
\langle\cE_W(w',x)w,v\rangle =0
\end{equation*}
for all $v$ in the maximal proper submodule $\cL_{2p-1}\subseteq\cV_{1}$. By \eqref{eqn:contragredient_EW}, this is equivalent to $Y_W(v,x)w=0$ for all $v\in\cL_{2p-1}\subseteq\cV_{1}= V^k(\mathfrak{sl}_2)$. Since $\cL_{2p-1}$ is the maximal proper ideal of $V^k(\mathfrak{sl}_2)$, this is equivalent to $W$ being an $L_k(\mathfrak{sl}_2)$-module in $KL_k(\mathfrak{sl}_2)$.
\end{proof}

In the setting of Theorem \ref{thm:KLk_univ_prop}, we now assume that the (abelian) tensor category $\cC$ is also a Grothendieck-Verdier category with dualizing object $K$. By definition (see \cite{BD, ALSW}), this means there is a contravariant anti-equivalence $\cD$ of $\cC$ and a natural isomorphism
\begin{equation*}
\hom_\cC(X\tens_\cC W, K) \xrightarrow{\cong}\hom_\cC(X, \cD(W))
\end{equation*}
for all objects $W$, $X$ in $\cC$. That is, $\cD(W)$ satisfies the same universal property in $\cC$ as contragredient modules do in $KL^k(\mathfrak{sl}_2)$. The action of $\cD$ on morphisms can be defined using the universal property: given $f: X_1\rightarrow X_2$ in $\cC$, the morphism $\cD(f):\cD(X_2)\rightarrow\cD(X_1)$ is unique such that
\begin{equation*}
\xymatrixcolsep{5pc}
\xymatrix{
\cD(X_2)\tens_\cC X_1 \ar[r]^{\Id_{\cD(X_2)}\tens_\cC f} \ar[d]^{\cD(f)\tens_\cC\Id_{X_1}} & \cD(X_2)\tens_\cC X_2 \ar[d]^{\varepsilon_{X_2}}\\
\cD(X_1)\tens_\cC X_1 \ar[r]^{\varepsilon_{X_1}} & K\\
}
\end{equation*}
commutes. Note that $\cD$ is exact since it is an equivalence between the abelian category $\cC$ and its opposite category.
\begin{thm}\label{thm:right_exact_to_exact}
In the setting of Theorem \ref{thm:KLk_univ_prop}, let $\cC$ be a Grothendieck-Verdier category with dualizing object $K$ and contravariant anti-equivalence $\cD$. Assume also that the right exact tensor functor $\cF$ satisfies $\cF(\cV_{1}')\cong K$, and that $\cF(W)$ is either simple or $0$ for all simple modules $W$ in $KL^k(\mathfrak{sl}_2)$, with $\cF(\varepsilon_W)\neq 0$ whenever $\cF(W)\neq 0$. Then $\cF$ is exact.
\end{thm}
\begin{proof}
Fix an isomorphism $\psi:\cF(\cV_{1}')\rightarrow K$. Then for any object $W$ of $KL^k(\mathfrak{sl}_2)$, there is a unique morphism $\varphi_W: \cF(W')\rightarrow\cD(\cF(W))$ such that the diagram
\begin{equation*}
\xymatrixcolsep{4pc}
\xymatrix{
\cF(W')\tens_\cC \cF(W) \ar[r]^(.525){F_{W',W}} \ar[d]^{\varphi_W\tens_\cC\Id_{\cF(W)}} & \cF(W'\tens W) \ar[d]^{\psi\circ\cF(\varepsilon_W)} \\
\cD(\cF(W))\tens_\cC\cF(W) \ar[r]^(.6){\varepsilon_{\cF(W)}} & K\\
}
\end{equation*}
commutes. If $W$ is simple, then $W'\cong W$, and thus $\cF(W')=0$ if and only if $\cD(\cF(W))=0$. In case both are zero, $\varphi_W$ is trivially an isomorphism. Otherwise, $\varphi_W\neq 0$ since $\cF(\varepsilon_W)\neq 0$ by hypothesis; moreover, $\cF(W')\cong\cF(W)$ is simple by hypothesis, and then $\cD(\cF(W))$ is also simple since $\cD$ is an anti-equivalence. Thus in this case also, $\varphi_W$ is an isomorphism since it is a non-zero morphism between simple objects of $\cC$.

We will show that $\varphi_W$ is an isomorphism for all modules $W$ in $KL^k(\mathfrak{sl}_2)$ by induction on the length of $W$, but we first need to show that $\varphi_W$ determines a natural transformation. That is, we want to show that 
$$\varphi_{W_1}\circ\cF(f')=\cD(\cF(f))\circ\varphi_{W_2}$$
for all morphisms $f: W_1\rightarrow W_2$ in $KL^k(\mathfrak{sl}_2)$. For this, the uniqueness assertion in the universal property of $\cD(\cF(W_1))$ implies it is enough to show
\begin{equation*}
\varepsilon_{\cF(W_1)}\circ[(\varphi_{W_1}\circ\cF(f'))\tens_\cC\Id_{\cF(W_1)}]=\varepsilon_{\cF(W_1)}\circ[(\cD(\cF(f))\circ\varphi_{W_2})\tens_\cC\Id_{\cF(W_1)}].
\end{equation*}
Indeed, the definitions imply
\begin{align*}
\varepsilon_{\cF(W_1)} & \circ[(\varphi_{W_1}\circ\cF(f'))\tens_\cC\Id_{\cF(W_1)}]  =\psi\circ\cF(\varepsilon_{W_1})\circ F_{W_1',W_1}\circ(\cF(f')\tens_\cC\Id_{\cF(W_1)})\nonumber\\
& =\psi\circ\cF(\varepsilon_{W_1}\circ(f'\tens\Id_{W_1}))\circ F_{W_2',W_1} =\psi\circ\cF(\varepsilon_{W_2})\circ(\Id_{W_2'}\tens f))\circ F_{W_2',W_1}\nonumber\\
& = \psi\circ\cF(\varepsilon_{W_2})\circ F_{W_2',W_2}\circ(\Id_{\cF(W_2')}\tens_\cC\cF(f)) =\varepsilon_{\cF(W_2)}\circ(\varphi_{W_2}\tens_\cC\cF(f))\nonumber\\
& =\varepsilon_{\cF(W_1)}\circ [(\cD(\cF(f))\circ\varphi_{W_2})\tens_\cC\Id_{\cF(W_1)}],
\end{align*}
as desired.

So far, we know that $\varphi_W$ is an isomorphism when the length of $W$ is either zero or one. Now when the length of $W$ is greater than one, there is some exact sequence
\begin{equation*}
0\longrightarrow X_1\xrightarrow{f_1} W\xrightarrow{f_2} X_2\rightarrow 0
\end{equation*}
where the lengths of $X_1$ and $X_2$ are strictly less than the length of $W$, and we assume by induction that $\varphi_{X_1}$ and $\varphi_{X_2}$ are isomorphisms. Using the naturality of $\varphi$, the right exactness of $\cF$, and the exactness of $\cD$ and the contragredient functor on $KL^k(\mathfrak{sl}_2)$, we then get a commutative diagram
\begin{equation*}
\xymatrixcolsep{2pc}
\xymatrix{
& \cF(X_2') \ar[rr]^{\cF(f_2')} \ar[d]^{\varphi_{X_2}} && \cF(W') \ar[rr]^{\cF(f_1')} \ar[d]^{\varphi_W} && \cF(X_1') \ar[r] \ar[d]^{\varphi_{X_1}} & 0\\
0 \ar[r] & \cD(\cF(X_2)) \ar[rr]^{\cD(\cF(f_2))} && \cD(\cF(W)) \ar[rr]^{\cD(\cF(f_1))} && \cD(\cF(X_1)) & \\
}
\end{equation*}
with exact rows. The short five lemma diagram chase now shows that $\varphi_W$ is an isomorphism, completing the induction to show that $\varphi_W$ is an isomorphism for all $W$.

We can now use the natural isomorphisms $\varphi_W$ to prove that $\cF$ is exact. Since $\cF$ is already right exact, it is enough to prove that $\cF(f): \cF(W_1)\rightarrow \cF(W_2)$ is injective whenever $f: W_1\rightarrow W_2$ is injective. Indeed, by right exactness, $\cF(f'):\cF(W_2')\rightarrow\cF(W_1')$ is surjective, and then so is
\begin{equation*}
\cD(\cF(f)) = \varphi_{W_1}\circ\cF(f')\circ\varphi_{W_2}^{-1}.
\end{equation*}
Thus $\cF(f)$ is injective because $\cD$ is an anti-equivalence.
\end{proof}

\section{Applications of the universal property}\label{sec:applications}

We now apply the universal property of $KL^k(\fsl_2)$ from Theorem \ref{thm:KLk_univ_prop} to obtain interesting tensor functors in several examples.

\subsection{Classifying \texorpdfstring{$KL^k(\fsl_2)$}{KLk(sl2)} up to tensor equivalence}\label{subsec:classify}

Here, we take $\cC$ in Theorem \ref{thm:KLk_univ_prop} to be $KL^{k'}(\fsl_2)$ where $k'=-2+p'/q'$ is another admissible level for $\fsl_2$. Note that $KL^k(\fsl_2)$ and $KL^{k'}(\fsl_2)$ satisfy the same universal property if $\cV_2^k$ and $\cV_2^{k'}$ have the same intrinsic dimension, which occurs if and only if $e^{\pi i q'/p'} = e^{\pm\pi i q/p}$. This latter condition holds if and only if $p'=p$ and $q'\in\pm q+2p\ZZ$. Thus from Theorem \ref{thm:KLk_univ_prop}, we obtain:
\begin{thm}\label{thm:KLk_tens_equiv}
For any admissible level $k$, $KL^{k}(\fsl_2)$ is tensor equivalent to $KL^{-2+p/q}(\fsl_2)$ for some $p\in\ZZ_{\geq 2}$ and some $q$ relatively prime to $p$ such that $1\leq q\leq p-1$. Specifically, for all such $p$ and $q$,
\begin{equation*}
KL^{-2+p/q}(\fsl_2)\cong
KL^{-2+p/(\pm q+2np)}(\fsl_2) 
\end{equation*}
as tensor categories, where $n\geq 0$ in the $+$ case, and $n\geq 1$ in the $-$ case.
\end{thm}

Not all the tensor equivalences in the previous theorem are braided tensor equivalences, but it is easy to determine which ones are using Theorems \ref{thm:KLk_braidings} and \ref{thm:univ_prop_F_braided}. Before doing so, we introduce some notation. For $k$ an admissible level for $\fsl_2$, let us now use $KL^k(\fsl_2)_{\pm}$ to denote the braided tensor category with ribbon twist given in Theorem \ref{thm:KLk_braidings}(1), with the $\pm$ indicating which sign to take for $\theta_{\cV_2}$. In particular, $KL^k(\fsl_2)_+$ denotes the tensor category $KL^k(\fsl_2)$ with its official braiding (recall Proposition \ref{prop:V12_self_braiding}) and official ribbon twist $\theta=e^{2\pi i L(0)}$. We then use $KL^k(\fsl_2)^{\mathrm{tw}}_\pm$ to denote the tensor category $KL^k(\fsl_2)$ equipped with the braiding and twist of Theorem \ref{thm:KLk_braidings}(2), $KL^k(\fsl_2)^{\mathrm{rev}}_\pm$ for $KL^k(\fsl_2)$ equipped with the braiding and twist of Theorem \ref{thm:KLk_braidings}(3), and $KL^k(\fsl_2)^{\mathrm{tw},\mathrm{rev}}_\pm$ for $KL^k(\fsl_2)$ equipped with the braiding and twist of Theorem \ref{thm:KLk_braidings}(4).
\begin{thm}\label{thm:KLk_class_braided}
Let $k=-2+p/q$ be an admissible level for $\fsl_2$ such that $\gcd(p,q)=1$ and $1\leq q\leq p-1$. Then for $n\geq 0$,
\begin{equation*}
KL^{-2+p/(q+2np)}(\fsl_2)_+\cong\left\lbrace\begin{array}{lll}
KL^k(\fsl_2)_+ & \text{if} & n\equiv 0\mod 2\\
KL^k(\fsl_2)^{\mathrm{tw}}_- & \text{if} & n\equiv 1\mod 2\\
\end{array}\right. ,
\end{equation*}
and for $n\geq 1$,
\begin{equation*}
KL^{-2+p/(-q+2np)}(\fsl_2)_+\cong\left\lbrace\begin{array}{lll}
KL^k(\fsl_2)^{\mathrm{rev}}_+ & \text{if} & n\equiv 0\mod 2\\
KL^k(\fsl_2)^{\mathrm{tw},\mathrm{rev}}_- & \text{if} & n\equiv 1\mod 2\\
\end{array}\right. ,
\end{equation*}
as braided tensor categories with ribbon twist.
\end{thm}
\begin{proof}
By Theorem \ref{thm:univ_prop_F_braided}, the tensor equivalence $KL^{-2+p/(\pm q+2np)}\rightarrow KL^k(\fsl_2)$ is braided if
\begin{align*}
\cR_{\cV_2^k,\cV_2^k} & =  e^{\pi i (\pm q+2np)/2p}\cdot\Id_{\cV_2^k\tens\cV_2^k} + e^{-\pi i (\pm q+2np)/2p}\cdot (i_{\cV_2^k}\circ e_{\cV_2^k})\nonumber\\
& = (-1)^n\left( e^{\pm\pi i q/2p}\cdot\Id_{\cV_2^k\tens\cV_2^k} + e^{\mp\pi i q/2p}\cdot (i_{\cV_2^k}\circ e_{\cV_2^k})\right).
\end{align*}
Thus the tensor equivalence is braided in the $+q$ and $n$ even case if we equip $KL^k(\fsl_2)$ with the first braiding in Theorem \ref{thm:KLk_braidings}. Then in the $+q$ and $n$ odd case, we need to equip $KL^k(\fsl_2)$ with the second braiding in Theorem \ref{thm:KLk_braidings}; in the $-q$ and $n$ even case, we need to equip $KL^k(\fsl_2)$ with the third braiding; and in the $-q$ and $n$ odd case, we need to equip $KL^k(\fsl_2)$ with the fourth braiding.

Similarly, the tensor equivalence preserves twists if 
\begin{equation*}
\theta_{\cV_2^k}= e^{3\pi i(\pm q+2np)/2p}\cdot\Id_{\cV_2^k} = (-1)^n e^{\pm 3\pi iq/2p}\cdot\Id_{\cV_2^k}.
\end{equation*}
Comparing with Theorem \ref{thm:KLk_braidings}, we need to give $KL^k(\fsl_2)$ the $+$ twist if $n$ is even and the $-$ twist if $n$ is odd.
\end{proof}

We now return to the cocycle twist $KL^k(\fsl_2)^\tau$ introduced in Section \ref{subsec:cocycle_twist}. As in Theorem \ref{thm:KLk_braidings}, we equip $KL^k(\fsl_2)^\tau$ with the braiding and twist characterized by
\begin{equation*}
\cR^\tau_{\cV_2,\cV_2} = e^{\pi i/2}\cdot\cR_{\cV_2,\cV_2},\qquad\theta_{\cV_2}^\tau =e^{\pi i/2}\cdot\theta_{\cV_2},
\end{equation*}
where $\cR_{\cV_2,\cV_2}$ and $\theta_{\cV_2}$ are the official braiding and twist of Theorem \ref{thm:KLk_braidings}(1).
\begin{thm}\label{thm:KLk_tau_equiv}
Let $k=-2+p/q$ for relatively prime $p\in\ZZ_{\geq 2}$ and $q\in\ZZ_{\geq 1}$. Then
\begin{equation*}
KL^k(\fsl_2)^\tau\cong KL^{-2+p/(p+q)}(\fsl_2)_-
\end{equation*}
as braided tensor categories with ribbon twists.
\end{thm}
\begin{proof}
As noted at the end of Section \ref{subsec:cocycle_twist}, $\cV_2$ has the same intrinsic dimension $e^{\pi iq/p}+e^{-\pi iq/p}$ in $KL^k(\fsl_2)^\tau$ as it does in the tensor category $KL^{-2+p/(p+q)}(\fsl_2)$. Thus by Theorem \ref{thm:KLk_univ_prop}, there is a unique right exact tensor functor
\begin{equation*}
\cF: KL^{-2+p/(p+q)}(\fsl_2)\longrightarrow KL^k(\fsl_2)^\tau
\end{equation*}
sending the $\cV_2$ in $KL^{-2+p/(p+q)}(\fsl_2)$ to the $\cV_2$ in $KL^k(\fsl_2)^\tau$. Note that since the cocycle twist $\tau$ only affects associativity isomorphisms, we can also view $\cF$ as a tensor functor from $KL^{-2+p/(p+q)}(\fsl_2)^\tau$ to $KL^k(\fsl_2)$. Similarly, Theorem \ref{thm:KLk_univ_prop} yields a right exact tensor functor
\begin{equation*}
\cG: KL^k(\fsl_2)\longrightarrow KL^{-2+p/(p+q)}(\fsl_2)^\tau
\end{equation*}
which we can equivalently view as a tensor functor from $KL^k(\fsl_2)^\tau$ to $KL^{-2+p/(p+q)}(\fsl_2)$. Then $\cF\circ\cG$ is a right exact tensor endofunctor of $KL^k(\fsl_2)$ which preserves $\cV_2$. Since the identity endofunctor also satisfies these properties, Theorem \ref{thm:KLk_univ_prop} implies $\cF\circ\cG\cong\Id_{KL^k(\fsl_2)}$, which we can equivalently view as the identity on $KL^k(\fsl_2)^\tau$. Similarly, $\cG\circ\cF\cong\Id_{KL^{-2+p/(p+q)}(\fsl_2)}$, so $\cF$ is a tensor equivalence.

We now check braidings. First note that the braiding of $KL^k(\fsl_2)$ in Theorem \ref{thm:KLk_braidings}(1) is given in terms of $f_{\cV_2}$, which is defined using the evaluation and coevaluation for $\cV_2$ in $KL^k(\fsl_2)$. However, we need to change either the evaluation or coevaluation by a sign to get the correct evaluation and coevaluation for $\cV_2$ in $KL^k(\fsl_2)^\tau$. Thus
\begin{equation*}
\cR_{\cV_{2},\cV_{2}} = e^{\pi i q/2p}\cdot\Id_{\cV_{2}\tens\cV_{2}} - e^{-\pi i q/2p}\cdot (i_{\cV_2}\circ e_{\cV_2})
\end{equation*}
is the braiding in $KL^k(\fsl_2)$ expressed in terms of the evaluation and coevaluation for $\cV_2$ in $KL^k(\fsl_2)^\tau$. Then
\begin{align*}
\cR_{\cV_{2},\cV_{2}}^\tau & =e^{\pi i/2}\left(e^{\pi i q/2p}\cdot\Id_{\cV_{2}\tens\cV_{2}} - e^{-\pi i q/2p}\cdot (i_{\cV_2}\circ e_{\cV_2})\right)\nonumber\\
& = e^{\pi i (p+q)/2p}\cdot\Id_{\cV_{2}\tens\cV_{2}} + e^{-\pi i (p+q)/2p}\cdot (i_{\cV_2}\circ e_{\cV_2}),
\end{align*}
so $\cF: KL^{-2+p/(p+q)}(\fsl_2)\rightarrow KL^k(\fsl_2)^\tau$ is a braided tensor equivalence by Theorem \ref{thm:univ_prop_F_braided}. Finally, since
\begin{equation*}
\theta_{\cV_2}^\tau = e^{\pi i/2} e^{3\pi i q/2p}\cdot\Id_{\cV_2} = -e^{3\pi i (p+q)/2p}\cdot\Id_{\cV_2},
\end{equation*}
$\cF$ also preserves twist if we equip $KL^{-2+p/(p+q)}(\fsl_2)$ with the $-$ twist in Theorem \ref{thm:KLk_braidings}(1).
\end{proof}

We can combine Theorems \ref{thm:KLk_class_braided} and \ref{thm:KLk_tau_equiv} to obtain:
\begin{cor}\label{cor:class_KLk_tau_braid}
Let $k=-2+p/q$ be an admissible level for $\fsl_2$ such that $\gcd(p,q)=1$ and $1\leq q\leq p-1$. Then for $n\geq 1$,
\begin{equation*}
KL^{-2+p/(q+(2n-1)p)}(\fsl_2)^\tau\cong\left\lbrace\begin{array}{lll}
KL^k(\fsl_2)_- & \text{if} & n\equiv 0\mod 2\\
KL^k(\fsl_2)^{\mathrm{tw}}_+ & \text{if} & n\equiv 1\mod 2\\
\end{array}\right. 
\end{equation*}
and 
\begin{equation*}
KL^{-2+p/(-q+(2n-1)p)}(\fsl_2)^\tau\cong\left\lbrace\begin{array}{lll}
KL^k(\fsl_2)^{\mathrm{rev}}_- & \text{if} & n\equiv 0\mod 2\\
KL^k(\fsl_2)^{\mathrm{tw},\mathrm{rev}}_+ & \text{if} & n\equiv 1\mod 2\\
\end{array}\right. 
\end{equation*}
as braided tensor categories with ribbon twists.
\end{cor}

Note from Theorem \ref{thm:KLk_class_braided} and \ref{cor:class_KLk_tau_braid} that if $k=-2+p/q$ for relatively prime $p\geq 2$ and $1\leq q\leq p-1$, then $KL^k(\fsl_2)$ with any of its braidings and twists (eight possibilities in total) is equivalent to either $KL^{k'}(\fsl_2)$ or $KL^{k'}(\fsl_2)^\tau$ equipped with its official braiding and twist, for suitable admissible levels $k'$. Corollary \ref{cor:class_KLk_tau_braid} also shows that $KL^{-2+p/q}(\fsl_2)$ is tensor equivalent to $KL^{-2+p/(p-q)}(\fsl_2)$ for $1\leq q\leq p-1$; combining this with Theorem \ref{thm:KLk_class_braided} again, we conclude:
\begin{cor}
For any admissible level $k$, the tensor category $KL^k(\fsl_2)$ is equivalent to either $KL^{-2+p/q}(\fsl_2)$ or $KL^{-2+p/q}(\fsl_2)^\tau$ for some relatively prime $p,q\in\ZZ_{\geq 1}$ such that $p\geq 2$ and $1\leq q\leq\frac{p}{2}$.
\end{cor}

Finally, we can also classify the tensor categories $KL_k(\fsl_2)$ of $L_k(\fsl_2)$-modules up to braided tensor equivalence.
\begin{thm}
All (braided) tensor equivalences in Theorems \ref{thm:KLk_tens_equiv} and \ref{thm:KLk_class_braided} hold with $KL^k$ replaced everywhere by $KL_k$.
\end{thm}
\begin{proof}
Set $k=-2+p/q$ for relatively prime $p$ and $q$ such that $p\geq 2$ and $1\leq q\leq p-1$, and set $k'=-2+p/(\pm q+2np)$ for some $n\in\ZZ_{\geq 0}$. Let $\cF: KL^k(\fsl_2)\rightarrow KL^{k'}(\fsl_2)$ be the tensor equivalence of Theorem \ref{thm:KLk_tens_equiv}, such that $\cF(\cV_2^k)\cong\cV_2^{k'}$. Since the tensor products of projective objects in $KL^k(\fsl_2)$ are the same as those in $KL^{k'}(\fsl_2)$, the proof of Proposition \ref{prop:F_essential_surj} shows that $\cF(\cP_r^k)\cong\cP_r^{k'}$ for all $r\geq 0$, where $\cP_r^k$ and $\cP_r^{k'}$ are the indecomposable projective objects of $KL^k(\fsl_2)$ and $KL^{k'}(\fsl_2)$, respectively.

We claim that $\cF(\cL_r^k)\cong\cL_r^{k'}$ for $1\leq r\leq p-1$, so that the image of $\cF\vert_{KL_k(\fsl_2)}$ is $KL_{k'}(\fsl_2)$. Indeed, consider the right exact sequence
\begin{equation*}
\cP_{2p-r}^k\xrightarrow{g_k} \cV_r^k\xrightarrow{f_k} \cL_r^k\longrightarrow 0
\end{equation*}
in $KL^k(\fsl_2)$, where $g_k$ maps the projective cover $\cP_{2p-r}^k$ onto the maximal proper submodule of $\cV_r^k$, which is isomorphic to $\cL_{2p-r}^k$ by Theorem \ref{thm:gen_Verma_structure}. Because $\cF$ is right exact, we get a corresponding commutative diagram with right exact rows in $KL^{k'}(\fsl_2)$:
\begin{equation*}
\xymatrixcolsep{3pc}
\xymatrix{
\cF(\cP_{2p-r}^k) \ar[d]^\cong \ar[r]^{\cF(g_k)} & \cF(\cV_r^k) \ar[d]^\cong \ar[r]^{\cF(f_k)} & \cF(\cL_r^k) \ar@{-->}[d]^{\varphi_r} \ar[r] & 0\\
\cP_{2p-r}^{k'} \ar[r]^{g_{k'}} & \cV_r^{k'} \ar[r]^{f_{k'}} & \cL_r^{k'} \ar[r] & 0\\
}
\end{equation*}
Here $\varphi_r$ is induced by the universal property of cokernels and is an isomorphism because it is a non-zero map between simple modules. Since $KL_k(\fsl_2)$ and $KL_{k'}(\fsl_2)$ are both semisimple with simple objects $\cL_r^k$ and $\cL_r^{k'}$ for $1\leq r\leq p-1$, this shows that $\cF$ restricts to an equivalence of categories between $KL_k(\fsl_2)$ and $KL_{k'}(\fsl_2)$.

To show that $\cF: KL_k(\fsl_2)\rightarrow KL_{k'}(\fsl_2)$ is also a tensor functor, note that by Theorem \ref{thm:inclusion_is_lax_monoidal}, the tensor products and associativity isomorphisms in $KL_k(\fsl_2)$ and $KL_{k'}(\fsl_2)$ are the restrictions of those in $KL^k(\fsl_2)$ and $KL^{k'}(\fsl_2)$. Thus the natural isomorphism
\begin{equation*}
F:\tens\circ(\cF\times\cF)\longrightarrow\cF\circ\tens
\end{equation*}
associated to $\cF$ is compatible with the associativity in $KL_k(\fsl_2)$ and $KL_{k'}(\fsl_2)$. For the unit isomorphisms, there is an isomorphism $\Phi:\cF(\cV_1^k)\rightarrow\cV_1^{k'}$ in $KL^{k'}(\fsl_2)$ such that $\Phi$ and $F$ are compatible with the unit isomorphisms in $KL^k(\fsl_2)$ and $KL^{k'}(\fsl_2)$. We can then define the isomorphism $\varphi_1: \cF(\cL_1^k)\rightarrow\cL_1^{k'}$ of the preceding paragraph such that the diagram
\begin{equation}\label{eqn:phi1_def}
\begin{matrix}
\xymatrixcolsep{3pc}
\xymatrix{
\cF(\cV_1^k) \ar[r]^{\cF(f_k)} \ar[d]^{\Phi} & \cF(\cL_1^k) \ar[d]^{\varphi_1} \\
\cV_1^{k'} \ar[r]^{f_{k'}} & \cL_1^{k'}\\
}
\end{matrix}
\end{equation}
commutes, where $f_k$ and $f_{k'}$ are the quotient maps from the universal affine vertex operator algebras to their simple quotients. Then for any object $W$ of $KL_k{\fsl_2}$, we need to show that
\begin{equation}\label{eqn:left_unit_compat}
l_{\cF(W)}\circ(\varphi_1\tens\Id_{\cF(W)}) = \cF(l_W)\circ F_{\cL_1^k, W}.
\end{equation}
In fact,
\begin{align}\label{eqn:unit_compat_calc}
l_{\cF(W)} \circ(\varphi_1\tens\Id_{\cF(W)})\circ(\cF(f_k)\tens\Id_{\cF(W)})  = l_{\cF(W)}\circ (f_{k'}\tens\Id_{\cF(W)})\circ(\Phi\tens\Id_{\cF(W)})
\end{align}
by \eqref{eqn:phi1_def}, and then $l_{\cF(W)}\circ(f_{k'}\tens\Id_{\cF(W)})$ is simply the left unit isomorphism $l_{\cF(W)}^{KL^{k'}(\fsl_2)}$ in $KL^{k'}(\fsl_2)$, since by \eqref{eqn:unit_isos}, the left unit isomorphisms in $KL^{k'}(\fsl_2)$ and $KL_{k'}(\fsl_2)$ are defined by the vertex operator action of the universal, respectively simple, affine vertex operator algebra on $\cF(W)$. Then because $\Phi$ and $F$ are compatible with left unit isomorphisms, the right side of \eqref{eqn:unit_compat_calc} becomes
\begin{align*}
\cF(l_W^{KL^k(\fsl_2)})\circ F_{\cV_1^k,W} & =\cF(l_W)\circ\cF(f_k\tens\Id_W)\circ F_{\cV_1^k,W}\nonumber\\
&=\cF(l_W)\circ F_{\cL_1^k,W}\circ(\cF(f_k)\tens\Id_{\cF(W)}).
\end{align*}
This proves \eqref{eqn:left_unit_compat} since $\cF(f_k)\tens\Id_{\cF(W)}$ is surjective. Similarly, $\varphi_1$ and $F$ are compatible with the right unit isomorphisms of $KL_k(\fsl_2)$ and $KL_{k'}(\fsl_2)$.

We have now proved that $F$ defines a tensor equivalence from $KL_k(\fsl_2)$ to $KL_{k'}(\fsl_2)$. If moreover $KL^{k}(\fsl_2)$ is equipped with braiding and twist such that $\cF: KL^k(\fsl_2)\rightarrow KL^{k'}(\fsl_2)$ is a braided tensor equivalence and preserves twists, then we can equip $KL_k(\fsl_2)$ with the corresponding braiding and twist so that $KL_k(\fsl_2)\cong KL_{k'}(\fsl_2)$ as braided tensor categories with twists as well.
\end{proof}

\begin{rem}
Tensor equivalences for $KL_k(\fsl_2)$ can also be deduced from \cite[Theorem~$A_{\ell}$]{KWe}, though one first needs to check that the intrinsic dimension of $\cL_2$ in the rigid semisimple tensor category $KL_k(\fsl_2)$ is the same as that of $\cV_2$ in $KL^k(\fsl_2)$.
\end{rem}

\subsection{A weak Kazhdan-Lusztig correspondence}\label{subsec:weak_KL}

We now take the tensor category $\cC$ of Theorem \ref{thm:KLk_univ_prop} to be the quantum group category $\cC(\zeta,\fsl_2)$ at a root of unity. Similar to Kazhdan and Lusztig's results in \cite{KL3, KL4} at irrational and negative rational shifted levels, we will obtain a tensor functor $\cF: KL^k(\fsl_2)\rightarrow\cC(\zeta,\fsl_2)$ for $k$ an admissible level. However, unlike in \cite{KL3, KL4}, $\cF$ cannot be an equivalence because $\cC(\zeta,\fsl_2)$ is rigid, while $KL^k(\fsl_2)$ is not. Thus we call the functor of the next theorem a \textit{weak Kazhdan-Lusztig correspondence}:
\begin{thm}\label{thm:weak_KL_correspondence}
Let $k=-2+p/q$ for relatively prime $p\in\ZZ_{\geq 2}$ and $q\in\ZZ_{\geq 1}$, and let $\zeta=e^{\pi iq/p}$. Then there is an exact essentially surjective tensor functor $\cF: KL^k(\fsl_2)\rightarrow\cC(\zeta,\fsl_2)$ such that $\cF(\cP_r)\cong T_{r-1}$ for $r\in\ZZ_{\geq 1}$.
\end{thm}
\begin{proof}
Since the standard object $\mathbf{X}$ in $\cC(\zeta,\fsl_2)$ has the same intrinsic dimension as $\cV_2$ in $KL^k(\fsl_2)$, the existence and right exactness of the tensor functor $\cF$ follows immediately from Theorem \ref{thm:KLk_univ_prop}, and the proof of Proposition \ref{prop:F_essential_surj} shows that $\cF(\cP_r)\cong T_{r-1}$ for $r\in\ZZ_{\geq 1}$, where as previously, $T_\lambda$ denotes the indecomposable tilting module for quantum $\fsl_2$ of highest weight $\lambda$. Note also that $\cF$ restricts to an equivalence between the subcategory $\cP^k\subseteq KL^k(\fsl_2)$ of projective objects and the subcategory $\cT_\zeta\subseteq\cC(\zeta,\fsl_2)$ of tilting modules, since $\cP^k\cong\cT_\zeta$ by Theorem \ref{thm:Tzeta_Pk_equiv}, and the identity is the only additive tensor endofunctor of $\cT_\zeta$ that preserves $\mathbf{X}$ by Theorem \ref{thm:tilt_univ_prop}.

To show that $\cF$ is exact, we need to verify the conditions of Theorem \ref{thm:right_exact_to_exact}. First, because $\cC(\zeta,\fsl_2)$ is rigid, it is a Grothendieck-Verdier category whose dualizing object is its unit object $L_0$. Thus we need to show that $\cF(\cV_1')\cong L_0$. To do so, note that by \eqref{eqn:Prs_exact_seq} and Proposition \ref{prop:Pr_self_contra}, there is a right exact sequence
\begin{equation*}
\cP_{2p+1} \longrightarrow \cP_{2p-1} \longrightarrow \cV_{1}'\longrightarrow 0.
\end{equation*}
in $KL^k(\fsl_2)$. Since $\cF$ is right exact and maps $\cP_r$ to $T_{r-1}$, we get a right exact sequence
\begin{equation*}
T_{2p} \longrightarrow T_{2p-2} \longrightarrow \cF(\cV_1')\longrightarrow 0
\end{equation*}
in $\cC(\zeta,\fsl_2)$, where the first arrow is non-zero because $\cF$ restricts to an equivalence between $\cP^k$ and $\cT_\zeta$. Consulting Theorem \ref{thm:tilting_structures}(3), we see that $L_0$ is the cokernel of the unique (up to scale) non-zero map $T_{2p}\rightarrow T_{2p-2}$, so $\cF(\cV_1')\cong L_0$, as required.

We also need to determine how $\cF$ acts on the simple objects of $KL^k(\fsl_2)$. For $1\leq r\leq p-1$, we have a right exact sequence
\begin{equation*}
\cP_{2p-r}\longrightarrow\cV_r\longrightarrow\cL_r\longrightarrow 0
\end{equation*}
in $KL^k(\fsl_2)$, which becomes a right exact sequence
\begin{equation*}
T_{2p-r-1}\longrightarrow T_{r-1}\longrightarrow \cF(\cL_r)\longrightarrow 0
\end{equation*}
in $\cC(\zeta, \fsl_2)$, where the first arrow is non-zero because $\cF$ restricts to an equivalence between $\cP^k$ and $\cT_\zeta$. Since $T_{r-1}= L_{r-1}$ and $T_{2p-r-1}= P_{r-1}$ by Theorem \ref{thm:tilting_structures}, the map $T_{2p-r-1}\rightarrow T_{r-1}$ is the surjection from the projective cover to its simple quotient, and it follows that $\cF(\cL_r)=0$ for $1\leq r\leq p-1$.

Next, for $n\in\ZZ_{\geq 1}$, we have $\cL_{np} =\cP_{np}$, so $\cF(\cL_{np})\cong P_{np-1}= L_{np-1}$. Then for $n\in\ZZ_{\geq 1}$ and $1\leq r\leq p-1$, we have a right exact sequence
\begin{equation*}
\cP_{np-r}\oplus\cP_{(n+2)p-r}\longrightarrow\cP_{np+r}\longrightarrow\cL_{np+r}\longrightarrow 0
\end{equation*}
by Theorem \ref{thm:Prs}, where the first arrow is the sum of two non-zero maps $\cP_{(n+1\pm 1)p-r}\rightarrow\cP_{np+r}$. Since $\cF$ is right exact and restricts to an equivalence between $\cP^k$ and $\cT_\zeta$, we get a right exact sequence
\begin{equation*}
T_{np-r-1}\oplus T_{(n+2)p-r-1} \longrightarrow T_{np+r-1}\longrightarrow \cF(\cL_{np+r})\longrightarrow 0
\end{equation*}
in $\cC(\zeta,\fsl_2)$ where the first arrow is the sum of two non-zero maps $T_{(n+1\pm 1)np-r-1}\rightarrow T_{np-r-1}$. Using Theorem \ref{thm:tilting_structures}, the image of the first map in the above right exact sequence is the maximal proper submodule of $T_{np+r-1}=P_{np-r-1}$, and thus $\cF(\cL_{np+r})\cong L_{np-r-1}$.

We have now shown that $\cF(\cL_r)$ is either simple or $0$ for all $r\in\ZZ_{\geq 1}$. In particular, $\cF(\cL_r)$ is non-zero for $r\geq p$, in which case the map $\varepsilon_{\cL_r}: \cL_r\boxtimes\cL_r\rightarrow\cV_1'$ is surjective by Lemma \ref{eqn:contra_EW_surjective}. Then because $\cF$ is right exact, $\cF(\varepsilon_{\cL_r})$ is also surjective and thus non-zero for $r\geq p$. It now follows from Theorem \ref{thm:right_exact_to_exact} that $\cF: KL^k(\fsl_2)\rightarrow\cC(\zeta,\fsl_2)$ is exact.

Finally, to prove that $\cF$ is essentially surjective, we note that because $\cC(\zeta,\fsl_2)$ has enough projectives, every object of $\cC(\zeta,\fsl_2)$ is the cokernel of some morphism between projective objects. Moreover, every projective object of $\cC(\zeta,\fsl_2)$ is an object of $\cT_\zeta$ from Theorem \ref{thm:tilt_univ_prop}, and $\cF$ restricts to an equivalence between $\cP^k$ and $\cT_\zeta$. It follows that for any object $X$ of $\cC(\zeta,\fsl_2)$, there is a right exact sequence
\begin{equation*}
\cF(\mathcal{Q})\xrightarrow{\cF(f)} \cF(\cP)\longrightarrow X\longrightarrow 0
\end{equation*}
in $\cC(\zeta,\fsl_2)$ such that $\cP$ and $\mathcal{Q}$ are projective in $KL^k(\fsl_2)$. Since $\cF$ is right exact, it follows that $X\cong\cF(\mathrm{Coker}\,f)$, and thus $\cF$ is essentially surjective.
\end{proof}

\begin{rem}
It is not so meaningful to consider whether the functor $\cF$ of Theorem \ref{thm:weak_KL_correspondence} is braided, since just like $KL^k(\fsl_2)$, $\cC(\zeta,\fsl_2)$ admits four braidings (see for example \cite[Remark 3.1]{GN}), and none of them seems to be more canonical than the others. However, it is clear from \cite[Lemma 6.4]{GN} and Theorem \ref{thm:univ_prop_F_braided} that $\cF$ is braided with respect to some choice of braiding on $\cC(\zeta,\fsl_2)$ (though which one precisely depends on one's convention for taking square roots of $\zeta$ as well as on the denominator $q$ of the admissible level $k$).
\end{rem}

Although the functor $\cF$ in Theorem \ref{thm:weak_KL_correspondence} is essentially surjective, it is neither full nor faithful. We now characterize the lack of faithfulness in the next two results:
\begin{lem}\label{lem:weak_KL_lack_of_faithful_objects}
In the setting of Theorem \ref{thm:weak_KL_correspondence}, $\cF(W)=0$ for an object $W$ of $KL^k(\fsl_2)$ if and only if $W$ is an object of $KL_k(\fsl_2)$.
\end{lem}
\begin{proof}
In the proof of Theorem \ref{thm:weak_KL_correspondence}, we showed that $\cF(\cL_r)=0$ for $1\leq r\leq p-1$. Thus because $KL_k(\fsl_2)$ is semisimple with simple objects $\cL_r$ for $1\leq r\leq p-1$, and because $\cF$ is additive, it follows that $\cF(W)=0$ if $W$ is an object of $KL_k(\fsl_2)$.

Conversely, suppose $\cF(W)=0$. We will show that $W$ is an object of $KL_k(\fsl_2)$ by induction on the length $\ell(W)$. The base case $\ell(W)=0$ is trivial, and the $\ell(W)=1$ case follows from the proof of Theorem \ref{thm:weak_KL_correspondence}. Now suppose $\ell(W)>1$; then $W$ has some simple quotient $L$, and we have an exact sequence
\begin{equation}\label{eqn:F(W)=0_exact_seq}
0\longrightarrow K \longrightarrow W\longrightarrow L\longrightarrow 0
\end{equation}
for some maximal proper submodule $K\subseteq W$. Since $\cF$ is exact by Theorem \ref{thm:weak_KL_correspondence}, and since $\cF(W)=0$, we get $\cF(K)=\cF(L)=0$. Thus by induction on length, $L\cong\cL_r$ for some $1\leq r\leq p-1$, and $K$ is a finite direct sum of such $\cL_r$.

To complete the argument that $W$ is an object of $KL_k(\fsl_2)$, it remains to show that the exact sequence \eqref{eqn:F(W)=0_exact_seq} splits. Indeed, since $L$ has projective cover $\cV_r$ for some $1\leq r\leq p-1$, there is a map $f:\cV_r\rightarrow W$ such that the diagram
\begin{equation*}
\xymatrixcolsep{3pc}
\xymatrix{
& \cV_r \ar[ld]_{f} \ar@{->>}[d] \\
W \ar@{->>}[r] & L\\
}
\end{equation*}
commutes. The map $f$ is not injective because the maximal proper submodule $\cL_{2p-r}\subseteq\cV_r$ cannot be a composition factor of $K$, so $f$ descends to an injection $L\hookrightarrow W$ splitting the exact sequence.
\end{proof}

\begin{prop}\label{prop:weak_KL_lack_of_faithful}
In the setting of Theorem \ref{thm:weak_KL_correspondence}, $\cF(f)=0$ for a morphism $f$ in $KL^k(\fsl_2)$ if and only if $\mathrm{Im}\,f$ is an object of $KL_k(\fsl_2)$.
\end{prop}
\begin{proof}
Given $f: W_1\rightarrow W_2$ in $KL^k(\fsl_2)$, we factorize $f=\iota\circ\pi$ where $\pi: W_1\rightarrow\mathrm{Im}\,f$ is surjective and $\iota:\mathrm{Im}\,f\rightarrow W_2$ is injective. Because $\cF$ is exact by Theorem \ref{thm:weak_KL_correspondence}, $\cF(\pi)$ is also surjective and $\cF(\iota)$ is also injective. Thus $\cF(f)=0$ if and only if $\cF(\mathrm{Im}\,f)=0$, which by Lemma \ref{lem:weak_KL_lack_of_faithful_objects}  occurs if and only if $\mathrm{Im}\,f$ is an object of $KL_k(\fsl_2)$.
\end{proof}

We can now see that $\cF$ is also not full. For example, from the proof of Theorem \ref{thm:weak_KL_correspondence}, $\cF(\cV_1')\cong L_0\cong\cF(\cV_1)$. Thus $\cF$ maps $\hom_{KL^k(\fsl_2)}(\cV_1,\cV_1')$ to $\mathrm{End}_{\cC(\zeta,\fsl_2)}(L_0)=\CC\cdot\Id_{L_0}$. However, $\hom_{KL^k(\fsl_2)}(\cV_1,\cV_1')$ is spanned by the map $f$ obtained by composing the surjection $\cV_1\twoheadrightarrow\cL_1$ with the inclusion $\cL_1\rightarrow\cV_1'$. Since $\mathrm{Im}\,f\cong\cL_1$ is an object of $KL_k(\fsl_2)$, $\cF(f)=0$ by Proposition \ref{prop:weak_KL_lack_of_faithful}, and thus
\begin{equation*}
\cF: \hom_{KL^k(\fsl_2)}(\cV_1,\cV_1')\longrightarrow\mathrm{End}_{\cC(\zeta,\fsl_2)}(L_0)
\end{equation*}
is not surjective. But most of the maps between morphism spaces that $\cF$ induces  are surjective:
\begin{prop}\label{prop:weak_KL_full_condition}
In the setting of Theorem \ref{thm:weak_KL_correspondence},  let $W$ be an object of $KL^k(\fsl_2)$ such that there is a right exact sequence
\begin{equation*}
Q_W\xrightarrow{q_W} P_W\xrightarrow{p_W} W\longrightarrow 0
\end{equation*}
with $P_W$, $Q_W$ objects of $\cP^k$ such that $\cF(P_W)$ and $\cF(Q_W)$ are projective in $\cC(\zeta,\fsl_2)$. Then
\begin{equation*}
\cF: \hom_{KL^k(\fsl_2)}(W,X)\longrightarrow\hom_{\cC(\zeta,\fsl_2)}(\cF(W),\cF(X))
\end{equation*}
is surjective for any object $X$ in $KL^k(\fsl_2)$.
\end{prop}
\begin{proof}
First, since $KL^k(\fsl_2)$ has enough projectives, we may fix a right exact sequence
\begin{equation*}
Q_X\xrightarrow{q_X} P_X\xrightarrow{p_X} X\longrightarrow 0
\end{equation*}
where $P_X$ and $Q_X$ are objects of $\cP^k$. Then both $\cF(P_W)$ and $\cF(P_X)$ are objects of $\cT_\zeta$, with $\cF(P_W)$ in addition projective, and $\cF(p_W)$ and $\cF(p_X)$ are both surjective because $\cF$ is right exact. Now because $\cF(P_W)$ is projective in $\cC(\zeta,\fsl_2)$, for any morphism $\til{f}: \cF(W)\rightarrow\cF(X)$ in $\cC(\zeta,\fsl_2)$, there is a morphism $\til{g}: \cF(P_W)\rightarrow\cF(p_X)$ such that the diagram
\begin{equation*}
\xymatrixcolsep{3pc}
\xymatrix{
\cF(P_W) \ar@{->>}[r]^{\cF(p_W)} \ar[d]^{\til{g}} & \cF(W) \ar[d]^{\til{f}} \\
\cF(P_X) \ar@{->>}[r]^{\cF(p_X)} & \cF(X)\\
}
\end{equation*}
commutes. Since $\cF$ restricts to an equivalence between $\cP^k$ and $\cT_\zeta$, we have $\til{g}=\cF(g)$ for some map $g: P_W\rightarrow P_X$ in $KL^k(\fsl_2)$. Next, we can extend the commutative diagram to
\begin{equation*}
\xymatrixcolsep{3pc}
\xymatrix{
\cF(Q_W) \ar[r]^{\cF(q_W)} \ar[d]^{\cF(h)} & \cF(P_W) \ar@{->>}[r]^{\cF(p_W)} \ar[d]^{\cF(g)} & \cF(W) \ar[d]^{\til{f}} \\
\cF(Q_X) \ar[r]^{\cF(q_X)} & \cF(P_X) \ar@{->>}[r]^{\cF(p_X)} & \cF(X)\\
}
\end{equation*} 
Indeed, the projectivity of $\cF(Q_W)$ in $\cC(\zeta,\fsl_2)$ and the fact that 
\begin{equation*}
\mathrm{Im}\,\cF(g)\circ\cF(q_W)\subseteq\mathrm{Ker}\,\cF(p_X) =\mathrm{Im}\,\cF(q_X).
\end{equation*}
implies that there is a map $\til{h}: \cF(Q_W)\rightarrow\cF(Q_X)$ such that $\cF(q_X)\circ\til{h}=\cF(g)\circ\cF(q_W)$.
Then $\til{h}=\cF(h)$ for some $h: Q_W\rightarrow Q_X$ again because $\cF$ restricts to an equivalence between $\cP^k$ and $\cT_\zeta$.

Now because $\cF\vert_{\cP^k}$ is faithful, we have $q_X\circ h=g\circ q_W$. Thus by the universal property of cokernels, we get a unique map $f: W\rightarrow X$ such that the diagram
\begin{equation*}
\xymatrixcolsep{3pc}
\xymatrix{
Q_W \ar[r]^{q_W} \ar[d]^{h} & P_W \ar@{->>}[r]^{p_W} \ar[d]^{g} & W \ar[d]^f \\
Q_X \ar[r]^{q_X} & P_X \ar@{->>}[r]^{p_X} & X\\
}
\end{equation*} 
commutes. Then
\begin{equation*}
\cF(f)\circ\cF(p_W) =\cF(p_X)\circ\cF(g) = \til{f}\circ\cF(p_W),
\end{equation*}
so $\til{f}=\cF(f)$ by the surjectivity of $\cF(p_W)$. This proves the proposition.
\end{proof}

Although $KL^k(\fsl_2)$ and $\cC(\zeta,\fsl_2)$ are not equivalent when $k$ is an admissible level, it was pointed out to us by Cris Negron that there is a tensor equivalence when $KL^k(\fsl_2)$ and $\cC(\zeta,\fsl_2)$ are replaced with suitable derived categories. Before stating the result, we introduce some notation: For $\cC$ a full additive subcategory of some abelian category, $K^b(\cC)$ denotes the bounded homotopy category of $\cC$ \cite[Section 10.1]{We}; its objects are bounded cochain complexes of objects in $\cC$, and its morphisms are cochain homotopy equivalence classes of cochain maps.  If $\cC$ is itself abelian, $D^b(\cC)$ denotes the bounded derived category of $\cC$ \cite[Section 10.4]{We}; it is the localization of $K^b(\cC)$ at the collection of quasi-isomorphisms. We also use $D(\cC)$ for the unbounded derived category of the Ind-category $\ind\,\cC$ of $\cC$. In the case $\cC=KL^k(\fsl_2)$, $\ind\,\cC$ is the category of generalized $V^k(\fsl_2)$-modules which are the unions of their $KL^k(\fsl_2)$-submodules; by \cite{CMY-completions}, it has the vertex algebraic braided tensor category structure of \cite{HLZ1}-\cite{HLZ8}.

 The category $D(\cC)$ is one natural cocompletion of $D^b(\cC)$. Alternatively, one can take the Ind-category of $D^b(\cC)$, but since $D^b(\cC)$ is a triangulated category and not abelian, one first needs to replace $D^b(\cC)$ with its $\infty$-category version (see \cite[Section 1.3]{Lur2}) and then take the Ind-category in the sense of \cite[Definition 5.3.5.1]{Lur1}. One can then take the homotopy category (in the sense of \cite[Definition 1.1.3.2]{Lur1}) to get a triangulated category; this is what we mean by $\ind\,D^b(\cC)$. Now the following result and its proof were communicated to us by Cris Negron; it is mainly a consequence of Theorem \ref{thm:Tzeta_Pk_equiv}:
\begin{thm}\label{thm:derived_KL_corr}
Let $k=-2+p/q$ for relatively prime $p\in\ZZ_{\geq 2}$ and $q\in\ZZ_{\geq 1}$, and let $\zeta=e^{\pi i q/p}$. Then there is a monoidal equivalence $\cF: \ind\,D^b(\cC(\zeta,\fsl_2))\rightarrow D(KL^k(\fsl_2))$ extending the equivalence $K^b(\cT_\zeta)\xrightarrow{\sim} K^b(\cP^k)$ induced by Theorem \ref{thm:Tzeta_Pk_equiv}.
\end{thm}
\begin{proof}
The inclusion $\cT_\zeta\hookrightarrow\cC(\zeta,\fsl_2)$ induces a monoidal equivalence $K^b(\cT_\zeta)\rightarrow D^b(\cC(\zeta,\fsl_2))$ by \cite[Proposition 2.7]{Os}, which is based on \cite[Section 1.5]{BBM}. Since $KL^k(\fsl_2)$ has enough projectives, there is also a fully faithful embedding $K^b(\cP^k)\hookrightarrow D(KL^k(\fsl_2))$ that identifies $K^b(\cP^k)$ as a subcategory of compact objects in $D(KL^k(\fsl_2))$. Thus using also Theorem \ref{thm:Tzeta_Pk_equiv}, we get a monoidal embedding
\begin{equation*}
\cG: D^b(\cC(\zeta,\fsl_2)) \xrightarrow{\sim} K^b(\cT_\zeta)\xrightarrow{\sim} K^b(\cP^k)\hookrightarrow D(KL^k(\fsl_2)).
\end{equation*}
The $\infty$-category lift of $\cG$ is still fully faithful by \cite[Proposition 5.10]{BGT}. Now \cite[Proposition 5.3.5.10]{Lur1} yields a functor $\cF: \ind\,D^b(\cC(\zeta,\fsl_2))\rightarrow D(KL^k(\fsl_2))$. By \cite[Proposition 5.3.5.11]{Lur1}, $\cF$ is an equivalence because $\cG$ is fully faithful with compact image in $D(KL^k(\fsl_2))$, and this compact image generates $D(KL^k(\fsl_2))$ (which follows from \cite[Lemma 2.2.1]{SS}, for example). Taking homotopy categories (in the sense of \cite[Definition 1.1.3.2]{Lur1}) then gives a monoidal equivalence of triangulated categories (see \cite[Corollary 5.11]{BGT}).
\end{proof}

\begin{rem}
We call the monoidal equivalence of Theorem \ref{thm:derived_KL_corr} the \textit{derived Kazhdan-Lusztig correspondence} for $\fsl_2$ at admissible levels. 
\end{rem}

\subsection{A tensor-categorical version of quantum Drinfeld-Sokolov reduction}\label{subsec:Vir}

We now take $\cC$ in Theorem \ref{thm:KLk_univ_prop} to be the category $\cO_{c_{p,q}}$ of $C_1$-cofinite grading-restricted generalized modules for the universal Virasoro vertex operator algebra $V_{c_{p,q}}$ at central charge $c_{p,q}=1-\frac{6(p-q)^2}{pq}$ (note that $c_{p,q}=c_{q,p}$). The vertex operator algebra $V_{c_{p,q}}$ is the affine $W$-algebra obtained from $V^k(\fsl_2)$ by quantum Drinfeld-Sokolov reduction \cite{FF}; see also \cite[Chapter 15]{FB}. Moreover, quantum Drinfeld-Sokolov reduction defines an exact functor from suitable categories of $V^k(\fsl_2)$-modules to $V_{c_{p,q}}$-modules \cite{FKW, Ar}. We now use Theorem \ref{thm:KLk_univ_prop} to obtain a right exact tensor functor from $KL^k(\fsl_2)$ to $\cO_{c_{p,q}}$; we expect that this functor is exact and naturally isomorphic to quantum Drinfeld-Sokolov reduction, although we are not able to prove this at the moment. Thus for now, we refer to our functor as a ``tensor-categorical version'' of quantum Drinfeld-Sokolov reduction.

It was shown in \cite{CJORY} that the category $\cO_{c_{p,q}}$ has the vertex algebraic braided tensor category structure of \cite{HLZ1}-\cite{HLZ8}, and that its simple objects are the simple quotients of the Virasoro Verma modules $\cV_{r,s}$ of central charge $c_{p,q}$ and lowest conformal weight
\begin{equation*}
h_{r,s} =\frac{(qr-ps)^2-(p-q)^2}{4pq}
\end{equation*}
for $r,s\in\ZZ_{\geq 1}$. For $q=1$, the detailed structure of the tensor category $\cO_{c_{p,1}}$ was determined in \cite{MY2}, and these results were used in \cite{GN} to show that $\cO_{c_{p,1}}$ is tensor equivalent to the quantum group category $\cC(e^{\pi i/p},\fsl_2)$. Thus in this case, Theorem \ref{thm:weak_KL_correspondence} already provides a tensor-categorical version of quantum Drinfeld-Sokolov reduction which is exact.

In the case $p,q\geq 2$, some details of the tensor category structure on $\cO_{c_{p,q}}$ were obtained in \cite{MS}. In particular, Theorem 5.7 and Proposition 5.8 of \cite{MS} show that one of the two length-$2$ quotients of $\cV_{2,1}$, denoted $\mathcal{K}_{2,1}$, is rigid and self-dual with intrinsic dimension $-e^{\pi i q/p}-e^{-\pi i q/p}$, and that one of the two length-$2$ quotients of $\cV_{1,2}$, denoted $\mathcal{K}_{1,2}$, is rigid and self-dual with intrinsic dimension $-e^{\pi ip/q} - e^{-\pi i p/q}$. Then we have:
\begin{thm}\label{thm:tens_cat_qDS_red}
Let $p,q\in\ZZ_{\geq 2}$ be relatively prime. Then there are unique right exact braided tensor functors $\cF_{p,q}: KL^{-2+p/q}(\fsl_2)\rightarrow\cO_{c_{p,q}}$ and $\cF_{q,p}: KL^{-2+q/p}(\fsl_2)\rightarrow\cO_{c_{p,q}}$ such that $\cF_{p,q}(\cV_2, e_{\cV_2}, i_{\cV_2})=(\mathcal{K}_{2,1}, e_{\mathcal{K}_{2,1}}, i_{\mathcal{K}_{2,1}})$ and $\cF_{q,p}(\cV_2, e_{\cV_2}, i_{\cV_2})=(\mathcal{K}_{1,2}, e_{\mathcal{K}_{1,2}}, i_{\mathcal{K}_{1,2}})$. Moreover, $\cF_{p,q}$ and $\cF_{q,p}$ preserve twists if $KL^{-2+p/q}(\fsl_2)$ and $KL^{-2+q/p}(\fsl_2)$ are equipped with the (non-standard) minus sign twists of Theorem \ref{thm:KLk_braidings}(1).
\end{thm}
\begin{proof}
Since $\mathcal{K}_{1,2}$ corresponds to $\mathcal{K}_{2,1}$ under the identification $\cO_{c_{p,q}}=\cO_{c_{q,p}}$, it is enough to prove the statements about $\cF_{p,q}$. The existence and uniqueness of the right exact tensor functor $\cF_{p,q}$ are immediate from \cite[Proposition 5.8]{MS} and Theorem \ref{thm:KLk_univ_prop}. 

To show that $\cF_{p,q}$ is braided, we can determine the self-braiding $\cR_{\mathcal{K}_{2,1},\mathcal{K}_{2,1}}$ similar to Proposition \ref{prop:V12_self_braiding} and then apply Theorem \ref{thm:univ_prop_F_braided}. In more detail, the structure of $\mathcal{K}_{2,1}\boxtimes\mathcal{K}_{2,1}$ from \cite[Theorem 1.1(2)]{MS} implies that $\mathrm{End}_{\cO_{c_{p,q}}}(\mathcal{K}_{2,1}\boxtimes\mathcal{K}_{2,1})$ is spanned by $\Id_{\mathcal{K}_{2,1}\boxtimes\mathcal{K}_{2,1}}$ and $f_{\mathcal{K}_{2,1}}:=i_{\mathcal{K}_{2,1}}\circ e_{\mathcal{K}_{2,1}}$, so
\begin{equation*}
\cR_{\mathcal{K}_{2,1},\mathcal{K}_{2,1}} = a\cdot\Id_{\mathcal{K}_{2,1}\boxtimes\mathcal{K}_{2,1}}+ b\cdot f_{\mathcal{K}_{2,1}}
\end{equation*}
for some $a,b\in\CC$. Then exactly as in \cite[Lemma 6.1]{GN} and Proposition \ref{prop:V12_self_braiding}, there are only four possibilities for $(a,b)$ consistent with the hexagon axioms, and the correct possibility can be determined from the constraint
\begin{equation*}
e_{\mathcal{K}_{2,1}}\circ\cR_{\mathcal{K}_{2,1},\mathcal{K}_{2,1}} = e^{-2\pi i h_{2,1}}\cdot e_{\mathcal{K}_{2,1}} = -e^{-3\pi i q/2p}\cdot e_{\mathcal{K}_{2,1}},
\end{equation*}
which was obtained in the proof of \cite[Theorem 5.7]{MS}. Thus as in Proposition \ref{prop:V12_self_braiding}, the conclusion is that $a=b^{-1}=e^{\pi i q/2p}$, and it follows from Theorem \ref{thm:univ_prop_F_braided} that $\cF_{p,q}$ is braided.

The statement about twists follows from the fact that
\begin{equation*}
\theta_{\mathcal{K}_{2,1}} =e^{2\pi i L(0)} = e^{2\pi i h_{2,1}}\cdot\Id_{\mathcal{K}_{2,1}}= -e^{3\pi iq/2p}\cdot\Id_{\mathcal{K}_{2,1}},
\end{equation*}
combined with Theorem \ref{thm:univ_prop_F_braided}.
\end{proof}

\begin{conj}\label{conj:tens_cat_qDS_red}
The braided tensor functors $\cF_{p,q}$ and $\cF_{q,p}$ of Theorem \ref{thm:tens_cat_qDS_red} are exact and are naturally isomorphic to the restrictions of the quantum Drinfeld-Sokolov reduction functors of \cite{FKW} to $KL^{-2+p/q}(\fsl_2)$ and $KL^{-2+q/p}(\fsl_2)$.
\end{conj}

It is probably possible to show that $\cF_{p,q}$ and $\cF_{q,p}$ are exact by verifying the conditions of Theorem \ref{thm:right_exact_to_exact}, as in the proof of Theorem \ref{thm:weak_KL_correspondence}. However, it would be necessary to determine the images of all projective modules in $KL^{-2+p/q}(\fsl_2)$ and $KL^{-2+q/p}(\fsl_2)$ under $\cF_{p,q}$ and $\cF_{q,p}$, and these would mostly be logarithmic $V_{c_{p,q}}$-modules which were not all constructed in \cite{MS}. It would probably also be necessary to show that $\cF_{p,q}\vert_{\cP^{-2+p/q}}$ and $\cF_{q,p}\vert_{\cP^{-2+q/p}}$ are fully faithful, similar to the proof of Theorem \ref{thm:Tzeta_Pk_equiv}.

\end{document}